\documentclass[reqno,12pt]{amsart}

\usepackage{amsmath,amssymb,amscd}
\usepackage{enumerate}
\usepackage{graphicx}
\usepackage{epic}
\usepackage[all,cmtip]{xy}
\usepackage{overpic}
\usepackage{subfig}
\usepackage{caption}
\usepackage{varioref}
\usepackage{tikz}

\newtheorem{thm}{Theorem}[section]
\newtheorem{prop}[thm]{Proposition}
\newtheorem{lem}[thm]{Lemma}
\newtheorem{cor}[thm]{Corollary}

\theoremstyle{definition}
\newtheorem{Def}[thm]{Definition}
\newtheorem{conj}[thm]{Conjecture}
\newtheorem{exmp}[thm]{Example}
\newtheorem*{que}{Question}
\newtheorem{prob}[thm]{Problem}
\newtheorem{rem}[thm]{Remark}
\newtheorem*{nota}{Notation}
\newtheorem{const}[thm]{Construction}
\newtheorem*{conv}{Notational Convention}

\newcommand{\kk}{\mathbf{k}}
\newcommand{\xx}{\mathbf{x}}
\newcommand{\uu}{\mathbf{u}}

\newcommand{\Aa}{\mathcal {A}}
\newcommand{\CC}{\mathcal {C}}
\newcommand{\ZZ}{\mathcal {Z}}

\newcommand{\JJ}{\mathcal {J}}
\newcommand{\OO}{\mathcal {O}}
\newcommand{\PP}{\mathcal {P}}
\newcommand{\QQ}{\mathcal {Q}}

\newcommand{\TT}{\mathcal {T}}
\newcommand{\Ss}{\mathcal {S}}

\newcommand{\VV}{\mathcal {V}}

\newcommand{\MM}{\mathcal {M}}
\newcommand{\FF}{\mathcal {F}}
\newcommand{\II}{\mathcal {I}}
\newcommand{\EE}{\mathcal {E}}
\newcommand{\Zz}{\mathbb {Z}}
\newcommand{\Cc}{\mathbb {C}}
\newcommand{\Rr}{\mathbb {R}}
\newcommand{\Qq}{\mathbb {Q}}

\def\w{\widetilde}
\def\wh{\widehat}
\def\ha{\hookrightarrow}
\def\xr{\xrightarrow}

\numberwithin{equation}{section}
\usepackage[colorlinks,linktocpage]{hyperref}
\hypersetup{citecolor=blue,linkcolor=blue}

\begin{document}
\title[Cohomological rigidity of manifolds]{Cohomological rigidity of manifolds with torus actions: I}

\author[F.~Fan, J.~Ma \& X.~Wang]{Feifei Fan, Jun Ma and Xiangjun Wang}

\thanks{The authors are supported by National Natural Science Foundation of China (Grant no. 12271183). 
The first named author was supported by National Natural Science Foundation of China (Grant no. 11801580) and by Chinese Postdoctoral Science Foundation (Grant no. 2021M691098). He is also supported by GuangDong Basic and Applied Basic Research Foundation (Grant no. 2023A1515012217).
The third named author was supported in part by National Natural Science Foundation of China (Grant nos. 11871284 and 11761072).}

\address{Feifei Fan, School of Mathematical Sciences, South China Normal University, Guangzhou, 510631, China.}
\email{fanfeifei@mail.nankai.edu.cn}
\address{Jun Ma, College of Mathematics, Taiyuan University of Technology, Taiyuan 030024, China.}
\email{majun01@tyut.edu.cn}
\address{Xiangjun Wang, School of Mathematical Sciences and LPMC, Nankai University, Tianjin 300071, China.}
\email{xjwang@nankai.edu.cn}


\maketitle

\begin{abstract}
We study the cohomological rigidity problem of two families of manifolds with torus actions: the so-called moment-angle manifolds, whose study
is linked with combinatorial geometry and combinatorial commutative algebra; and topological toric manifolds, which are topological generalizations of toric varieties. 
In this paper we prove that when a simplicial sphere satisfies certain combinatorial conditions, the corresponding moment-angle manifold and topological toric manifolds are cohomologically rigid, i.e., their homeomorphism classes in their own families are determined by their cohomology rings. In the case of toric varieties, cohomology even determine the isomorphism classes of varieties.
Our main strategy is to show that the combinatorial types of these simplicial spheres are determined by the $\mathrm{Tor}$-algebras of their face rings.
This turns out to be a solution to a known problem in combinatorial commutative algebra for a class of spheres.
\end{abstract}

\tableofcontents

\section{Introduction}
\renewcommand{\thethm}{\arabic{thm}}
A central problem in topology is to classify spaces by equivalent relations such as a homeomorphism or a homotopy equivalence. The \emph{cohomological rigidity problem} for manifolds asks whether two manifolds are homeomorphic (or even diffeomorphic) or not if their cohomology rings are isomorphic as graded rings. Although answering this question is very difficult even for the simplest case: spheres, and the answer is no in the general case, if we put further conditions on one or both manifolds we can exploit this additional structure in order to show that the desired homeomorphism or diffeomorphism must exist.

In this paper, we study the cohomological rigidity problem for two families of manifolds with torus actions: moment-angle manifolds and topological toric manifolds.
They play key roles in the emerging field of toric topology, which has rich connections with algebraic and symplectic geometry, commutative algebra and combinatorics.
\cite{MS08} and \cite{CMS11} are good references for related results and problems about their cohomological rigidity.

Associated with every finite simplicial complex $K$, there is a CW-complex $\ZZ_K$, called a
\emph{moment-angle complex}, whose topology is determined by the combinatorial type of $K$ (see Section \ref{subsec:m-a complex}); if $K$ is a triangulation of a sphere or more generally a homology sphere,
$\ZZ_K$ is a compact manifold, called a \emph{moment-angle manifold}. In particular, if $K$ is a polytopal sphere, or more generally the underlying sphere
of a complete simplicial fan, then $\ZZ_K$ admits a smooth structure.

Suppose $K$ has $m$ vertices and $\kk[m]=\kk[x_1,\dots,x_m]$ ($\deg x_i=2$) is the polynomial algebra over a commutative ring $\kk$.
Let $\kk[K]$ be the \emph{face ring} of $K$.
Then the cohomology ring of the moment-angle complex $\ZZ_K$ can be calculated as follows.
\begin{thm}[{{\cite[Proposition 4.5.5]{BP15}}}]
The following isomorphism of algebras holds:
\[
H^*(\ZZ_K;\kk)\cong\mathrm{Tor}^{*,*}_{\kk[m]}(\kk[K],\kk).
\]
\end{thm}

The cohomology of $\ZZ_K$ therefore acquires a bigrading, and we may ask the following question about the cohomological rigidity for moment-angle manifolds.
\begin{prob}\label{prob:rigidity for m-a}
Suppose $\mathcal {Z}_{K_1}$ and $\mathcal {Z}_{K_2}$ are two moment-angle manifolds such that
\[H^*(\mathcal {Z}_{K_1})\cong H^*(\mathcal {Z}_{K_2})\] as graded (or bigraded) rings.
Are $\mathcal {Z}_{K_1}$ and $\mathcal {Z}_{K_1}$ homeomorphic?
\end{prob}
Although the cohomology ring is known to be a weak invariant even under homotopy equivalence, no example
providing the negative answer to this problem, in both graded and bigraded versions, has been found yet. On the contrary, many results support the affirmative answer to the problem (see for example \cite{CK11,Bos14,FMW15,CFW20}).

Another important class of spaces in toric topology is formed by the so-called \emph{topological toric manifolds},
which can be seen as the topological generalization of compact smooth toric varieties. For topological toric manifolds, we can also ask the same question:
\begin{prob}\label{prob:rigidity for tpm}
Let $M_1$ and $M_2$ be two topological toric manifolds with isomorphic cohomology rings. Are they homeomorphic?
\end{prob}
The problem is solved positively for some particular families of topological toric manifolds, such as some special classes of Bott manifolds \cite{CMS10,CM12,Cho15},
quasitoric manifolds over a product of two simplices \cite{CPS12}, and $6$-dimensional quasitoric manifolds over polytopes from the Pogorelov class \cite{BEMPP17}.
The cohomological rigidity problem is also open for general topological toric manifolds.

This paper is divided into two parts. In the first part of this paper we focus our attention on moment-angle manifolds corresponding to a special class of homology spheres and topological toric manifolds related to such sort of simplicial fans.
This class consists of homology spheres which are flag and satisfy the \emph{separable circuit condition} (see Definition \ref{def:SCC}).
The main results of this paper are the following theorems.
\begin{thm}[Theorem \ref{thm:B-rigidity}, Corollary \ref{cor:rigidity of m-a manifold}]\label{thm:B}
Let $\ZZ_K$ and $\ZZ_{K'}$ be two moment-angle manifolds. Suppose the cohomology of $\ZZ_K$ and $\ZZ_{K'}$ are isomorphic as bigraded rings.
Then if $K$ is flag and satisfies the separable circuit condition, there is a homeomorphism $\ZZ_K\cong \ZZ_{K'}$, which is induced by a combinatorial equivalence $K\approx K'$.
\end{thm}

\begin{thm}[Theorem \ref{thm:rigidity of toric manifold}]\label{thm:C}
Let $M$ and $M'$ be two topological toric manifolds.
If the underlying sphere of the simplicial fan corresponding to $M$ is flag and satisfies the separable circuit condition, then the following two conditions are equivalent:
\begin{enumerate}[(i)]
\item $H^*(M)\cong H^*(M')$ as graded rings.
\item $M$ and $M'$ are weakly equivariantly homeomorphic.
\end{enumerate}
\end{thm}

Theorem \ref{thm:B} and Theorem \ref{thm:C} are generalizations of cohomological rigity results in \cite{FMW15} and \cite{BEMPP17}, respectively, for simplicial spheres of dimension $2$ to all dimensions.

By using a result of Masuda \cite{M08}, we can get a stronger version of Theorem \ref{thm:C} for toric varieties.
\begin{thm}[Theorem \ref{thm:variety}]
Let $X$ and $X'$ be two compact smooth toric varieties. Assume the underlying simplicial sphere of the simplcial fan corresponding to $X$ is flag and satisfies the separable circuit condition. Then  $X$ and $X'$ are isomorphic as varieties if and only if their cohomology rings are isomorphic.
\end{thm}

In Section \ref{sec:n=3} we further discuss the special case of simplicial $2$-spheres.
This section can be regarded as an introduction to the second part  \cite{FMW16} of this paper.
\subsection*{Acknowledgements} We thank the Institute for Mathematical Science, National University of Singapore,
and the organizers of the Program on Combinatorial and Toric Homotopy in summer 2015, where the first and the second authors got financial support, and special thanks to Victor Buchstaber and Nikolay Erokhovets for delightful conversations and insightful comments there.

Thanks to Suyoung Choi for enlightening conversations about the cohomological rigidity problem of quasitoric manifolds at Fudan University. Thanks also to the organizers of the Fudan Topology Workshop in January 2016 for giving us an opportunity to discuss.

We also thank the Departments of Mathematics, Princeton University and Rider University, and the organizers of the Princeton-Rider Workshop on the Homotopy Theory of Polyhedral Products in May-June 2017, for giving an opportunity to the first author to present our work and for providing the first and the second authors financial support, and special thanks to Fr\'ed\'eric Bosio for helpful comments and for introducing us to the important result of \cite{B17} there, which is closely related to the {(bigraded) $B$-rigidity} result of our work.

Many thanks to Nikolay Erokhovets for providing many important references and remarks, as well as proofreading early versions for misprints.

\section{Preliminaries}
\renewcommand{\thethm}{\thesection.\arabic{thm}}
\subsection{Notations and conventions}Throughout this paper $\kk$ is a commutative ring with unit.
Let $K$ be an abstract simplicial complex. We denote the geometric realisation of $K$ by $|K|$ and the vertex set of $K$ by $\VV(K)$; if the cardinality $|\VV(K)|=m$, we often identify $\VV(K)$ with the index set $[m]=\{1,\dots,m\}$.
By $\Delta^{m-1}$ we denote the simplicial complex $2^{[m]}$ consisting of all subsets of $[m]$, and by $\partial\Delta^{m-1}$ the boundary complex of $\Delta^{m-1}$.

A subset $I\subset\VV(K)$ is a \emph{missing face} of $K$ if $I\not\in K$ but $J\in K$ for all proper subsets $J\subset I$.
Denote by $MF(K)$ the set of all missing faces of $K$. $K$ is called a \emph{flag complex} if each of its missing faces consists of two vertices.

The \emph{link} and the \emph{star} of a face $\sigma\in K$ are the subcomplexes
 \[\begin{split}
 \mathrm{lk}_K\sigma:=&\{\tau\in K:\tau\cup\sigma\in K,\tau\cap\sigma=\emptyset\};\\
 \mathrm{st}_K\sigma:=&\{\tau\in K:\tau\cup\sigma\in K\}.
 \end{split}\]
For a subset $I\subset\VV(K)$, the \emph{full subcomplex} $K_I\subset K$ is defined to be
\[K_I:=\{\sigma\in K:\sigma\subset I\}.\]
Set $\mathrm{core}\,\VV(K):=\{i\in\VV(K):\mathrm{st}_K\{i\}\neq K\}$. The \emph{core of $K$} is the subcomplex $\mathrm{core}\,K:=K_{\mathrm{core}\,\VV(K)}$.

 The \emph{join} of two simplicial complexes $K$ and $K'$, where the vertex set $\VV(K)$ is disjoint from $\VV(K')$, is the simplicial complex
\[K*K':=\{\sigma\cup \sigma':\sigma\in K, \sigma'\in K'\}.\]
In particular, $\Delta^0*K$ is called the \emph{cone} over $K$, and  $\partial\Delta^1*K$ is called the \emph{suspension} of $K$.

A simplicial complex is said to be \emph{pure} if all its facets have the same dimension.
Let $K$, $K'$ be two pure simplicial complexes of the same dimension.
Choose two facets $\sigma\in K,\ \sigma'\in K'$ and fix an identification of $\sigma$ and $\sigma'$ (by identifying their vertices).
Then the simplicial complex \[K\#_{\sigma,\sigma'}K':=(K\cup K')\setminus\{\sigma=\sigma'\}\]
is called a \emph{connected sum} of $K$ and $K'$.
Its combinatorial type depends on the way of choosing the two facets and identifying their vertices.
Let $\mathcal {C}(K\#K')$ denote the set of connected sums of $K$ and $K'$.
If $\mathcal {C}(K\#K')$ has only one combinatorial type we also use the abbreviation $K\# K'$ to denote a connected sum.

A simplicial complex $K$ is called a \emph{triangulated manifold} (or \emph{simplicial manifold}) if the geometric realization $|K|$ is a topological manifold.
More generally, $K$ is a \emph{$\kk$-homology $n$-manifold} if the
link of each $i$-face of $K$, $0\leqslant i\leqslant n$, has the $\kk$-homology of an $(n-i-1)$-sphere.
$K$ is a \emph{$\kk$-homology $n$-sphere} if it is a $\kk$-homology $n$-manifold having the $\kk$-homology of an $n$-sphere, and if $\kk=\Zz$, it is simply called a \emph{homology sphere}. (Remark: Usually, the terminology ``homology sphere"  means a manifold having the homology of a sphere. Here we take it in a more relaxed sense than its usual meaning.) {By universal coefficient theorem}, any $\kk$-homology manifold or sphere is a rational homology manifold or sphere.

A $\kk$-homology sphere is called \emph{reducible} if it is a connected sum of two $\kk$-homology spheres; otherwise it is called \emph{irreducible}. It is easy to see that an $(n-1)$-dimensional $\kk$-homology sphere $K$ is reducible if and only if there is a missing face $I\in MF(K)$ such that $|I|=n$.
	
A \emph{fan} is a finite collection $\Sigma$ of strongly convex cones in $\mathbb{R}^n$, which means that the origin is the apex, such that every face
of a cone in $\Sigma$ is also a cone in $\Sigma$ and the intersection of any two cones in $\Sigma$ is a face of each.
A fan $\Sigma$ is \emph{simplicial} (resp. \emph{nonsingular}) if each of its cones is generated by a part
of a basis of $\mathbb{R}^n$ (resp. $\mathbb{Z}^n$).
A fan $\Sigma$ is \emph{complete} if the union of its cones is the whole $\mathbb{R}^n$.

A \emph{starshaped sphere} is a triangulated sphere isomorphic to the underlying complex of a complete simplicial fan.
A \emph{polytopal sphere} is a triangulated sphere isomorphic to the boundary complex of a simplicial polytope.
Clearly a polytopal sphere is always a starshaped sphere.

\subsection{Face rings and $\mathrm{Tor}$-algebras}\label{subsec:Tor}
Let $K$ be a simplicial complex with vertex  set $[m]$. The \emph{face ring} (or \emph{Stanley-Reisner ring}) of K is defined as the quotient of the polynomial ring $\kk[x_1,\dots\,x_m]$ by the square-free monomial ideal
generated by non-faces of K:
\[\kk[{K}]:=\kk[x_1,\dots\,x_m]/(x_{i_1}\cdots x_{i_s}:\{i_1,\dots,i_s\}\not\in K).\]
$\kk[K]$ can be graded or multigraded by setting $\deg x_i=2$ or $\mathrm{mdeg}\,x_i=2\boldsymbol{e}_i$,
where $\boldsymbol{e}_i\in\Zz^m$ is the $i$th unit vector.

The \emph{Koszul complex} (or the \emph{Koszul algebra}) of the face ring $\kk[K]$ is defined as
the differential $\Zz\oplus\Zz^m$-graded algebra $(\Lambda[y_1,\dots,y_m]\otimes\kk[K],d)$,
where $\Lambda[y_1,\dots,y_m]$ is the exterior algebra on $m$ generators over $\kk$, and the multigrading and differential are given by
\begin{gather*}
\mathrm{mdeg}\,y_i=(-1,2\boldsymbol{e}_i),\ \mathrm{mdeg}\,x_i=(0,2\boldsymbol{e}_i);\\
dy_i=x_i,\quad dx_i=0.
\end{gather*}
It is known that \[H^*(\Lambda[y_1,\dots,y_m]\otimes\kk[K],d)=\mathrm{Tor}_{\kk[x_1,\dots,x_m]}(\kk[K],\kk).\]
Then the Tor-algebra $\mathrm{Tor}_{\kk[x_1,\dots,x_m]}(\kk[K],\kk)$ is canonically an $\Zz\oplus\Zz^m$-graded algebra.
It can also be seen as a bigraded algebra by setting $\mathrm{bideg}\,y_i=(-1,2)$ and $\mathrm{bideg}\,x_i=(0,2)$.
We refer to $\mathrm{Tor}_{\kk[x_1,\dots,x_m]}(\kk[K],\kk)$ as the \emph{$\mathrm{Tor}$-algebra} of $K$ and simply denote it by $\mathrm{Tor}(K;\kk)$, or $\mathrm{Tor}(K)$ if the coefficient ring is clear.

Another way to calculate $\mathrm{Tor}(K)$ is by using the Taylor resolution for $\kk[K]$.
This method was introduced  by Yuzvinsky \cite{Y99}, and used by Wang and Zheng \cite{WZ13} to calculate the cohomology rings of moment-angle complexes.

Namely, let
$\mathbb{P}=MF(K)$, and let $\Lambda[\mathbb{P}]$ be the exterior algebra over $\kk$ generated by $\mathbb{P}$.
Then the tensor product $\TT=\Lambda[\mathbb{P}]\otimes\kk[x_1,\dots,x_m]$ can be endowed with a differential as follows.
For a generator $\uu\otimes f\in \Lambda[\mathbb{P}]\otimes\kk[x_1,\dots,x_m]$ with $\uu=\sigma_{k_1}\sigma_{k_2}\cdots{\sigma_{k_q}}$ (${\sigma_{k_i}}\in\mathbb{P}$),
let $S_{\uu}=\sigma_{k_1}\cup\sigma_{k_2}\cup\cdots\cup\sigma_{k_q}$, and define
\begin{align*}
\partial_i\uu:=\sigma_{k_1}\cdots&\wh \sigma_{k_i}\cdots\sigma_{k_q}=\sigma_{k_1}\cdots\sigma_{k_{i-1}}\sigma_{k_{i+1}}\cdots\sigma_{k_q};\\
d(\uu\otimes f)&:={\sum_{i=1}^q(-1)^{i-1}}\partial_i\uu\otimes \xx_{S_\uu\setminus S_{\partial_i\uu}}\cdot f.
\end{align*}
Here $\xx_I\in\kk[x_1,\dots,x_m]$ denotes the monomial $\prod_{i\in I}x_i$ for $I\subset[m]$.
It can be shown that $(\TT, d)$ is a $\kk[x_1,\dots,x_m]$-free resolution of $\kk[K]$, which is known as the \emph{Taylor resolution} (cf. \cite[\S 4.3.2]{MS05}). It follows that $H(\TT\otimes_{\kk[x_1,\dots,x_m]}\kk,d)=\mathrm{Tor}(K)$.

An easy calculation shows that $(\TT\otimes_{\kk[x_1,\dots,x_m]}\kk, d)$ is just the cochain complex
\begin{equation}\label{eq:differential}
(\Lambda[\mathbb{P}],d),\quad d(\uu)=\sum_{i}(-1)^{i-1}\varepsilon_i\cdot\partial_i\uu,
\end{equation}
where $\varepsilon_i=1$ if $S_{\uu}=S_{\partial_i(\uu)}$ and zero otherwise.
Define a new product ($\times$-product) structure on $\Lambda[\mathbb{P}]$ by
\[\uu_1\times \uu_2=
\begin{cases}
\uu_1\cdot \uu_2 \quad &\text{if } S_{\uu_1}\cap S_{\uu_2}=\emptyset,\\
0 &\text{otherwise,}
\end{cases}\]
where {\LARGE$\cdot$} denotes the ordinary product in the exterior algebra $\Lambda[\mathbb{P}]$. Then $(\Lambda[\mathbb{P}],\times,d)$ becomes a differential $\Zz\oplus\Zz^m$-graded (bigraded) algebra by setting
\[\mathrm{mdeg}\,\uu=(-1,2S_\uu)\ (\mathrm{bideg}\,\uu=(-1,2|S_\uu|))\ \text{for a monomial }\uu\in \Lambda[\mathbb{P}].\]
Here we identify $S_\uu$ with the vector $\sum_{i\in S_\uu}\boldsymbol{e}_i\in\Zz^m$.
This makes $\mathrm{Tor}(K)$ into a $\Zz\oplus\Zz^m$-graded (or bigraded) algebra which agrees with the one induced by the Koszul algebra of the face ring. From \eqref{eq:differential} we can readily get the following
\begin{lem}\label{lem:missing face and tor}
Let $\kk$ be a field. Then the number of missing faces of $K$ with $j$ vertices is equal to $\dim_\kk\mathrm{Tor}^{-1,\,2j}(K)$.
\end{lem}

\subsection{Gorenstein complexes}
In the subject of combinatorial commutative algebra, Gorenstein complexes play an especially significant role because of their nice algebraic-combinatorial properties. We review here the most important aspects
of Gorenstein complexes.

A simplicial complex $K$ of dimension $n-1$ is called a \emph{Cohen-Macaulay} complex over a field $\kk$, if its face ring $\kk[K]$ is a Cohen-Macaulay ring. Namely, there is a regular sequence of length $n$ in $\kk[K]$. 
A Cohen-Macaulay complex $K$ of dimension $n-1$ with $m$ vertices is called \emph{Gorenstein} over a field $\kk$ if $\mathrm{Tor}^{-(m-n)}(K)=\kk$.
{We say that $K$ is a \emph{Gorenstein complex over $\mathbb{Z}$}, or simply a \emph{Gorenstein complex} if $K$ is Gorenstein over any field}. Further, $K$ is called \emph{Gorenstein*} if $K$ is Gorenstein and $\mathrm{core}\,K=K$.

A fundamental result of Avramov and Golod says that the $\mathrm{Tor}$-algebra of a Gorenstein complex {over a field} satisfies Poincar\'e duality.
\begin{thm}[Avramov-Golod \cite{AG71}, {\cite[Theorem 3.4.4]{BP15}}]\label{thm:Gorenstein Tor-alg}
	A simplicial complex $K$ is Gorenstein over a field $\kk$
	if and only if  $\mathrm{Tor}(K)$ is a Poincar\'e duality $\kk$-algebra with respect to the homology degree, i.e., there exists an integer $d\geqslant0$ such that the $\kk$-homomorphisms 
	\[
	\mathrm{Tor}^{-i}(K)\to \mathrm{Hom}_\kk(\mathrm{Tor}^{i-d}(K), \mathrm{Tor}^{-d}(K)),\quad a\mapsto\phi_a \text{ with } \phi_a(b)=ab,
	\]
	are isomorphisms for $0\leqslant i\leqslant d$.
\end{thm}
 {
{Gorenstein* complexes also have combinatorial-topological characterizations.}
{\begin{thm}[{\cite[Theorem II.5.1]{S96}} and {\cite[Lemma 3.3]{FW21}} or {\cite[\S 59.1]{Rin75}}]\label{thm:Gorenstein}
Suppose $\kk$ is a field or $\Zz$. Let $K$ be a simplicial complex with vertex set $[m]$. Then the following assersions are equivalent
\begin{enumerate}[(a)]
	\item  $K$ is Gorenstein* over $\kk$.
	\item $K$ is a $\kk$-homology sphere.
	\item $\w H^{i}(K_J;\kk)\cong\w H_{n-i-1}(K_{[m]\setminus J};\kk)$ for all $J\subset[m]$ and $-1\leqslant i\leqslant n$, where $n=\dim K$. Here we assume $\w H^{-1}(K_\emptyset)=\kk$.
\end{enumerate}
\end{thm}}

\subsection{Moment-angle complexes and manifolds}\label{subsec:m-a complex}
Moment-angle complexes play a key role in toric topology. They were first introduced by Davis and Januszkiewicz \cite{DJ91} and extensively studied in
detail and named by Buchstaber and Panov \cite{BP00}, {where a general construction from an arbitrary topological pair $(X,A)$ was also considered (following a suggestion of N. Strickland). Further generalizations of moment-angle complexes to a set
of pairs of spaces, now known as polyhedral products, were studied in the work of Grbi\'c and Theriault \cite{GT07} as well as Bahri, Bendersky, Cohen and Gitler \cite{BBCG1}, where the term `polyhedral product' (due to W. Browder) was introduced.} Let us recall their construction in the most general way.

Given a collection of CW pairs $(\underline{X},\underline{A})=\{(X_i,A_i)\}_{i=1}^m$ together with a simplicial complex $K$ with vertex set $[m]$. The \emph{polyhedral product} determined by $(\underline{X},\underline{A})$ and $K$ is defined to be the CW complex:
\[\mathcal {Z}_K(\underline{X},\underline{A})=\bigcup_{\sigma\in K}B_\sigma,\]
where \[B_\sigma=\prod_{i=1}^m Y_i\quad \text{and}\quad Y_i=\begin{cases}X_i\quad \text{if}\quad i\in\sigma,\\
A_i\quad \text{if}\quad i\not\in\sigma.\end{cases}\]

If $(\underline{X},\underline{A})=(D^2, S^1)$, i.e. the pairs $(X_i,A_i)$ are identically $(D^2, S^1)$,  then the space $\mathcal {Z}_K(D^2,S^1)$ is referred to as a
\emph{moment-angle complex} and simply denoted by $\mathcal {Z}_K$; in the case $(\underline{X},\underline{A})=(D^1, S^0)$, the space $\Rr\ZZ_K=\ZZ_K(D^1,S^0)$ is referred to as a \emph{real moment-angle complex}.

Note that $\ZZ_K\subset (D^2)^m$, the standard coordinatewise action of the $m$-torus $T^m=\mathbb{R}^m/\mathbb{Z}^m$ on $(D^2)^m$ induces
a \emph{canonical} $T^m$-action on $\ZZ_K$.

A result of Cai \cite[Corollary 2.10]{C17} says that $\ZZ_K$ is a closed orientable topological manifold of dimension $m+n$ if and only if $K$ is a homology $(n-1)$-sphere. In this case,  we call $\ZZ_K$ a \emph{moment-angle manifold}.
In particular, if $K$ is a polytopal sphere (see \cite[Corollary 6.2.5]{BP15}) or a starshaped sphere (\cite{PU12}), then there exists a smooth structure on $\ZZ_K$ compatible with its canonical torus action. In general, the question whether $\ZZ_K$ has a smooth tructure is open.

\subsection{Cohomology of moment-angle complexes}\label{subsec:cohomology of m-a}
The cohomology ring of $\ZZ_K$ was studied rationally by Buchstaber and Panov
in \cite{BP00} {and integrally by Franz in \cite{M03,M06} as well as others} \cite{B02,BBP04,P08}.
Throughout this subsection, the coefficient ring $\kk$ will be omitted from the notation.

\begin{thm}[{\cite[Theorem 4.5.4]{BP15}}]\label{thm:cohomology of m-a}
The following isomorphism of algebras holds:
\[
H^*(\ZZ_K)\cong\mathrm{Tor}(K),\quad H^p(\ZZ_K)\cong\bigoplus_{-i+2|J|=p}\mathrm{Tor}^{-i,2J}(K),
\]
where $J=(j_1,\dots,j_m)\in\Zz^m$ and $|J|=j_1+\cdots+j_m$.
\end{thm}
Viewing a subset $J\subset[m]$ as a $(0,1)$-vector in $\Zz^m$ whose $j$th coordinate
is $1$ if $j\in J$ and $0$ otherwise, there is the following well-known
Hochster's formula. (This formula was obtained by Hochster \cite{H75} for field coefficients, and generalized to arbitrary coefficients by Panov \cite{P08}.)
\begin{thm}[{\cite[Theorem 3.2.9]{BP15}}]\label{thm:Hochster}
$\mathrm{Tor}^{-i,2J}(K)=0$ if $J$ is not a $(0,1)$-vector, and for any subset $J\subset[m]$ we have
\[\mathrm{Tor}^{-i,2J}(K)\cong\w H^{|J|-i-1}(K_J).\]
Here we assume $\w H^{-1}(K_\emptyset)=\kk$.
\end{thm}

So $\mathrm{Tor}(K)$ is isomorphic to $\bigoplus_{J\subset[m]}\w H^*(K_J)$  as $\kk$-modules,
and this isomorphism endows the direct sum $\bigoplus_{J\subset[m]}\w H^*(K_J)$ with a $\kk$-algebra structure.

On the other hand, Baskakov \cite{B02} directly defined a natural multiplication structure on $\bigoplus_{J\subset[m]}\w H^*(K_J)$ as follows. For an oriented $(p-1)$-face $\sigma\in K$, whose orientation is induced by an ordering of its vertices, denote by $\sigma^*\in\widetilde C^{p-1}(K)$ the basis cochain dual to $\sigma$; it takes value $1$ on $\sigma$ and vanishes on all other faces. Consider the product in the simplicial cochains of
	full subcomplexes given by 
\begin{equation}\label{eq:BHR}
	\begin{split}
		\w C^{p-1}(K_I)\otimes \w C^{q-1}(K_J)&\to \w C^{p+q-1}(K_{I\cup J}),\\
		\sigma^*\otimes\tau^*&\mapsto \begin{cases}
			(\sigma\cup \tau)^*\ \text{ if } I\cap J=\emptyset \text{ and } \sigma\cup\tau\in K_{I\cup J};\\
			0\ \text{ otherwise. }
		\end{cases}
	\end{split}
\end{equation}
Here $\sigma\cup \tau$ means the juxtaposition of $\sigma$ and $\tau$. We refer to $\bigoplus_{J\subset[m]}\w H^*(K_J)$ with the product given by \eqref{eq:BHR} as the \emph{Baskakov-Hochster ring} (BHR for abbreviation) of $K$, and refer to the first line and the second line of \eqref{eq:BHR} as the \emph{degree rule} and \emph{intersection rule} of BHR, respectively.
Moreover, we use $p_J$ to denote the projection map
\begin{equation}\label{eq:projection}
p_J:\bigoplus_{J\subset[m]}\w H^*(K_J)\to \w H^*(K_J).
\end{equation}

\begin{thm}[Baskakov's formula, {\cite[Theorem 4.5.8]{BP15}}]\label{thm:union product}
There is a ring isomorphism (up to sign for products).
\begin{equation}\label{eq:Hochster ring}
H^*(\ZZ_K)\cong \bigoplus_{J\subset[m]}\w H^*(K_J),\quad H^p(\ZZ_K)\cong\bigoplus_{J\subset[m]}\w H^{p-|J|-1}(K_J),
\end{equation}
where the right side of the isomorphism is the BHR of $K$.
\end{thm}
\begin{rem}\label{rem:sign}
The sign for products was omitted in Baskakov's original paper \cite{B02}. This was pointed out by Bosio and Meersseman \cite[Remark 10.7]{BM06}, as well as Buchstaber and Panov \cite[Proposition 3.2.10]{BP15}. Bosio and Meersseman \cite[Theorem 10.1]{BM06} gave the explicit formula of this sign in the sence of Alexander duality, i.e., in the form of the intersection product of the homology of the full subcomplexes. See \cite[Theorem 2.15]{FW15} for the formula of this sign in the cohomology form.
\end{rem}

The following lemma is useful for later proofs of our results, and we include its proof in Appendix \ref{appdx:BHR2} for the reader's convenience.
\begin{lem}\label{lem:hoch}
Given $\alpha\in \w {H}^i(K_I)$, $\beta\in \w {H}^j(K_J)$ with $I\cap J=\emptyset$. Suppose $\alpha=\phi^*(\alpha')$ for some $\alpha\in \w {H}^i(K_{I'})$, $I\subset I'$ and $\phi:K_I\ha K_{I'}$. Let $J'=(I\cup J)\setminus I'$, $\psi:K_{J'}\ha K_J$ and $\beta'=\psi^*(\beta)$.
Then in the BHR of $K$, $\alpha\beta=\alpha'\beta'$.
\end{lem}

\subsection{Quasitoric manifolds and topological toric manifolds}\label{subsec:toric manifold}
In the pioneering work \cite{DJ91} Davis and Januszkiewicz suggested a topological generalization of \emph{smooth projective toric varieties}, which became
known as \emph{quasitoric manifolds}. A quasitoric manifold is a close $2n$-dimensional manifold $M$ with a locally standard
action of $T^n$ (i.e. it locally looks like the standard coordinatewise action of $T^n$ on $\Cc^n$) such that the quotient $M/T^n$ can be identified with a simple $n$-polytope $P$. Let us review their construction.

Let $P$ be a simple $n$-polytope, $\FF=\{F_1,\dots,F_m\}$ the set of facets of $P$. Given a map $\lambda: \FF\to\mathbb{Z}^n$, write $\lambda(F_i)$
in the standard basis of $\mathbb{Z}^n$:
\[\lambda(F_i)=\boldsymbol{\lambda}_i=(\lambda_{1i},\dots,\lambda_{ni})^T\in\mathbb{Z}^n,\quad 1\leqslant i\leqslant m.\]
If the matrix
\[\mathit{\Lambda}=
\begin{pmatrix}
\lambda_{11}&\ldots&\lambda_{1m}\\
\vdots&\ddots&\vdots\\
\lambda_{n1}&\ldots&\lambda_{nm}
\end{pmatrix}
\]
has the following property:
\begin{equation}\label{eq:matrix}
\det(\boldsymbol{\lambda}_{i_1},\dots,\boldsymbol{\lambda}_{i_n})=\pm1\quad \text{whenever }F_{i_1}\cap\dots\cap F_{i_n}\neq\emptyset\text{ in }P,
\end{equation}
then $\lambda$ is called a \emph{characteristic function} for $P$, and $\mathit{\Lambda}$ is called a \emph{characteristic matrix}.

Let $(P,\mathit{\Lambda})$ be a \emph{characteristic pair} consisting of a simple polytope $P$ and its characteristic matrix $\mathit{\Lambda}$.
Denote by $T_i:=\{(e^{2\pi\lambda_{1i}t},\dots, e^{2\pi\lambda_{ni}t})\in T^n\}$ the circle subgroup of $T^n=\mathbb{R}^n/\mathbb{Z}^n$.
For each point $x\in P$, define a subtorus
\[T(x):=\prod_{i:\,x\in F_i}T_i\subset T^n.\]
Then the \emph{quasitoric manifold} $M(P,\mathit{\Lambda})$ is defined to be
\begin{equation}\label{eq:quasitoric}
M(P,\mathit{\Lambda}):=P\times T^n/\sim,
\end{equation}
where the equivalence relation $\sim$ is given by $(x,g)\sim(x',g')$ if and only if $x=x'$ and $g^{-1}g'\in T(x)$.
The action of $T^n$ on $P\times T^n$ by the right translations descends to a {$T^n$-action} on $M(P,\mathit{\Lambda})$, and the orbit space of this action is just $P$.

Two characteristic pairs $(P,\mathit{\Lambda})$ and $(P',\mathit{\Lambda}')$ are said to be \emph{equivalent} if $P$ and $P'$ are combinatorially equivalent,
and $\mathit{\Lambda}=A\cdot\mathit{\Lambda}'\cdot B$, where $A\in GL(n,\mathbb{Z})$ and $B$ is a diagonal $(m\times m)$-matrix with $\pm 1$ on the diagonal.
Here we identify the sets of facets of $P$ and $P'$ so that the characteristic functions are defined on the same set.

We say that two $T^n$-manifolds $M$ and $M'$ are \emph{weakly equivariantly homeomorphic} if there is a homeomorphism $f:M\to M'$ and an automorphism $\theta: T^n\to T^n$ such that $f(t\cdot x)=\theta(t)\cdot f(x)$ for any $x\in M$ and $t\in T^n$.

\begin{prop}[{\cite[Proposition 1.8]{DJ91}}]\label{prop:equiv hoeom}
There is a one-to-one correspondence between weakly equivariant homeomorphism classes of quasitoric
manifolds and equivalence classes of characteristic pairs.
\end{prop}
In \cite{DJ91}, Davis and Januszkiewicz also generalized the construction of quasitoric manifolds to the cases where the base space is more general than a simple polytope. We {explain} such a generalization below.

Let $K$ be an $(n-1)$-dimensional simplicial complex on $[m]$ and let $K'$ denote its barycentric subdivision.
For each face $\sigma\in K$ (including $\emptyset$), let $F_\sigma$ denote the geometric realization of the poset
$K_{\geqslant\sigma}=\{\tau\in K:\tau\geqslant\sigma\}$.
Hence, for $\sigma\neq\emptyset$, $F_\sigma$ is the subcomplex of $K'$ consisting of all
simplices of the form $\sigma_0<\sigma_1<\cdots<\sigma_k$ with $\sigma\leqslant\sigma_0$, which is combinatorially equivalent to the cone over the barycentric subdivision of $\mathrm{lk}_K\sigma$, and $F_\emptyset$ is the cone over $K'$.
The polyhedron $P_K=F_\emptyset$ together with its decomposition into ``faces" $\{F_\sigma\}_{\sigma\in K}$ will be called a \emph{simple polyhedral complex}.
In particular, there are $m$ facets $F_i,\dots,F_m$ of $P_K$, where $F_i$ is the geometric realization of the star of the $i$th vertex of $K$ in $K'$.

Suppose $\lambda:[m]\to \Zz^n$, $i\mapsto\boldsymbol{\lambda}_i$ is a characteristic function, i.e. it satisfies the condition that for each $(k-1)$-face $\sigma=\{i_1,\dots,i_k\}\in K$, the lattice generated by $\{\boldsymbol{\lambda}_{i_1},\dots,\boldsymbol{\lambda}_{i_k}\}$ is a direct summand of $\Zz^n$ of rank $k$. Let $\mathit{\Lambda}$ be the characteristic matrix corresponding to $\lambda$.
Then define $M(P_K,\mathit{\Lambda}):=P_K\times T^m/\sim$, where the equivalence relation is defined exactly as in \eqref{eq:quasitoric}.

The case that $K$ is a starshaped $(n-1)$-sphere in the construction  above was considered in \cite{IFM13}.
It turns out that in this case $M(P_K,\mathit{\Lambda})$ is a compact smooth $2n$-dimensional manifold.
This is because it can also be obtained as a quotient of the moment-angle manifold $\ZZ_K$ by a freely and smoothly acting subtorus (see \cite[Proposition 7.3.12]{BP15} for the polytopal case).

{Roughly speaking, $M(P_K,\mathit{\Lambda})$ is a \emph{topological toric manifold} introduced in \cite{IFM13}.  
Indeed, by the definition in \cite{IFM13}, a topological toric manifold  
is a compact smooth $2n$-dimensional manifold with an effective smooth
action of  $(\mathbb{C}^\times)^n$ having an open dense orbit and covered by finitely many invariant open subsets each equivariantly diffeomorphic to a sum of complex $1$-dimensional
smooth representation spaces of $(\mathbb{C}^\times)^n$. A topological toric manifold  carries more information of smooth representation of $(\mathbb{C}^\times)^n$ than $M(P_K,\mathit{\Lambda})$. But in this paper, we only care about the topological information of a topological toric manifold, which is the same as $M(P_K,\mathit{\Lambda})$.}  

If $\mathit{\Lambda}$ is defined by the primitive vectors of the rays ($1$-dimensional cones) of a nonsingular complete fan $\Sigma$ and $K$ is the underlying sphere of $\Sigma$, then $M(P_K,\mathit{\Lambda})$ is homeomorphic to the compact smooth toric variety corresponding to $\Sigma$ (see for example \cite{Fra10}). So compact smooth toric varieties are contained in the family of topological toric manifolds. The construction of $M(P_K,\mathit{\Lambda})$ implies that Proposition \ref{prop:equiv hoeom} is also valid for topological toric manifolds. (See \cite[Corollary 8.2]{IFM13} for the equivariant homeomorphism case, and as we know, the extension to the weakly equivariant homeomorphism case corresponds to an automorphism of the torus $T^n$, which is uniquely determined by the matrix $A\in GL(n,\Zz)$ in the definition of the equivalence of characteristic pairs.)

The cohomology ring of a topological toric manifold has a simple expression, which is well-known as 
Jurkiewicz-Danilov Theorem for the case of compact smooth toric varieties (see for example \cite[Theorem 12.4.4]{CLS11}).
\begin{thm}[{\cite[Proposition 8.3]{IFM13}}]\label{thm:topo toric}
Let $M=M(P_K,\mathit{\Lambda})$ be a topological toric $2n$-manifold with $\mathit{\Lambda}=(\lambda_{ij})$, $1\leqslant i\leqslant n$, $1\leqslant j\leqslant m$.
Then the cohomology ring of $M$ is given by $H^*(M)=\mathbb{Z}[K]/\JJ$, where $\JJ$ is the ideal generated by linear forms
$\lambda_{i1}x_1+\cdots+\lambda_{im}x_m,\ 1\leqslant i\leqslant m$.
\end{thm}
Adapting the proof of \cite[Lemma III.2.4]{S96} to integral coefficients (the proof only uses the fact that in the characteristic matrix, the column vectors corresponding to any facet of $K$ span $\kk^n$, which is just the condition of \eqref{eq:matrix} for $\kk=\Zz$), one easily sees that in Theorem \ref{thm:topo toric}, $\mathbb{Z}[K]/\JJ$ is spanned by the face monomials $\xx_\sigma$, $\sigma\in K$. In particular, since $K$ is Gorenstein*, $(\mathbb{Z}[K]/\JJ)_{2n}=\mathbb{Z}$ (see for example \cite[Lemma 5.6.4]{BH98}) is spanned by $\xx_F$ for any facet $F$ of $K$. On the other hand, since any face $\sigma$ of $K$ is contained in some facet $F$, $\xx_\sigma\xx_{F\setminus\sigma}=\xx_F\neq 0$, it follows that any face monomial $\xx_\sigma$ is not zero in $\mathbb{Z}[K]/\JJ$. This fact will be used later.

\begin{rem}
	The result of \cite[Lemma 5.6.4]{BH98} is stated for field coefficient, but it can be easily extended to the integral coefficient case as follows. Since $K$ is Cohen-Macaulay, the condition of \eqref{eq:matrix} ensures that $\{\theta_i=\lambda_{i1}x_1+\cdots+\lambda_{im}x_m\}_{i=1}^n$ is a regular sequence in $\Zz[K]$ (cf. \cite[Lemma 7.3.27]{BP15}), i.e., $\Zz[K]$ is a free $\Zz[\theta_1,\dots,\theta_n]$-module. Hence $\Zz[K]/\JJ$ is a free $\Zz$-module of rank $\dim_\Qq(\Qq[K]/\JJ)$. Note that $\Qq[K]/\JJ=(\Zz[K]/\JJ)\otimes\Qq$ because $\Qq$ is flat as a $\Zz$-module.
\end{rem}

A result in \cite{CPS10} says that the rational cohomology of a topological toric manifold $M(P_K,\mathit{\Lambda})$ determines the rational $\mathrm{Tor}$-algebra of $K$. In \cite{BEMPP17}, this result was extended to the case when the coefficient ring is $\Zz$.
\begin{lem}[{\cite[Lemma 3.7]{CPS10}} and {\cite[Proposition 3.5]{BEMPP17}}]\label{lem:cps}
Let $M=M(P_K,\mathit{\Lambda})$ and $M'=M(P_{K'},\mathit{\Lambda}')$ be two $2n$-dimensional topological toric manifolds.
If there is a graded ring isomorphism $H^*(M)\cong H^*(M')$, then $|\VV(K)|=|\VV(K')|$ and there is an isomorphism of bigraded algebras
$\mathrm{Tor}(K)\cong \mathrm{Tor}(K')$.
\end{lem}

\subsection{$B$-rigidity and $BB$-rigidity of simplicial complexes}\label{subsec:B-rigidity}
It is well-known that the combinatorics of a simplicial complex and its face ring are determined by each other.
\begin{thm}[Bruns-Gubeladze \cite{BG96}]
Let $\kk$ be a field, and $K$ and $K'$ be two simplicial complexes. Suppose $\kk[K]$ and $\kk[K']$ are isomorphic as $\kk$-algebras.
Then there exists a bijective map $\VV(K)\to\VV(K')$ which induces an isomorphism between $K$ and $K'$.
\end{thm}
Note that the $\mathrm{Tor}$-algebra of $K$ is also determined by its face ring. In the converse direction, Buchstaber asked the following question in his lecture note \cite{B08}.
\begin{prob}
{Let $K$ and $K'$ be simplicial complexes. Suppose their $\mathrm{Tor}$-algebras are isomorphic as bigraded algebras. When there exists a combinatorial equivalence $\mathrm{core}\,K\approx \mathrm{core}\,K'$? (The reason why we consider the cores of simplicial complexes is that the $\mathrm{Tor}$-algebras of a simplicial complex and the cone over it are always the same.) A more topological question of when the cohomology ring isomorphism $H^*(\ZZ_K)\cong H^*(\ZZ_{K'})$ implies the  combinatorial equivalence $\mathrm{core}\,K\approx \mathrm{core}\,K'$ was asked in \cite{CPS10}. These lead to the following two notions of rigidity for simplicial complexes.} 
\end{prob}

\begin{Def}
{A simplicial complex $K$ is said to be \emph{bigraded $B$-rigid} (\emph{$BB$-rigid} for short) if any bigraded isomorphism $\mathrm{Tor}(K;\Zz)\cong \mathrm{Tor}(K';\Zz)$ implies a combinatorial
	equivalence  $\mathrm{core} K\approx \mathrm{core} K'$.} 
\end{Def}

\begin{Def}
	{A simplicial complex $K$ is said to be \emph{$B$-rigid} if any cohomology ring isomorphism $H^*(\ZZ_K)\cong H^*(\ZZ_{K'})$ of moment-angle complexes implies a combinatorial equivalence $\mathrm{core}\,K\approx \mathrm{core}\,K'$.}
\end{Def}

\begin{rem}
	Clearly, $B$-rigidity implies $BB$-rigidity. But it is not known whether $B$-rigidity is equivalent to
	$BB$-rigidity, and this is unlikely to be true in general.
\end{rem}
There is also another rigidity problem for simplicial complexes called \emph{combinatorial rigidity}. A simplicial complex $K$ is combinatorially rigid (over a field $\kk$) if the combinatorial equivalence class of $\mathrm{core}\,K$ is determined by the bigraded Betti numbers $\beta^{-i,2j}(K)=\dim_\kk\mathrm{Tor}^{-i,2j}(K)$. It is clear that combinatorial rigidity implies $BB$-rigidity. But the reverse implication does not hold. (A counterexample was found by Choi in \cite{Cho13}.)

In this paper, we mainly consider the {$BB$-rigidity} problem. In fact, {$BB$-rigid} simplicial complexes are rare. For example, letting $K_1$ and $K_2$ be simplicial complexes and $K_1\cup_\sigma K_2$ be the simplicial complex obtained from $K_1$ and $K_2$ by gluing along a common simplex $\sigma$, the $\mathrm{Tor}$-algebra of $K_1\cup_\sigma K_2$ does not depend on the choice of $\sigma$ (\cite[{Proposition 4.6}]{FW21}), but the combinatorial type of $K_1\cup_\sigma K_2$ generally does.

Intuitively, the more algebraic information $\mathrm{Tor}(K)$ has, the more likely $K$ is {$BB$-rigid}, so the most interesting case is the class of Gorenstein* complexes since their $\mathrm{Tor}$-algebras have Poincar\'e duality. However, even in this class the {$BB$-rigidity} does not hold in most cases.
For example, let $K_1$ and $K_2$ be two homology spheres of the same dimension. Then the set $\CC(K_1\# K_2)$ in general contains more than one combinatorial type, but the $\mathrm{Tor}$-algebras of all simplicial complexes in $\CC(K_1\# K_2)$ are the same (see \cite[{Theorem 4.2}]{FW21}).

Examples of irreducible non-{$BB$-rigid} simplicial spheres can be found in \cite{FCMW16} (see also \cite[Proposition 2.4]{B17}): there are three different irreducible polytopal $3$-spheres with $8$ vertices that have the same $\mathrm{Tor}$-algebras; Iriye  \cite{Iri16} showed that their corresponding moment-angle manifolds are all homeomorphic to a connected sum of sphere products, with a summand being the product of
three spheres, whereas all  previously known examples of this kind have every summand being
a product of two spheres. 

Another interesting example is that there are non-{$BB$-rigid} polytopal $4$-spheres with $8$ vertices, whose combinatorial types are determined by the cohomology of quasitoric manifolds over them \cite{CP19}. 

All non-{$BB$-rigid} examples  above are not flag.
Flag non-{$BB$-rigid} examples can be constructed by \emph{puzzle-move} operations in the next subsection (see Section \ref{sec:SCC} for more details).

On the other hand, the {$BB$-rigidity} problem is solved positively for some particular classes of simplicial spheres, such as joins of boundary complexes of finite simplices and the connected sum of such a complex and the boundary of a simplex \cite[Lemma 5.2 and Theorem 6.3]{CPS10}, as well as some highly symmetric simplicial $2$-spheres \cite{CK11}. (In fact, \cite{CPS10} and \cite{CK11} proved the stronger results that these simplicial spheres are combinatorially rigid.)
Another family of {$BB$-rigid} $2$-spheres was obtained by the authors in \cite{FMW15}, and one of the purposes of this paper is to generalize this result to all dimensions.

Since the topology of a moment-angel complex $\mathcal {Z}_{K}$ is uniquely determined by the combinatorics of $K$, if $K$ is a {$BB$-rigid} homology sphere, then the moment-angle manifold $\ZZ_K$ is cohomologically rigid in the bigraded sense. This is our strategy to deal with Problem \ref{prob:rigidity for m-a}.

\subsection{Puzzle-moves and puzzle-rigidity}\label{subsec:puzzle}
 In \cite{B17}, Bosio defined two transformations of simple polytopes, which preserve the diffeomorphism classes of the corresponding moment-angle manifolds. Here we introduce one of these constructions in the dual simplicial version.
\begin{Def}
Two vertices of a simplicial complex $K$ are called \emph{wedge equivalent} if the transposition of
	these two vertices induces an automorphism of $K$. In particular, we say a vertex is wedge equivalent to itself.
\end{Def}

\begin{Def}\label{def:wedge}
	Let $K$ be a simplicial complex on vertices $[m]$. Choose a fixed vertex $i$ in $K$ and define the \emph{simplicial wedge} $K(i)$ of $K$ on $i$ to be the simplicial complex with $m+1$ vertices labeled by $\{1,\dots,i-1,i,i',i+1,\dots,m\}$:
		\[K(i)=(\{i,i'\}*\mathrm{lk}_{K}\{i\})\cup(\{i\}*K_{[m]\setminus\{i\}})\cup(\{i'\}*K_{[m]\setminus\{i\}}).\]
	The two vertices $i$ and $i'$ are called the \emph{main vertices} of the simplicial wedge $K(i)$.
\end{Def}
For a homology sphere $K$, it is easy to see that the main vertices of a simplicial wedge of $K$ are wedge equivalent, and the two vertices of $\partial\Delta^1$ in the suspension $\partial\Delta^1*K$ are wedge equivalent. Conversely, if two vertices of a homology sphere are wedge equivalent, then they  have to be one of the above cases. To see this, suppose that $i$ and $i'$ are wedge equivalent in $K$. By definition, if $\sigma\in K$ is a facet such that $i\in\sigma$ and $i'\not\in\sigma$, then $\{i'\}\cup\sigma\setminus\{i\}$ is also a facet of $K$, and vice versa. Thus, if $\{i,i'\}\not\in K$, then $\mathrm{lk}_K\{i\}=\mathrm{lk}_K\{i'\}$. Since $K$ and the subcomplex $\partial\Delta^1*\mathrm{lk}_K\{i\}$ in $K$ are both homology spheres of the same dimension, we have $K=\partial\Delta^1*\mathrm{lk}_K\{i\}$. Note that, by Alexander duality, a homology sphere can not contain a proper subset which is also a homology sphere of the same dimension. On the other hand, if $\{i,i'\}\in K$, then for any $\sigma\in\mathrm{lk}_K\{i\}$ with $i'\not\in\sigma$, we have $\sigma\in \mathrm{lk}_K\{i'\}$ since $i$ and $i'$ are wedge equivalent. It follows that the simplicial wedge $\mathrm{lk}_K\{i\}(i')$ of $\mathrm{lk}_K\{i\}$ is a subcomplex of $K$, and then $K$ is isomorphic to $\mathrm{lk}_K\{i\}(i')$ since they are both homology spheres of the same dimension.

Let $P$ and $P'$ be two simple $n$-polytopes in $\mathbb{R}^n$. Suppose that $H$ is a hyperplane of $\mathbb{R}^n$ such that $H\cap P,\,H\cap P'\neq\emptyset$ and the vertices of $P$ and $P'$ are not in the intersections. Let $K$ and $K'$ be the polytopal $(n-1)$-spheres corresponding to the dual simplicial polytopes of $P$ and $P'$ respectively, and let $L$ and $L'$ be the polytopal $(n-2)$-spheres corresponding to the dual simplicial polytopes of the simple $(n-1)$-polytopes $H\cap P$ and $H\cap P'$ respectively. We may embed $L$ in $K$ (resp. $L'$ in $K'$) as follows. Note that a facet of $H\cap P$ is of the form $H\cap F$, where $F$ is a facet of $P$. Hence a simplex $\sigma=\{i_1,\dots,i_k\}$ in $L$ corresponds to the nonempty intersection $H\cap F_{i_1}\cap\cdots\cap F_{i_k}$, where $F_{i_s}$ are facets of $P$, and then $\sigma$ can be viewed as the simplex in $K$ dual to $F_{i_1}\cap\cdots\cap F_{i_k}$. Clearly, $L$ (resp. $L'$) separates $K$ (resp. $K'$) into two simplicial disks $K=K_+\cup_{L}K_-$ (resp. $K'=K'_+\cup_{L'}K'_-$).
We assume that there are isomorphisms $i_+:K_+\to K'_+$ and $i_-:K_-\to K'_-$. By abuse of notation, we write their restrictions to $L$ also as $i_+$ and $i_-$.

{Consider the automorphism $\phi=(i_-)^{-1}\circ i_+$ of $L$.
If $\phi$ fixes wedge equivalent classes of vertices, i.e., for any vertex $i$ of $L$, $\phi(i)$ is wedge equivalent to $i$,}
then we say that we pass from $K$ to $K'$ by a \emph{puzzle-move}, denoted $(K,L,\phi)$, and $K'$ is denoted $K_{L,\phi}$.
Two {polytopal} spheres are called \emph{puzzle-equivalent} if one can be obtained from the other by a sequence of puzzle-moves.
We say that a {polytopal} sphere is \emph{puzzle-rigid} if it is not puzzle-equivalent to other polytopal spheres.

\begin{thm} \cite[Theorem 3.5]{B17}\label{thm:puzzle}
{Suppose that $K$ and $K'$ are puzzle-equivalent polytopal spheres. Then there is a graded diffeomorphism $\ZZ_K\cong \ZZ_{K'}$, i.e., a diffeomorphism inducing a bigraded ring isomorphism $\mathrm{Tor}(K)\cong \mathrm{Tor}(K')$.}
\end{thm}

Theorem \ref{thm:puzzle} implies that for {polytopal} spheres,  {$BB$-rigidity} implies puzzle-rigidity.
{Note that {$BB$-rigidity} and puzzle-rigidity are not equivalent in general (see \cite[\S 4]{Bos15} and \cite[\S 3.3]{B17}). In \cite{B17}, Bosio conjectured that if there is a graded diffeomorphism  $\ZZ_K\cong \ZZ_{K'}$ for simplicial $2$-spheres $K,\,K'$, then $K$ and $K'$ are puzzle equivalent (\cite[Conjecture 1]{B17}). Here we give a stronger conjecture:}
\begin{conj}\label{conj:puzzle}
	{Puzzle-rigidity and $BB$-rigidity are equivalent for simplicial $2$-spheres.}
\end{conj}
See Section \ref{sec:n=3} for further discussion of this problem.

\section{{$BB$-rigidity} of a class of homology spheres}\label{sec:SCC}
In \cite{FMW15}, we proved that the $B$-rigidity holds for a class of simplicial $2$-spheres, whose dual class is also known as the \emph{Pogorelov class} (see \cite{BEMPP17}). It is characterized by the flagness condition and the ``no $\square$-condition" described below.

\begin{Def}
A $1$-dimensional simplicial complex is called a \emph{$k$-circuit} if it is a triangulation of $S^1$ and contains $k$ vertices.
A simplicial complex $K$ is said to satisfy the \emph{no $\square$-condition} if there is no full subcomplex $K_I\subset K$ such that $K_I$ is a $4$-circuit.
\end{Def}
\begin{thm}[\cite{FMW15}, {see Remark \ref{rem:n=3}}]\label{thm:B-rigidity of 2-spheres}
Suppose $K$ is a flag simplicial $2$-sphere satisfying the no $\square$-condition. Then $K$ is $B$-rigid.
\end{thm}
A natural question  arises of whether these conditions can be applied to higher dimensional simplicial spheres to get more $B$-rigid or {$BB$-rigid} examples.
However the following result makes this question meaningless.
\begin{thm}[{\cite[Proposition 12.2.2]{CLS11}}]\label{thm:no 4-circuit for higher}
Suppose a flag homology sphere of dimension $n$ satisfies the no $\square$-condition. Then $n\leqslant3$.
\end{thm}

Even for the case $n=3$, the examples satisfying the conditions in Theorem \ref{thm:no 4-circuit for higher} are rather rare.
(The only known example is the $600$-cell.)
So we should change our strategy in higher dimensions.

In fact, we notice that only a special combinatorial property, which is essentially weaker than the no $\square$-condition for higher dimensional spheres, is needed in the proof of Theorem \ref{thm:B-rigidity of 2-spheres}. It is the following condition.

\begin{Def}\label{def:SCC}
A flag complex $K$ with $m\geqslant3$ vertices and $K=\mathrm{core}\,K$ is said to satisfy the \emph{separable circuit condition} (SCC for short) if for any three different vertices $i_1,i_2,i_k$ with $\{i_1,i_2\}\not\in K$,
there exists a full subcomplex $K_I$ with $|K_I|\cong S^1$ such that $i_1,i_2\in I$, $i_k\notin I$ and $\w H_0(K_{J})\neq0$ for $J=\{i_k\}\cup I\setminus\{i_1,i_2\}$.
\end{Def}
\begin{exmp}
The dual polytopal $2$-spheres of fullerenes satisfy the SCC. (A \emph{fullerene} is a simple $3$-polytope with only pentagonal and hexagonal facets.)
This can be deduced from Proposition \ref{prop:4-circuit} below {and the fact that the boundary spheres of the polytopes dual
	to fullerenes are flag and satisfy the no $\square$-condition (see \cite{Dos98,Dos03} or \cite[Corollary 3.16]{BE15}).}
\end{exmp}

\begin{exmp}\label{exmp:join}
If two flag homology spheres $K$ and $K'$ both satisfy the SCC, then so does $K*K'$. Indeed, if $\{i_1,i_2\}\in MF(K*K')$, then either $\{i_1,i_2\}\in MF(K)$ or $\{i_1,i_2\}\in MF(K')$. So we may assume that $\{i_1,i_2\}\in MF(K)$.
The case $i_k\in\VV(K)$ is trivial, by the SCC on $K$.
For $i_k\in\VV(K')$, we can choose a vertex $j_1\in\VV(K')$ such that $\{i_k,j_1\}\in MF(K')$. By Proposition \ref{prop:nsc} below, $K'$ is not a suspension. Hence $\mathrm{lk}_{K'}\{i_k\}\neq\mathrm{lk}_{K'}\{j_1\}$ since otherwise the subcomplex $\partial\Delta^1*\mathrm{lk}_{K'}\{i_k\}$ of $K'$, $\VV(\Delta^1)=\{i_k,j_1\}$, is a homology sphere of the same dimension as $K'$, and so $K'=\partial\Delta^1*\mathrm{lk}_{K'}\{i_k\}$, a contradiction. Note that $\mathrm{lk}_{K'}\{i_k\}$ is a full subcomplex of $K'$ and it can not be a proper subcomplex of $\mathrm{lk}_{K'}\{j_1\}$ since they are both homology spheres of the same dimension, so there exists $j_2\in\VV(\mathrm{lk}_{K'}\{i_k\})\setminus\VV(\mathrm{lk}_{K'}\{j_1\})$, and then $\{j_1,j_2\}\in MF(K')$.
It is easily verified that $I=\{i_1,i_2,j_1,j_2\}$ is the desired subset.
\end{exmp}

See Appendix \ref{app:scc} for more examples of flag spheres satisfying the SCC.

\begin{prop}\label{prop:nsc}
Let $K$ be a flag $\kk$-homology sphere of dimension $n\geqslant2$. Assume $K$ satisfies the SCC. Then for any missing face $I\in MF(K)$, say $I=\{i_1,i_2\}$, and any $K_J\subset\mathrm{lk}_K\{i_1\} \cap\mathrm{lk}_K\{i_2\}$, we have $\w H^{n-2}(K_J;\kk)=0$.
Moreover, $K$ is not a suspension.
\end{prop}
\begin{proof}
The coefficient $\kk$ is omitted throughout the proof.
Suppose on the contrary that $\w H^{n-2}(K_J)\neq0$ for some $K_J\subset\mathrm{lk}_K\{i_1\} \cap\mathrm{lk}_K\{i_2\}$.
Since $K$ is flag, $K_{I\cup J}=K_I*K_J$, so that $\w H^{n-1}(K_{I\cup J})\neq0$. Set $S=\VV(K)\setminus(I\cup J)$.
Then $\w H_0(K_S)\neq0$ by {Theorem \ref{thm:Gorenstein}}.
Taking two vertices $s_1$ and $s_2$ to be in different components of $K_S$,
the assumption that $K$ satisfies the SCC shows that there exists a subset $U\subset\VV(K)$ with $|K_U|\cong S^1$, such that $s_1,s_2\in U$, $i_1\notin U$ and $\w H_0(K_{V})\neq0$, where $V=\{i_1\}\cup U\setminus\{s_1,s_2\}$.

The subcomplex $K_{U\setminus\{s_1,s_2\}}$ clearly has two components, say $K_{U_1}$ and $K_{U_2}$. We claim that $U_i\cap J\neq\emptyset$ for $i=1,2$.
To see this, note that $K_{U_i\cup\{s_1,s_2\}}$ is a path connecting $s_1$ and $s_2$, so $U_i\cap (\{i_2\}\cup J)\neq\emptyset$ for $i=1,2$, since $s_1$ and $s_2$ are in different components of $K_S$ and $i_1\not\in U$.
Moreover, {we have $i_2\not\in U_i$ for $i=1,2$, since otherwise, say $i_2\in U_1$, the fact that $U_i\cap (\{i_2\}\cup J)\neq\emptyset$ for $i=1,2$ would give $U_2\cap  J\neq\emptyset$. But  $K_J\subset\mathrm{lk}_K\{i_2\}$ and $i_2\in U_1$ would imply that $K_{U_1\cup U_2}$ is path connected, a controdiction.} So the claim is true. However since $K_J\subset\mathrm{lk}_K\{i_1\}$, we would have $\w H_0(K_{V})=0$, which contradicts the assumption in the first paragraph.
Hence we prove the first statement of the proposition.

For the second statement, note that if $K=K_I*L$ with $I=\{i_1,i_2\}\in MF(K)$, then $L$ would be a $\kk$-homology $(n-1)$-sphere, and for any $j\in \VV(L)$, $\mathrm{lk}_L\{j\}$ would be a $\kk$-homology $(n-2)$-sphere, so that $\w H^{n-2}(\mathrm{lk}_L{j})\neq0$.
However the assumption that $K$ is flag implies that $\mathrm{lk}_L\{j\}\subset \mathrm{lk}_K\{i_1\} \cap\mathrm{lk}_K\{i_2\}$ is a full subcomplex of $K$, contradicting the first statement of the proposition.
\end{proof}

Proposition \ref{prop:nsc} is equivalent to the following necessary condition for a flag homology $n$-sphere $K$ to satisfy the SCC: There must be no full subcomplex $K_I\subset K$ such that $K_I$ is a suspension and $\w H^{n-1}(K_I)\neq0$. We refer to this condition as the \emph{no suspension of codimension one condition} (NSC for short).
Since the NSC is much easier to be verified (for computer algorithms) than the SCC, we expect that NSC $\Rightarrow$ SCC. So we propose the following conjecture.

\begin{conj}\label{conj:SCC-NSC}
The SCC and NSC are equivalent for all flag homology spheres.
\end{conj}
It is easy to see that for flag $2$-spheres the NSC implies the no $\square$-condition.
The following result shows that the converse is also true, and therefore Conjecture \ref{conj:SCC-NSC} holds for flag $2$-spheres.

\begin{prop}[\cite{FMW15}, see also {\cite[Lemma A.3]{BEMPP17}}]\label{prop:4-circuit}
Let $K$ be a flag $2$-sphere. If $K$ satisfies the no $\square$-condition, then $K$ satisfies the SCC.
\end{prop}
We will prove a more general version of Proposition \ref{prop:4-circuit} in Appendix \ref{app:n=3}. Here we digress to introduce the following result about the puzzle-rigidity of polytopal spheres:

\begin{prop}
Let $K$ be a flag polytopal $n$-sphere. If $K$ satisfies the NSC, then $K$ is puzzle-rigid.
\end{prop}
\begin{proof}
	Suppose that $K'=K_{L,\phi}$ for a puzzle-move $(K,\phi,L)$. By the definition of a puzzle-move, it suffices to show that $\phi$ can be extended to an automorphism of $K_+$ or $K_-$ in the decomposition $K=K_+\cup_L K_-$. As in the discussion following Definition \ref{def:wedge}, $L$ is either  a simplicial wedge  $\Gamma(v)$ of an $(n-2)$-sphere $\Gamma$ on some vertex $v$, or $L=(\partial\Delta^1*\cdots *\partial\Delta^1)*\Gamma$, where $\Gamma$ is a sphere of proper dimension and any vertex in $\VV(\Gamma)$ is not  wedge equivalent to any vertex in $\VV(L)$.

	In the first case, note that  $\{v,v'\}\in L$ and for any $\sigma\in \Gamma_{\VV(\Gamma)\setminus\{v\}}$, $\sigma\cup\{v\}$ and $\sigma\cup\{v'\}$ are both in $L$, so the flagness of $K$ implies that $D:=\Delta^1*(L_{\VV(L)\setminus\{v,v'\}})\subset K$.
	Since $D$ is a simplicial $n$-disk with boundary $L$, it follows that in the decomposition $K=K_+\cup_L K_-$, one of $K_+$ and $K_-$, say $K_+$ contains $D$, noting that $|D|\setminus |L|$ is path-connected and that $|K|\setminus|L|$ has two components $|K_+|\setminus |L|$ and $|K_-|\setminus |L|$. Let $Y=|K|\setminus|D|=(|K_+|\setminus |D|)\sqcup (|K_-|\setminus |L|)$. Then applying the Alexander duality to $|K|$, $|D|$ and $Y$ we obtain $\w H^0(Y)=0$, thus $|K_+|\setminus |D|=\emptyset$, i.e. $K_+=D$. Let $I\subset\VV(L)$ be the set of vertices wedge equivalent to $v$. If $I\setminus\{v,v'\}\neq\emptyset$, then for any $u\in I\setminus\{v,v'\}$, we have seen that $\{u,v\}\in L$. Since $u$ is wedge equivalent to $v$ in $L$, $u$, $v$ must be the main vertices of a simplicial wedge. Hence the same argument as above shows that $K_+=\Delta^1*(L_{\VV(L)\setminus \{u,v\}})$. Similarly, $K_+=\Delta^1*(L_{\VV(L)\setminus \{u,v'\}})$, and so $K_+=\Delta^2*(L_{\VV(L)\setminus \{u,v,v'\}})$. Continuing inductively, we eventually obtain $K_+=\Delta^{|I|-1}*(L_{\VV(L)\setminus I})$. Clearly, $\phi$ transposes the vertices in $I$ and $\phi(L_{\VV(L)\setminus I})=L_{\VV(L)\setminus I}$. Thus $\phi$ can be extended to an automorphism of $K_+$ in the obvious way.
	
	For the second case, we may assume that $\Delta^1\not\subset K$ for every join factor $\partial\Delta^1$ of $L$, since otherwise $(\partial\Delta^1*\cdots *\Delta^1*\cdots *\partial\Delta^1)*\Gamma\subset K$ is an $n$-dimensional ball with boundary $L$, and we reduce to the previous argument.
		The flagness of $K$ implies that $K_{\VV(L)}=(\partial\Delta^1*\cdots *\partial\Delta^1)*K_{\VV(\Gamma)}$.
	So if in addition $K$ satisfies the NSC, then we have $\w H^{n-1}(K_{\VV(L)})=0$, and $\w H_0(|K|-|K_{\VV(L)}|)=0$ by Alexander duality. Thus one of $K_+$ and $K_-$, say $K_+$, satisfies $\VV(K_+)=\VV(L)$. We claim that $K_+=D:=(\partial\Delta^1*\cdots *\partial\Delta^1)*(K_+)_{\VV(\Gamma)}$. The inclusion $K_+\subset D$ is obvious. Since $K$ is flag, $D\subset K$. If $\sigma\in D\setminus L$, then either $\sigma\in (K_+)_{\VV(\Gamma)}\setminus\Gamma$ or one of the faces of $\partial\sigma$ belongs to $(K_+)_{\VV(\Gamma)}\setminus\Gamma$. Hence $\sigma\in K_+\setminus L$ since $D\setminus L\subset (K_+\setminus L)\sqcup (K_-\setminus L)$ and $(K_+)_{\VV(\Gamma)}\setminus\Gamma\subset K_+\setminus L$. The claim is proved. Since $\phi$ can only transpose the vertices of each $\partial\Delta^1$ and fixes the vertices in $\VV(\Gamma)$, it follows that $\phi$ can be extended to an automorphism of $K_+$. We conclude that $K$ is puzzle-rigid.
\end{proof}

\begin{thm}\label{thm:B-rigidity}
If $K$ is a flag rational homology sphere satisfying the SCC, then $K$ is {$BB$-rigid}.
\end{thm}

As a direct consequence we get the following (bigraded) cohomological rigidity result for moment-angle manifolds.

\begin{cor}\label{cor:rigidity of m-a manifold}
Let $\ZZ_K$ and $\ZZ_{K'}$ be two moment-angle manifolds. Suppose the cohomology of $\ZZ_K$ and $\ZZ_{K'}$ are isomorphic as bigraded rings.
Then if $K$ is flag and satisfies the SCC, $\ZZ_K$ and $\ZZ_{K'}$ are homeomorphic.
\end{cor}

According to Lemma \ref{lem:missing face and tor} a simplicial complex $K$ is flag if and only if \[\mathrm{Tor}^{-1,\,2j}(K)=0\ \text{ for } j\geqslant3.\]
Thus Theorem \ref{thm:B-rigidity} is an immediate consequence of the following theorem. {Note that $\mathrm{Tor}(K;\Qq)\cong \mathrm{Tor}(K;\Zz)\otimes\Qq$.}
\begin{thm}\label{thm:rigid}
Let $K$ be a flag $\kk$-homology ($\kk$ a field) sphere satisfying the SCC, $K'$ be a flag complex satisfying $K'=\mathrm{core}\,K'$.
If there is a graded ring isomorphism $\phi:H^*(\ZZ_K;\kk)\xr{\simeq} H^*(\ZZ_{K'};\kk)$, then $K\approx K'$.
\end{thm}

\begin{rem}\label{rem:n=3}
In the case $\dim K=2$, Theorem \ref{thm:rigid} remains valid if the flagness condition of $K'$ is omitted. So flag $2$-spheres satisfying the no $\square$-condition are {$B$-rigid by Proposition \ref{prop:4-circuit}}. 

To see this, first note that $H^*(\ZZ_K;\kk)$ is a Poincar\'e duality $\kk$-algbera. So if $H^*(\ZZ_K;\kk)\cong H^*(\ZZ_{K'};\kk)$, then $K'$ is a $\kk$-homology sphere by {Theorem \ref{thm:Gorenstein} and \cite[Theorem 4.6.8]{BP15}}. For a $\kk$-homology sphere of dimension $n-1\geqslant 2$, there is a well-known formula (the rigidity inequality in \cite{Ka87}): $f_1\geqslant nf_0-\binom{n+1}{2}$, where $f_0$ and $f_1$ are the numbers of vertices and edges respectively. In particular, when $n=3$, the inequality above turns into equality. {Hence, for a $\kk$-homology sphere $L$ of dimension $n-1\geqslant 2$}, we have 
\begin{equation}\label{eq:betti inequality}
\begin{split}
{\dim_\kk H^3(\ZZ_L;\kk)}=\binom{f_0}{2}-f_1&\leqslant \frac{f_0^2-(2n+1)f_0}{2}+\binom{n+1}{2}\\
&=2n^2-(2d-1)n+\frac{d^2-d}{2},
\end{split}
\end{equation}
where $d=f_0+n$ is the top degree of {$H^*(\ZZ_L)$}, and the first equality comes from Hochster's formula. {If $n=3$, \eqref{eq:betti inequality} becomes an equality.}
\eqref{eq:betti inequality}, together with the clear fact that $d\geqslant 2n+1$ and the fact that the function $f(n)=2n^2-(2d-1)n$ is strictly decreasing for $n\leqslant \frac{d-1}{2}$,
implies that $K'$ is a simplicial sphere of dimension $\leqslant 2$. By Lemma \ref{lem:nil of gorenstein} below, $|K'|\neq S^1$, so  $K'$ is also a simplicial $2$-sphere, noting that in dimension $\leqslant 2$, $\kk$-homology spheres are actually simplicial spheres.   
Finally, since whether a simplicial $2$-sphere $K$ is flag or not depends only on some algebraic property of $H^*(\ZZ_K)$ (see \cite[Section 6]{FW15} or \cite[Theorem 4.8]{BEMPP17}, {\cite[Proposition 2.2]{Ero20}}), Theorem \ref{thm:rigid} applies {for the two simplicial $2$-spheres $K$ and $K'$}.
\end{rem}

\section{Proof of Theorem \ref{thm:rigid}}
Throughout this section, the coefficient ring $\kk$ is a field, and the cohomology with coefficients in $\kk$ will be implicit.

First let us recall some notions in commutative algebra.
\begin{Def}
Let $A=\bigoplus_{i\geqslant0}A^i$ be a nonnegatively graded
commutative connected $\kk$-algebra. The \emph{nilpotence length} of $A$, denoted by $nil(A)$,
is the greatest integer $n$ (or $\infty$) such that \[(A^+)^{*n}:=A^+\cdots A^+ \text{ ($n$ factors) }\neq0.\]
\end{Def}

\begin{Def}
Let $R$ be a commutative ring. For an element $r\in R$, the \emph{annihilator} of $r$ is defined to be
\[\text{ann}(r):=\{a\in R: a\cdot r=0\}.\]
In particular, if $R=\bigoplus_i R^i$ is a graded commutative ring, then the \emph{annihilator of degree $k$} of $r$ is defined to be
\[\text{ann}_k(r):=\{a\in R^k: a\cdot r=0\}.\]
\end{Def}
\begin{conv}
In what follows we will often not distinguish the cohomology ring $H^*(\ZZ_K)$ of $\ZZ_K$ and the BHR $\bigoplus_{J\subset[m]}\w H^*(K_J)$ of $K$ because of the isomorphism \eqref{eq:Hochster ring} (the proofs throughout this paper do not involve the sign problem in Remark \ref{rem:sign}), so that an element $\xi_J\in \w H^k(K_J)$ can also be seen as an element in $H^{|J|+k+1}(\ZZ_K)$.
In particular, if $K$ is flag, then for any $\omega\in MF(K)$, $\kk\cong\w H^0(K_\omega)\subset H^3(\ZZ_K)$, denote  $\w\omega$ a generator of this subgroup of $H^3(\ZZ_K)$.
\end{conv}
\begin{lem}\label{lem:upper bound of nil}
Let $K$ be a simplicial complex with $m$ vertices, $R=H^*(\ZZ_K)$. Let $[c]$ be a homogenous cohomology class in $R$.
Suppose $[c]\in (R^+)^{*k}$ and $0\neq p_J([c])\in\w H^i(K_J)$ for some $J\subset[m]$; here $p_J$ is the projection map in \eqref{eq:projection}. Then $i\geqslant k-1$.
\end{lem}
\begin{proof}
It is easy to see that the condition $[c]\in (R^+)^{*k}$ implies $p_J([c])\in (R^+)^{*k}$. {Thus the statement of the lemma  follows from the degree rule of BHR.}
\end{proof}

\begin{lem}\label{lem:nil of gorenstein}
Let $K$ be a $\kk$-homology $n$-sphere with $m$ vertices, $R=H^*(\ZZ_K)$.
Then $nil(R)\leqslant n+1$. If $K$ is flag in addition, then $nil(R)=n+1$, and $(R^+)^{*(n+1)}=R^{m+n+1}$.
\end{lem}
\begin{proof}
The inequality is a direct consequence of Lemma \ref{lem:upper bound of nil}.
To prove the equality for the flag case, we use an induction on the dimension $n$.
The case $n=0$ is trivial. Suppose inductively that this is true for $\dim K<n$. Take $L$ to be the link of a vertex of $K$.
Then $L$ is a $\kk$-homology $(n-1)$-sphere, and since $K$ is flag, $L$ is a full subcomplex of $K$.
So by induction, $nil(H^*(\ZZ_L))=n$ and $(H^+(\ZZ_L))^{*n}=\w H^{n-1}(K_I)\cong \kk$, where $I=\VV(L)$. Let $J=[m]\setminus\VV(L)$. Then by {Theorem \ref{thm:Gorenstein}},
\[\w H^*(K_J)=\w H^0(K_J)=\w H_{n-1}(K_I)=\kk.\]
Since $R$ is a Poincar\'e duality algebra we have
\[\w H^{n-1}(K_I)\cdot\w H^0(K_J)=\w H^n(K)=R^{m+n+1}.\]
Note that $H^*(\ZZ_L)$ is a subring of $H^*(\ZZ_K)$ since $L$ is a full subcomplex of $K$. Hence from the definition of nilpotence length, we immediately get $nil(R)\geqslant n+1$.  So $nil(R)= n+1$, since we already have the inequality $nil(R)\leqslant n+1$.

It remains to prove the last statement. First, the above argument shows that $R^{m+n+1}\subset (R^+)^{*(n+1)}$.
If $\xi\in R$ and $\xi\notin R^{m+n+1}$, then
there exists a proper subset $I\subsetneq[m]$ such that $0\neq p_I(\xi)\in\w H^*(K_I)$, and therefore $\xi\notin (R^+)^{*(n+1)}$ by Lemma \ref{lem:upper bound of nil} {and the obvious fact that $\w H^{n}(K_I)=0$ for $I\neq [m]$.}
This gives the inverse inclusion relation $(R^+)^{*(n+1)}\subset R^{m+n+1}$.
\end{proof}

\begin{cor}\label{cor:goren}
Let $K$ be a flag $\kk$-homology sphere, $K'$ be a flag simplicial complex with $K'=\mathrm{core}\,K'$. Suppose there is a graded ring isomorphism  $\phi:H^*(\ZZ_K)\xr{\simeq} H^*(\ZZ_{K'})$. Then $K'$ is also a $\kk$-homology sphere of the same dimension as $K$ and $|\VV(K')|=|\VV(K)|$.
\end{cor}
\begin{proof}
{By Theorem \ref{thm:Gorenstein} and \cite[Theorem 4.6.8]{BP15}, $K'$ is also a $\kk$-homology sphere.}
The fact that $\dim K'=\dim K$ and $|\VV(K)|=|\VV(K')|$ follows from Lemma \ref{lem:nil of gorenstein}.
\end{proof}

\begin{prop}\label{prop:scc}
Let $K$ and $K'$ be two simplicial complexes, and suppose $K$ is flag and satisfies the SCC.
Assume that there is a graded ring isomorphism $\phi:H^*(\ZZ_K)\xr{\simeq} H^*(\ZZ_{K'})$. 
\begin{enumerate}[(a)]
\item\label{st:a} If $\w H^0(K_I)\neq0$ for some $I\subset\VV(K)$, $|I|\leqslant3$, then $\phi(\w H^0(K_I))=\w H^0(K'_{J})$ for some $K_J'\approx K_I$. The converse is also true: if $\w H^0(K_J')\neq0$ for some $J\subset\VV(K')$, $|J|\leqslant3$, then $\phi^{-1}(\w H^0(K_J'))=\w H^0(K_I)$ for some $K_I\approx K_J'$.
\item\label{st:b} In \eqref{st:a}, if $|I|=3$ and $\phi(\w H^0(K_I))=\w H^0(K_J')$, then for any $\omega\in MF(K_I)$, $\phi(\w\omega)=\w\omega'$ for some $\omega'\in MF(K'_J)$.
\item\label{st:c} If $K_I\subset K$ is an $n$-circuit, then there exists an $n$-circuit $K'_J\subset K'$ such that $\phi(\w H^1(K_I))=\w H^1(K'_{J})$.
\end{enumerate}
\end{prop}
Since the proof of Proposition \ref{prop:scc} involves many complicated algebraic arguments, we include it in Appendix \ref{app:proof of prop:scc}.

\begin{nota}
In Proposition \ref{prop:scc}, the isomorphism $\phi$ induces a bijection between $MF(K)$ and $MF(K')$.
We denote this map by $\phi_\MM:MF(K)\to MF(K')$.
\end{nota}

Here is a consequence of Proposition \ref{prop:scc}.
\begin{cor}\label{cor:intersect}
	Under the hypotheses of Proposition \ref{prop:scc}, if two different missing faces $\omega_1,\omega_2\in MF(K)$ satisfy $\omega_1\cap\omega_2\neq\emptyset$, then $\phi_\MM(\omega_1)\cap\phi_\MM(\omega_2)\neq\emptyset$.
\end{cor}
\begin{proof}
	Note that $\omega_1\cap\omega_2\neq\emptyset$ implies that $|\omega_1\cup\omega_2|=3$ and $\w H^0(K_{\omega_1\cup\omega_2})\neq 0$. Hence by Proposition \ref{prop:scc} \eqref{st:b}, if $\omega_1\cap\omega_2\neq\emptyset$, then there exists $J\subset\VV(K')$, $|J|=3$, such that $\phi_\MM(\omega_1),\phi_\MM(\omega_2)\subset J$, and the result follows.
\end{proof}

\begin{lem}\label{lem:gorenstein subcomplex}
Let $K$ and $K'$ be two flag $\kk$-homology spheres with $K$ satisfying the SCC, and assume that there is a graded ring isomorphism  $\phi:H^*(\ZZ_K)\xr{\simeq} H^*(\ZZ_{K'})$. Suppose $K_I$ (resp. $K_J'$) is a $\kk$-homology sphere of dimension $p$, and $0\neq\xi_I\in\w H^p(K_I)$ (resp. $0\neq\xi_J'\in\w H^p(K_J')$). Then $\phi(\xi_I)\in\w H^{q}(K'_J)$ (resp. $\phi^{-1}(\xi_J')\in\w H^{q}(K_I)$) for some $J\subset\VV(K')$ (resp. $I\subset\VV(K)$) and $q\geqslant p$. Especially, if $p\geqslant\dim K-1$ then $q=p$.
\end{lem}
\begin{proof} For simplicity we only prove the lemma in one direction since the converse can be proved in the same way.
Let $R=H^*(\ZZ_K)$ and $R'=H^*(\ZZ_{K'})$. By Lemma \ref{lem:nil of gorenstein}, $\xi_I\in (R^+)^{*(p+1)}$, so $\phi(\xi_I)\in (R'^+)^{*(p+1)}$.
Hence it follows from Lemma \ref{lem:upper bound of nil} that \[\phi(\xi_I)\in\bigoplus_{J\subset\VV(K')}\w H^{\geqslant p}(K'_J).\]
Now suppose $MF(K_I)=\{\omega_1,\dots,\omega_t\}$ and suppose $\phi(\xi_I)=\sum_{i=1}^k\xi_{J_i}'$, where $0\neq\xi_{J_i}'\in\w H^{q_i}(K_{J_i}')$, $q_i\geqslant p$, $J_i\neq J_j$ for $i\neq j$. Let $\omega_s'=\phi_\MM(\omega_s)$ for $1\leqslant s\leqslant t$.
Since $K_I$ is a $\kk$-homology sphere, $\w\omega_s$ is a factor of $\xi_I$ for $1\leqslant s\leqslant t$.
It follows that $\w\omega_s'$ is a factor of $\xi_{J_i}'$ for each $J_i$, and so $\omega_s'\in MF(K_{J_i}')$ for all $1\leqslant s\leqslant t,\,1\leqslant i\leqslant k$.

Suppose $\phi^{-1}(\xi_{J_1}')=\sum_{j=1}^l\xi_{I_j}$, where $0\neq\xi_{I_j}\in \w H^{p_j}(K_{I_j})$,  $I_i\neq I_j$ for $i\neq j$. Using Lemma \ref{lem:upper bound of nil} again we have $p_j\geqslant p$ for $1\leqslant j\leqslant l$ since $\xi_{J_1}'\in (R'^+)^{*(p+1)}$.
The same reasoning shows that $\w\omega_s$ is a factor of $\xi_{I_j}$ for all $1\leqslant s\leqslant t,\,1\leqslant j\leqslant l$. Hence $MF(K_I)\subset MF(K_{I_j})$.
This implies that $I\subset I_j$ as $K_I=\mathrm{core}\,K_I$. Since $\xi_I$ and $\xi_{I_j}$ have the same cohomological degree in $H^*(\ZZ_K)$, $p+|I|+1=p_j+|I_j|+1$. Combining this with $p_j\geqslant p$ and $I\subset I_j$ we have $p_j=p$ and $|I|=|I_j|$. So $\phi^{-1}(\xi_{J_1}')\in \w H^p(K_I)$. Applying this argument for all $\xi_{J_i}'$ we get $\phi^{-1}(\xi_{J_i}')\in \w H^p(K_I)$ for $1\leqslant i\leqslant k$. Note that $\w H^p(K_I)\cong\kk$. Hence we must have $k=1$ in the formula for $\phi(\xi_I)$, i.e., $\phi(\xi_I)\in\w H^q(K'_J)$ for some $J\subset\VV(K')$ and $q\geqslant p$.

{Since $\dim K=\dim K'$ by Corollary \ref{cor:goren}, the  second statement comes from the fact that $\w H^{n}(K_I)\neq 0$ (resp. $\w H^{n}(K'_J)\neq 0$), where $n=\dim K$,  if and only if $I=\VV(K)$ (resp. $J=\VV(K')$).}
\end{proof}

\begin{rem}
From the proof of Lemma \ref{lem:gorenstein subcomplex} we can see that for the subcomplex $K_I$ in Lemma \ref{lem:gorenstein subcomplex}, the restriction of the map $\phi_\MM$ to $MF(K_I)\subset MF(K)$ is
\[\phi_\MM\mid_{MF(K_I)}:MF(K_I)\to MF(K_J').\]
\end{rem}
Next, we state the following two lemmas whose proofs we defer to Appendix \ref{app:proof of two lemmata}, again because of their complexity.
\begin{lem}\label{lem:1-1}
If $p\geqslant n-2$ ($n=\dim K$) in Lemma \ref{lem:gorenstein subcomplex}, then the map $\phi_\MM\mid_{MF(K_I)}:MF(K_I)\to MF(K_J')$ is a bijection.
\end{lem}

\begin{lem}\label{lem:link}
Under the hypotheses of Theorem \ref{thm:rigid}, let $L$ be the link of a vertex of $K$. Then for a nonzero element $[L]\in\w H^*(L)\cong \kk$, we have $\phi([L])\in\w H^*(L')$, where $L'$ is the link of some vertex of $K'$.
\end{lem}

\begin{rem}\label{rem:link}
	With the same hypotheses as Theorem \ref{thm:rigid}, let $\mathcal L$ (resp. $\mathcal L'$) be the set of different links of vertices in $K$ (resp. $K'$). Then Lemma \ref{lem:link} implies that 
		\[\dim_\kk\bigoplus_{L\in\mathcal L}\w H^*(L)\leqslant \dim_\kk\bigoplus_{L'\in\mathcal L'}\w H^*(L').\]
		Note that $K'$ is a $\kk$-homology sphere with $|\VV(K')|=|\VV(K)|$ by Corollary \ref{cor:goren}, which gives
		\[\dim_\kk\bigoplus_{L'\in\mathcal L'}\w H^*(L')=|\mathcal L'|\leqslant |\VV(K')|=|\VV(K)|.\]
		Since $K$ is not a suspension by Proposition \ref{prop:nsc}, $\mathrm{lk}_K\{i\}=\mathrm{lk}_K\{j\}$ if and only if $i=j$ for any $i,j\in\VV(K)$ since if $\mathrm{lk}_K\{i\}=\mathrm{lk}_K\{j\}$ for $i\neq j$, then clearly $\{i,j\}\not\in K$, and the full subcomplex $\mathrm{st}_K\{i\}\cup \mathrm{st}_K\{j\}=\partial\Delta^1*\mathrm{lk}_K\{i\}$ would be a suspension and a homology sphere of the same dimension as $K$, which must be $K$. Hence, we have 
	\[\dim_\kk\bigoplus_{L\in\mathcal L}\w H^*(L)=|\mathcal L|=|\VV(K)|.\] 
	Together, these imply that $K'$ is not a suspension either, and the isomorphism $\phi:H^*(\ZZ_K)\xr{\simeq} H^*(\ZZ_{K'})$ induces a bijection: $\psi:\VV(K)\to\VV(K')$, $\psi(i)=i'$.
\end{rem}

\begin{cor}\label{cor:link}
In the notation of Lemma \ref{lem:link}, suppose $K_I\subset L$  is an $(n-2)$-dimensional $\kk$-homology sphere {($n=\dim K$)}.
Then for $0\neq\xi_I\in\w H^{n-2}(K_I)$, we have $\phi(\xi_I)\in\w H^{n-2}(K'_J)$ for some $J\subset\VV(L')$.
The converse is also true: If $K_J'\subset L'$ is an $(n-2)$-dimensional $\kk$-homology sphere, then for $0\neq\xi_J'\in\w H^{n-2}(K_J')$, we have $\phi^{-1}(\xi_J')\in\w H^{n-2}(K_I)$ for some $I\subset\VV(L)$.
\end{cor}
\begin{proof}
By Lemma \ref{lem:gorenstein subcomplex} $\phi(\xi_I)\in\w H^{p}(K'_J)$ for some $J\subset[m]$ and $p\geqslant n-2$.
Note that $\xi_I$ is a nontrivial factor of $[L]$, so {$\phi(\xi_I)$} is a nontrivial factor of {$\phi([L])\in\w H^*(L')$}, thus $J\subsetneq \VV(L')$.
It follows that $p\leqslant n-2$ since $L'$ is a $(n-1)$-dimensional $\kk$-homology sphere, and therefore $p=n-2$. The converse can be proved in the same way.
\end{proof}

\begin{proof}[Proof of Theorem \ref{thm:rigid}]
{We first show that $\psi$ in Remark \ref{rem:link} is a simplicial map.}
Since $K$ and $K'$ are both flag, it suffices to verify that $\psi(\{i_1,i_2\})=\{i_1',i_2'\}\in K'$ whenever $\{i_1,i_2\}\in K$.
First note that if $\sigma=\{i_1,i_2\}\in K$, then $\mathrm{lk}_K\sigma $ is an $(n-2)$-dimensional  $\kk$-homology sphere. Let $I=\VV(\mathrm{lk}_K\sigma )$.
Since $K$ is flag,
\[\mathrm{lk}_K\sigma=K_I=\mathrm{lk}_K\{i_1\}\cap\mathrm{lk}_K\{i_2\}.\]
{Take a nonzero element $\xi_I\in\w H^{n-2}(K_I)$.}
From Corollary \ref{cor:link} it follows that $\phi(\xi_I)\in\w H^{n-2}(K'_J)$ for some $K'_J\subset\mathrm{lk}_{K'}\{i_1'\}\cap\mathrm{lk}_{K'}\{i_2'\}.$
{Setting $\omega'=\{i_1',i_2'\}$, we get $K_{\omega'\cup J}'=K'_{\omega'}*K_J'$. Hence, if $\omega'\not\in K'$, then $\w\mu'\cdot\w\omega'\neq0$ for all $\mu'\in MF(K_J')$. Setting $\omega=\phi_\mathcal M^{-1}(\omega')$, this, together with the fact that the map
\[\phi_\MM\mid_{MF(K_I)}:MF(K_I)\to MF(K_J')\]
is a bijection (Lemma \ref{lem:1-1}), shows that $\w\mu\cdot\w\omega\neq0$, for all $\mu\in MF(K_I)$, which implies that $K_{\mu\cup\omega}$ is a $4$-circuit. Hence for any vertex $i\in I$ and any $j\in\omega$, $i\in\mathrm{lk}_K\{j\}$, since $i\in\mu$ for some $\mu\in MF(K_I)$. However this forces $K_{\omega\cup I}$ to be $K_\omega*K_I$ because of the flagness of $K_{\omega\cup I}$}, contradicting Proposition \ref{prop:nsc}. So we get $\{i_1',i_2'\}\in K'$.

Thus $\psi$ induces a simplicial injection $\bar\psi:K\to K'$. 
Now let us finish the proof by showing that $\bar\psi$ is also a surjection, i.e. $\psi^{-1}(\sigma')\in K$ for any $\sigma'\in K'$. It suffices to prove this for $1$-faces since $K$ and $K'$ are both flag. Suppose $\sigma'=\{i_1',i_2'\}\in K'$.
Then  $\mathrm{lk}_{K'}{\sigma'}$ is a $(n-2)$-dimensional $\kk$-homology sphere and $\mathrm{lk}_{K'}{\sigma'}=\mathrm{lk}_{K'}\{i_1'\}\cap\mathrm{lk}_{K'}\{i_2'\}$ is a full subcomplex. Let $J=\VV(\mathrm{lk}_{K'}{\sigma'})$ and {take a nonzero element $\xi_J'\in\w H^{n-2}(K_J')$.} Applying Corollary \ref{cor:link} again we have $\phi^{-1}(\xi_J')\in\w H^{n-2}(K_I)$ for some $K_I\subset\mathrm{lk}_K\{i_1\}\cap\mathrm{lk}_K\{i_2\}$.
It follows that $K_{I\cup\{i_1,i_2\}}=K_I*K_{\{i_1,i_2\}}$. So we must have $\{i_1,i_2\}\in K$, otherwise it would contradict Proposition \ref{prop:nsc}.
The theorem has been proved.
\end{proof}

\section{Cohomological rigidity of topological toric manifolds}
The following theorem is the main result of this section.
\begin{thm}\label{thm:rigidity of toric manifold}
Let $M=M(P_K,\mathit{\Lambda})$ and $M'=M(P_{K'},\mathit{\Lambda}')$ be two topological toric $2n$-manifolds.
Assume that the starshaped sphere $K$ is flag and satisfies the SCC. Then the following two conditions are equivalent:
\begin{enumerate}[(i)]
\item there is a cohomology ring isomorphism $h:H^*(M;\Zz)\xr{\simeq} H^*(M';\Zz)$.
\item $M$ and $M'$ are weakly equivariantly homeomorphic.
\end{enumerate}
\end{thm}

\begin{rem}
For the case $n=3$, the theorem is also obtained by Buchstaber et al. \cite{BEMPP17}. Their proof was carried out on the Koszul algebras of $H^*(M)=\Zz[K]/\JJ$ and $H^*(M')=\Zz[K']/\JJ'$. Here we give another proof by means of the Taylor resolution of $\Zz[K]/\JJ$ and $\Zz[K']/\JJ'$.
\end{rem}

\begin{lem}\label{lem:suspension coloring}
Let $K$ be a flag homology $(n-1)$-sphere on $[m]$, and let $\lambda:[m]\to\Zz^n$ be a characteristic function for $K$. Then for the ideal $\JJ$ defined in Theorem \ref{thm:topo toric}, if there is a linear form $a_ix_i+a_jx_j\in \JJ$, $i\neq j$, $a_i,a_j\in\mathbb{Z}$, then $K$ is a suspension over a homology $(n-2)$-sphere.
\end{lem}
\begin{proof}
Since $\lambda$ is a characteristic function, the assumption in the lemma implies that any facet of $K$ contains $i$ or $j$. We claim that $K=\partial\Delta^1*\mathrm{lk}_K\{i\}$, $\VV(\Delta^1)=\{i,j\}$, and so the lemma follows. First we show that $\{i,j\}\not\in K$. If $\{i,j\}\in K$, taking any facet $\tau$ from $K\setminus\mathrm{st}_K\{i,j\}$ (note that $\mathrm{st}_K\{i,j\}$ is a proper subcomplex of $K$), then $\tau$ contains $i$ or $j$, say $i\in\tau$. Let $\sigma=\tau\setminus\{i\}$. Clearly $\sigma\not\in \mathrm{lk}_K\{i,j\}$ and  $\sigma$ is a face of the sphere $K$ of codimension $1$, and so $\mathrm{lk}_K\sigma=S^0$. Set $v=\VV(\mathrm{lk}_K\sigma)\setminus\{i\}$. Then $\sigma\cup\{v\}$ is a facet of $K$, which means that $v=j$, since any facet of $K$ contains $i$ or $j$. However this would imply that $\sigma\in\mathrm{lk}_K\{i,j\}$ since $K$ is flag, a contradiction.
	
Next we show that $\mathrm{lk}_K\{i\}=\mathrm{lk}_K\{j\}$, then since $\{\{i\},\{j\}\}*\mathrm{lk}_K\{i\}\subset K$ is a full subcomplex by the flagness of $K$, and is a homology sphere of the same dimension as $K$, it follows that $K=\{\{i\},\{j\}\}*\mathrm{lk}_K\{i\}$.
To see this, let $\sigma$ be any facet of $\mathrm{lk}_K\{i\}$ so that $\sigma$ is a face of the sphere $K$ of codimension $1$. As we have seen above, $j=\VV(\mathrm{lk}_K\sigma)\setminus\{i\}$. Therefore $\sigma\in\mathrm{lk}_K\{j\}$, and it follows that $\mathrm{lk}_K\{i\}\subset \mathrm{lk}_K\{j\}$. Similarly, we have $\mathrm{lk}_K\{j\}\subset \mathrm{lk}_K\{i\}$. Hence the desired equality holds.
\end{proof}

\begin{proof}[Proof of Theorem \ref{thm:rigidity of toric manifold}:]
If $\Zz[K]/\JJ=H^*(M)\xr{h} H^*(M')=\Zz[K']/\JJ'$ is an isomorphism, then by Lemma \ref{lem:cps}, $|\VV(K)|=|\VV(K')|$ and
there is an isomorphism of $\mathrm{Tor}$-algebras: $\phi:\mathrm{Tor}^{*,*}(K)\xr{\simeq}\mathrm{Tor}^{*,*}(K')$.
Thus $K\approx K'$ by Theorem \ref{thm:B-rigidity}.

Set $m=|\VV(K)|$, and use the abbreviated notation $\Zz[m]$ for $\Zz[x_1,\dots,x_m]$.
It is known that $\phi$ is induced by a chain map (up to chain homotopic) between the $\Zz[m]/\JJ$-projective resolution of $\Zz[K]/\JJ$ and the $\Zz[m]/\JJ'$-projective resolution of $\Zz[K']/\JJ'$ (see for example \cite[Theorem 6.16]{Rot09}), as shown in the following diagram. Here we use the Taylor resolution, i.e. $\TT/\JJ\to\Zz[K]/\JJ$, where $\TT$ is the chain complex defined in Section \ref{subsec:Tor}. 
\[
\xymatrix{
\cdots\ar[r]^-{f_2}&P_1=\bigoplus\limits_{\omega\in MF(K)}(\Zz[m]/\JJ)_\omega\ar[r]^-{f_1}\ar@{.>}[d]^{\phi_1}&P_0=\Zz[m]/\JJ\ar[r]^-{f_0}\ar@{.>}[d]^{\phi_0}&\Zz[K]/\JJ\ar[d]^{h}\\
\cdots\ar[r]^-{f_2'}& P_1'=\bigoplus\limits_{\omega'\in MF(K')}(\Zz[m]/\JJ')_{\omega'}\ar[r]^-{f_1'}& P_0'=\Zz[m]/\JJ'\ar[r]^-{f_0'}&\Zz[K']/\JJ'}
\]
in which $P_1$ (resp. $P_1'$) is a free $\Zz[m]/\JJ$- module (resp. $\Zz[m]/\JJ'$- module) with basis $\{e_\omega:\omega\in MF(K)\}$
(resp. $\{e_{\omega'}:\omega'\in MF(K')\}$). That $\TT/\JJ\to \Zz[K]/\JJ$ is a $\Zz[m]/\JJ$-free resolution of $\Zz[K]/\JJ$ comes from the fact that $\TT\to\Zz[K]$ is a $\Zz[m]$-free resolution of $\Zz[K]$ and the following equation: 
	\[\begin{split}
		\mathrm{Tor}_{\Zz[m]}(\Zz[K],\Zz[m]/\JJ)&=\mathrm{Tor}_{\Zz[m]/\JJ}(\mathrm{Tor}_{\Gamma}(\Zz[K],\Zz),\Zz[m]/\JJ)\\
		&=\mathrm{Tor}_{\Zz[m]/\JJ}(\Zz[K]/\JJ,\Zz[m]/\JJ)\\
		&=\mathrm{Tor}_{\Zz[m]}^0(\Zz[K],\Zz[m]/\JJ)=\Zz[K]/\JJ,
	\end{split}\]
	where $\Gamma$ is the subring of $\Zz[m]$ generated by the linear forms in $\JJ$. The second equality uses the fact that $\Zz[K]$ is a free $\Gamma$-module and $\Zz[K]\otimes_\Gamma\Zz=\Zz[K]/\JJ$, and the first equality follows from \cite[Theorem XVI.6.1]{CE56}.
Using formula \eqref{eq:differential} and applying Proposition \ref{prop:scc} \eqref{st:a} to all fields of finite characteristics, we see that $\phi_1(e_\omega)=\pm e_{\omega'}$ for some $\omega'\in MF(K')$. Note that $e_\omega$ and $e_{\omega'}$ correspond to the generators of $\w H^0(K_\omega)$ and  $\w H^0(K'_{\omega'})$ respectively.

It is clear that $f_1(e_\omega)=\xx_\omega:=\prod_{i\in\omega}\bar x_i$, where $\bar x_i$ is the coset of $x_i$ in $\Zz[m]/\JJ$. Since the diagram is commutative, we have
\[\prod_{i\in\omega}\phi_0(\bar x_i)=\phi_0(\xx_\omega)=\phi_0f_1(e_\omega)=f_1'\phi_1(e_\omega)=\pm\xx_{\omega'}=\pm\prod_{j\in\omega'}\bar x_j.\]
Thus from the fact that $\Zz[m]/\JJ\cong\Zz[m]/\JJ'\cong\Zz[m-n]$ is a $UFD$, we deduce that if $i\in\omega\in MF(K)$, then $\phi_0(\bar x_i)=\pm \bar x_{j}$ for some $j\in\omega'\in MF(K')$.
Furthermore, since $K=\mathrm{core}\,K$, any vertex of $K$ belongs to a missing face.
Combining these facts together, we see that for each $i\in[m]$, $h(\bar x_i)=\pm\bar x_j$ for some $j\in[m]$.

Next we will show that $h$ induces a bijection between the vertex set $\VV(K)$ and $\VV(K')$, and this map is actually a simplicial isomorphism.
That is, define $\bar h:[m]\to[m]$ to be such that $h(\bar x_i)=\pm\bar x_{\bar h(i)}$.
We first verify that $\bar h$ is well defined, i.e. if $i\neq j$ then $\bar x_i\neq\pm \bar x_j$ in both $\Zz[K]/\JJ$ and $\Zz[K']/\JJ'$.
Indeed, otherwise $x_i\pm x_j\in \JJ$ or $\JJ'$. However this implies that $K$ or $K'$ is a suspension by Lemma \ref{lem:suspension coloring},
contradicting Proposition \ref{prop:nsc}.  So $\bar h$ is well defined and it is clearly a bijection.
To see that $\bar h$ is a simplicial map, {recall from the discussion following Theorem \ref{thm:topo toric} that} if $\sigma\in K$, then $0\neq\xx_\sigma=\prod_{i\in\sigma}\bar x_i\in\kk[m]/\JJ$.
So $0\neq h(\xx_\sigma)=\pm \xx_{\bar h(\sigma)}$, which means that $\bar h(\sigma)\in K'$. Thus $\bar h$ is a simplicial injection, and so it is a bijection since $K\approx K'$.

Reordering $\VV(K')$ if necessary we may assume that $\bar h:[m]\to[m]$ is the identity map.
Now let $N$ and $N'\subset\Zz^m$ be the sublattices generated by the row vectors of the characteristic matrices $\mathit{\Lambda}$ and $\mathit{\Lambda}'$ respectively.
Then we have $(\Zz[K]/\JJ)_2=\Zz^m/N$ and $(\Zz[K']/\JJ')_2=\Zz^m/N'$. So $h$ induces a lattice isomorphism $h:\Zz^m/N\xr{\simeq} \Zz^m/N'$. 
By changing the order of vertices of $K$ if necessary, we may assume that $[n]\in K$. Then, by multiplying the matrices $\mathit{\Lambda}$ and $\mathit{\Lambda}'$ from the left
by appropriate matrices from  $GL(n,\mathbb{Z})$ we may assume that
\[\mathit{\Lambda}=
\begin{pmatrix}
	1&\ldots&0&\lambda_{1,n+1}&\ldots&\lambda_{1m}\\
	\vdots&\ddots&\vdots&\vdots&\ddots&\vdots\\
	0&\ldots&1&\lambda_{n,n+1}&\ldots&\lambda_{nm}
\end{pmatrix},\quad 
\mathit{\Lambda'}=
\begin{pmatrix}
	1&\ldots&0&\lambda_{1,n+1}'&\ldots&\lambda_{1m}'\\
	\vdots&\ddots&\vdots&\vdots&\ddots&\vdots\\
	0&\ldots&1&\lambda_{n,n+1}'&\ldots&\lambda_{nm}'
\end{pmatrix}
\]
The entries $\lambda_{ij}$, $n+1\leqslant j\leqslant m$, are the coefficients in the expression of $-\bar x_i$, $1\leqslant i\leqslant n$, via the basis $\bar x_{n+1},\dots,\bar x_m$ of $\Zz^m/N$. The same is ture for the $\lambda_{ij}'$. Since $h(\bar x_i)=\pm\bar x_i$, we have $\lambda_{ij}=\pm\lambda_{ij}'$. These together are equivalent to saying that there is a matrix $A\in GL(n,\mathbb{Z})$ and a diagonal $(m\times m)$-matrix $B$ with $\pm 1$ on the diagonal such that $\mathit{\Lambda}=A\cdot\mathit{\Lambda}'\cdot B$. Hence the theorem follows by Proposition \ref{prop:equiv hoeom}.
\end{proof}

We end this section by a special version of Theorem \ref{thm:rigidity of toric manifold} for toric varieties.
\begin{thm}\label{thm:variety}
Let $X$ and $X'$ be two compact smooth toric varieties. Assume the underlying simplicial sphere of the simplcial fan corresponding to $X$ is flag and satisfies the SCC. Then  $X$ and $X'$ are isomorphic as varieties if and only if their cohomology rings are isomorphic.
\end{thm}
\begin{proof}
The necessity part is straightforward. For the sufficiency part, we use the result of Masuda \cite[Theorem 1.1 and Corollary 1.2]{M08}.
That is, two compact smooth toric varieties $X$ and $X'$ are isomorphic as varieties if and only if their equivariant cohomology algebras are weakly isomorphic, i.e.,
there is a ring isomorphism $\Phi :H^*_T(X)\xr{\simeq} H^*_T(X')$ together with an automorphism $\gamma$ of $T=(\mathbb{C}^*)^n$ such that
$\Phi(u\alpha)=\gamma^*(u)\Phi(\alpha)$ for any $u\in H^*(BT)$ and $\alpha\in H^*_T(X)$, where $\gamma^*$ denotes the automorphism of
$H^*(BT)$ induced by $\gamma$.

If $H^*(X;\mathbb{Z})\cong H^*(X';\mathbb{Z})$, then $X$ and $X'$ are weakly equivariantly {homeomorphic in the sense of Subsection \ref{subsec:toric manifold} by Theorem \ref{thm:rigidity of toric manifold}. Since the action of the compact torus $(S^1)^n\subset (\mathbb{C}^*)^n$ is locally standard on $X$ and $X'$, a weakly equivariant homeomorphism $X\cong X'$ induces a weak isomorphism $H^*_T(X)\xr{\simeq} H^*_T(X')$}, and the theorem follows.
\end{proof}

\section{The case of simplicial $2$-spheres}\label{sec:n=3}
In this section we further discuss the $B$-rigidity problem of simplicial $2$-spheres, and the cohomological rigidity problem of topological toric manifolds of dimension $6$.
Note that in dimension $2$ any simplicial sphere is polytopal. So any topological toric $6$-manifold is also a quasitoric $6$-manifold.
Unless otherwise stated, $K$ denotes a simplicial $2$-sphere and $P$ denotes a simple $3$-polytope in this section.
\subsection{On the $B$-rigidity of $2$-spheres}\label{subsec:the class Q}
First let us consider the case that $K$ is flag.
In this case suppose $L=K_I\subset K$ is a $4$-circuit full subcomplex, and suppose $MF(L)=\{\omega_1,\omega_2\}$.
Then there are three puzzle-moves $(K,L,\phi_1)$, $(K,L,\phi_2)$ and $(K,L,\phi_2\circ\phi_1)$ such that $\phi_i$ ($i=1,2$) exchanges the two vertices of $\omega_i$ and fixes the vertices of the other missing face. From the discussion of puzzle-moves preceding Theorem \ref{thm:B-rigidity}, we know that if $K$ is flag, then any nontrivial puzzle-move is of this form. However if $L$ is the link of a vertex in $K$, then these puzzle-moves are all trivial. In this case, we refer to $L$ as a \emph{simple $4$-circuit} (or \emph{simple $\square$}).

Let $\QQ$ be the class consisting of simplicial $2$-spheres which are flag and satisfy the condition that any $4$-circuit full subcomplex is simple, or dually, simple $3$-polytopes which are flag and satisfy the condition that any $4$-belt  bounds a $4$-gonal facet.
(For $k\geqslant4$, a \emph{$k$-belt} in a simple $3$-polytope is a cyclic sequence $\{F_{i_1},\dots,F_{i_k}\}$
of facets in which two facets have a nonempty intersection if
and only if they follow each other (including $F_{i_k}$, $F_{i_1}$). For a \emph{$3$-belt} we assume additionally that $F_{i_1}\cap F_{i_2}\cap F_{i_3}=\emptyset$.)

In \cite{Bar74}, $\QQ$ is known as the class of simple $3$-polytopes with \emph{strongly cyclically $4$-connected graphs}, and it is also called the \emph{almost Pogorelov class} in \cite{Ero19}. (Remark: as shown in \cite[Theorem 10.3.1]{DO01} and \cite{Ero19}, there is a one-to-one correspondence between almost Pogorelov polytopes different from the cube and the pentagonal prism and right-angled polyhedra of finite volume in hyperbolic space $\mathbb{H}^3$. This is basically a consequence of  Andreev's result in \cite{And70}.) In the recent paper \cite{Ero20}, Erokhovets proved that the property of a simplicial $2$-sphere $K$ to belong to $\QQ$ is determined by the cohomology ring $H^*(\ZZ_K)$.

We have seen that if $K\in\QQ$, there are no nontrivial puzzle-moves on $K$. So in the spirit of {Conjecture \ref{conj:puzzle} and Remark \ref{rem:n=3}}, we may ask the following question.

\begin{prob}
If $K\in\QQ$, is $K$ $B$-rigid?
\end{prob}
We will show in \cite{FMW16} that the answer to this question is yes. For the special case in which $K\in \QQ$ and $\VV(K)\leqslant 11$, this was also obtained by Erokhovets \cite{Ero20}. Also for any $K\in \QQ$ such that 
	there are no adjacent vertices of valence $4$ (the \emph{valence} of a vertex  is the number of vertices in its link) and any facet of $K$ contains a vertex of valence $4$, i.e. the boundary spheres of polytopes dual to \emph{ideal almost Pogorelov polytopes} (see Example \ref{exmp:ideal} (ii) below), this was proved by Erokhovets in \cite{Ero20b}. Now let us look at some examples of spheres in $\QQ$.

\begin{exmp}\label{exmp:1}
Let $K$ be an arbitrary simplicial $2$-sphere, and $K'$ be the barycentric subdivision of $K$. Then $K'\in\QQ$. The flagness of $K'$ is obvious.
Set
\[\VV(K')=\{v_\sigma:\sigma\in K,\sigma\neq\emptyset\}.\]
Note that any $\{v_\sigma,v_\tau\}$ with $|\sigma|=|\tau|$ is a missing face of $K'$.
Let $K_I'$ be a $4$-circuit, and suppose $I=\{v_{\sigma_i}\}_{i=1}^4$.
Since $K_I'$ has exactly two missing faces, there are no three vertices in $I$ such that their corresponding simplices in $K$ have the same dimension.
So all the possible dimensions of $\sigma_i,\,1\leqslant i\leqslant4$, up to a permutation of the index $i$, are
\[(0,0,1,1),\,(0,0,1,2),\,(0,0,2,2),\,(0,1,1,2),\,(0,1,2,2),\,(1,1,2,2).\]
Furthermore, if $\{v_\sigma,v_\tau\}\in K'$ with $|\sigma|<|\tau|$, then $\sigma$ is a proper face of $\tau$ in $K$. Hence the cases $(0,0,1,1)$ and $(1,1,2,2)$ are  impossible because two different faces of the same dimension have at most one common face of codimension one. If $(0,0,1,2)$ is the case, then the two vertices $\sigma_1$ and $\sigma_2$ belong to the edge $\sigma_3$ and the $2$-face $\sigma_4$, which implies that $\{v_{\sigma_3},v_{\sigma_4}\}\in K'$, a contradiciton. Similarly, if  $(0,1,1,2)$ or $(0,1,2,2)$ is the case, then $\sigma_1\subset\sigma_2,\sigma_3\subset\sigma_4$ or $\sigma_1\subset\sigma_2\subset\sigma_3,\sigma_4$ gives a contradiction.
So $(0,0,2,2)$ is the only possible case.
In this case,  $\sigma_1,\sigma_2\subset\sigma_3,\sigma_4$, which means that $\tau=\sigma_1\cup\sigma_2$ is the common edge of $\sigma_3$ and $\sigma_4$.
Hence we have $K'_I=\mathrm{lk}_{K'}\{v_{\tau}\}$.
\end{exmp}

Recall some definitions from \cite{CK11}: Let $\Gamma$ be the boundary sphere of a $3$-polytope which is not necessarily simplicial or simple.
Suppose $\{F_1,\dots,F_m\}$ is the set of facets of $\Gamma$, and suppose $\{e_1,\dots,e_n\}$ is the set of edges of $\Gamma$.
Define $\xi_1(\Gamma)$ to be the simplicial $2$-sphere obtained from $\Gamma$ by replacing every $F_i$ by the cone over its boundary. Obviously, $e_i$ is also an edge of $\xi_1(\Gamma)$ for $1\leqslant i\leqslant n$. Define $\xi_2(\Gamma)$ to be the simplicial $2$-sphere obtained from $\xi_1(\Gamma)$ by doing stellar subdivision operations at every $e_i$ (see Figures 1 and 2 in \cite{CK11}). {In other words, $\xi_2(\Gamma)$ is the barycentric subdivision of $\Gamma$.}
\begin{exmp}\label{exmp:ideal}
(i) Let $P$ be a simple $3$-polytope without triangular faces, $\Gamma=\partial P$. 
Then $\xi_1(\Gamma)\in\QQ$. 
To see this, let $v_{F_i}$ denote the coning vertex corresponding to a facet $F_i$ in $\Gamma$.
Note that $\{v_{F_i},v_{F_j}\}\in MF(\xi_1(\Gamma))$ for any two facets $F_i,F_j$. Moreover, when $\{v_{F_i},j\},\{v_{F_i},k\},\{j,k\}\in\xi_1(\Gamma)$ for some facet $F_i$ and $j,k\in\VV(\Gamma)$, the fact that $P$ is a polytope implies that $\{j,k\}$ is an edge of $F_i$, so that $\{v_{F_i},j,k\}\in \xi_1(\Gamma)$.  Hence, if there is $I\in MF(\xi_1(\Gamma))$ with $|I|=3$, then  $I\subset\VV(\Gamma)$. However, since $P$ is simple, $\xi_1(\Gamma)_I$ has to be the boundary of a triangular face of $\Gamma$ in this case, but this contradicts the assumption on $P$. So $\xi_1(\Gamma)$ is flag.
Now, suppose that $\xi_1(\Gamma)_I$ is a $4$-circuit. {If $v_{F_i}\in I$ for some facet $F_i$, then either $I\setminus\{v_{F_i}\}\subset \VV(\Gamma)$ or there is some other $v_{F_j}\in I$ such that $I\setminus\{v_{F_i},v_{F_j}\}\subset \VV(\Gamma)$. In the first case, the three vertices in $I\setminus\{v_{F_i}\}$ must belong to a facet  $F_k\neq F_i$, and two of them belong to $F_i$ and do not form an edge of $\Gamma$,  because $P$ is simple and $\xi_1(\Gamma)_I$ is a $4$-circuit. But this is impossible since $P$ is a polytope. In the second case, $I\setminus\{v_{F_i},v_{F_j}\}\subset\VV(F_i)\cap\VV(F_j)$ and  $I\setminus\{v_{F_i},v_{F_j}\}\in MF(\xi_1(\Gamma))$. But this is impossible too, since in a polytope, $|\VV(F_i)\cap\VV(F_j)|=2$ if and only if $\VV(F_i)\cap\VV(F_j)$ forms a common edge of $F_i$ and $F_j$. Hence we have $I\subset\VV(\Gamma)$, and then the assumption that $P$ is simple implies that any three vertices in $I$ must belong to a facet of $P$. Moreover, since $P$ is not a tetrahedron, all vertices in $I$ actually belong to a common facet. But this occurs only if $\xi_1(\Gamma)_I$ is the boundary of a quadrangular face of $\Gamma$}, which means that $\xi_1(\Gamma)_I$ is the link of the center of this quadrangular face.

(ii) $\xi_2(\partial P)\in\QQ$ for any $3$-polytope $P$.  This can be verified in the same way as in Example \ref{exmp:1}. The dual simple polytopes of this type are exactly the \emph{ideal almost Pogorelov polytopes} defined in \cite{Ero20}, where it is shown that whether {a simplicial $2$-sphere} $K$ is of this type is determined by the cohomology ring $H^*(\ZZ_K)$. 
\end{exmp}

Now we move on to the case of non-flag simplicial $2$-spheres.
Note that a simplicial $2$-sphere $K$ is non-flag if and only if $K$ is reducible or $K=\partial\Delta^3$.
Ignoring the trivial case $K=\partial\Delta^3$, the following result provides a necessary condition for a reducible simplcial sphere to be $B$-rigid.
\begin{prop}[{\cite[Theorem 4.2]{FW21}}]
Let $K_1$ and $K_2$ be two polytopal spheres of the same dimension, and suppose $K\in\CC(K_1\#K_2)$.
If $K$ is $B$-rigid, then $K$ is the only element in $\CC(K_1\#K_2)$.
\end{prop}

This proposition can also be deduced from Bosio's result \cite[Theorem 3.5]{B17} (Theorem \ref{thm:puzzle}), since all elements in $\CC(K_1\#K_2)$ are actually puzzle-equivalent.
Combining this proposition with the conclusion of {\cite[Proposition 2.2, Lemma 2.3 and Proposition 2.4]{CK11}}, we immediately get the following theorem,
which is an analog of \cite[Theorem 1.2]{CK11}.

\begin{thm}\label{thm:sum}
Let $K$ be a reducible simplicial $2$-sphere {(i.e. $K$ can be decomposed as the connected sum of two simplicial $2$-spheres)}. If $K$ is $B$-rigid, then $K=T_4\#T_4\#T_4$ or $K=K_1\#K_2$, where
\begin{align*}
&K_1\in\{T_4,O_6,I_{12}\},\\
&K_2\in\{T_4,O_6,I_{12}, \xi_1(C_8), \xi_2(C_8), \xi_1(D_{20}), \xi_2(D_{20})\}\cup\{B_n : n\geqslant7\}.
\end{align*}
\end{thm}

The notations in the theorem are adopted from \cite{CK11}, in which $T_4$, $C_8$, $O_6$, $D_{20}$ and $I_{12}$ are the boundary spheres of the five Platonic solids: the tetrahedron, the cube, the octahedron, the dodecahedron and the icosahedron, respectively, subscripts
indicating the vertex number; $B_n$ is the suspension of the boundary of an $(n-2)$-gon, called a \emph{bipyramid} (see Figure 3 in \cite{CK11}).

\begin{rem}
We will shown in \cite{FMW16} that the converse of Theorem \ref{thm:sum} is also true, i.e. all the reducible simplicial spheres $K$ appearing in Theorem \ref{thm:sum} are actually {$B$-rigid}, which classifies all reducible {$B$-rigid} simplicial $2$-spheres.
\end{rem}

\subsection{On the cohomological rigidity of quasitoric $6$-manifolds}
Let $\PP$ be the \emph{Pogorelov class}, i.e. the class of simple flag $3$-polytopes without $4$-belts, or dually, simplicial $2$-spheres which are flag and satisfy the no $\square$-condition.
By the Four Color Theorem we know that there is at least one quasitoric manifold over any simple $3$-polytope.
So Theorem \ref{thm:rigidity of toric manifold} applies efficiently to $6$-dimensional quasitoric manifolds over polytopes from $\PP$.
However, the applicability of Theorem \ref{thm:variety} for (complex) $3$-dimensional toric varieties is more limited since a simplicial $2$-sphere from $\PP$ usually can not be obtained from a nonsingular fan. (A few examples of simplicial spheres from $\PP$, which are obtained from nonsingular fans, were produced by Suyama \cite{Suy15}.)
Especially, any $P\in\PP$ has no Delzant realizations in $\Rr^3$ (full dimensional lattice simple polytopes whose normal fans are nonsingular are called \emph{Delzant}), so that there are no smooth projective toric varieties over any $P\in\PP$. This is because a Delzant $3$-polytope must have at least one triangular or quadrangular face by the result of Delaunay \cite{D05}.

One might expect that Theorem \ref{thm:rigidity of toric manifold} is also true for quasitoric $6$-manifolds over simple polytopes from $\QQ$, but this is not the case in general. A simplest counterexample over the simple polytope dual to $B_6$ or $B_7$  can be constructed as follows. (Note that $B_6$ and $B_7$ are the only two bipyramids that belong to $\QQ$.)

\begin{exmp}\label{exmp:counterexmp}
Consider the family of quasitoric manifolds over a cube (dual to $B_6$) with characteristic matrices
\[{\mathit{\Lambda}_k=}\bordermatrix{
~&F_1&F_2&F_3&F_1'&F_2'&F_3'\cr
~&1&0&0&-1&0&0\cr
~&0&1&0&0&-1&k\cr
~&0&0&1&0&0&-1},\quad {k\in\mathbb{Z}},
\]
where $\{F_i,F_i'\}$ is a pair of opposite faces for $1\leqslant i\leqslant 3$. For each $k\in\mathbb{Z}$, the corresponding quasitoric manifold is $S^2\times H_k$, where $H_k$ is the \emph{Hirzebruch surface} $\mathbb{C}P(\OO(k)\oplus\mathbb{C})$, i.e. the associated projective bundle of the sum of a
trivial line bundle $\Cc$ and {the line bundle  $\OO(k)$ with the first Chern class $k$ over $\Cc P^1$ (see \cite[Example 5.1.8]{BP15} or \cite[Example 7.3.4]{CLS11})}. The homeomorphism classes of {$H_k$} depend only on the parity of $k$ {(cf. \cite[\S 5.6]{OR70}):}
\[\mathbb{C}P(\OO(k)\oplus\mathbb{C})=
\begin{cases}
S^2\times S^2\quad &\text{ if $k$ is even,}\\
\mathbb{C}P^2\# \overline{\mathbb{C}P^2} &\text{ if $k$ is odd,}
\end{cases}
\]
where $\#$ denotes the connected sum, and $\overline{\mathbb{C}P^2}$ is $\mathbb{C}P^2$ with the reverse orientation, while they are equivariantly distinct for different $k$. {Consequently, quasitoric manifolds over a cube with characteristic matrices $\mathit{\Lambda}_k$ are homeomorphic to $S^2\times S^2\times S^2$ or $S^2\times(\mathbb{C}P^2\# \overline{\mathbb{C}P^2})$, while their weakly equivariant homeomorphism classes are infinite.}
	
The case for quasitoric manifolds over the product of a $5$-gon and an interval, which is dual to $B_7$, is similar, by the classification result  \cite[p. 553]{OR70} for quasitoric manifolds over polygons.
\end{exmp}

Even if we exclude the cases of $K=B_6,B_7$, Theorem \ref{thm:rigidity of toric manifold} may still not hold for $K\in\QQ$, which can be seen as follows.
Let $\mathit{\Lambda}$ be the characteristic matrix of a quasitoric manifold $M$ over a simple polytope $P\in\QQ$ with facet set $\FF=\{F_1,\dots,F_m\}$.
Suppose $F_1$ is a quadrangular face, and suppose the column vectors of $\mathit{\Lambda}$ for $F_1$ and the four surrounding faces are as in Figure \ref{fig:coloring}. Let $M'$ be another quasitoric manifold over the same $P$ with characteristic matrix $\mathit{\Lambda}'$ such that $\mathit{\Lambda}'(F_1)=(0,p',1)$ and $\mathit{\Lambda}'(F_i)=\mathit{\Lambda}(F_i)$ for $i\neq1$. We claim that $M$ and $M'$ are homeomorphic if  $p\equiv p'$ mod $2$.
\input{coloring.TpX}

First let us recall the conception of an \emph{adjacent connected sum} of manifolds (cf. \cite{M79}). Let $Y$ and $Z$ be closed manifolds of dimension $m$, $X$ be a closed manifold of dimension $n<m$ and suppose that $X\times D^{m-n}$ is embedded in both $Y$ and $Z$, i.e. $X$ is a submanifold with a trivial tubular neighborhood. Remove from $Y$ and $Z$ the interiors of the embedded copies of $X\times D^{m-n}$, obtaining
$\bar Y$ and $\bar Z$, each with boundary homeomorphic to $X\times S^{m-n-1}$. Fix a homeomorphism 
\[f:X\times S^{m-n-1}\to X\times S^{m-n-1}.\]
The adjacent connected sum of $Y$ and $Z$ along $X$ is defined to be
\[Y\#_{X,f}Z:=\bar Y\cup_f\bar Z.\]
In particular, if $f$ is the identity map, we simply denote it by $Y\#_X Z$.

In our case, we consider a hyperplane $H\subset \Rr^3$, whose associated half-spaces will be denoted $H_+$ and $H_-$, intersecting $P$ but not at any vertex of $P$, such that $\VV(P\cap H_+)\cap\VV(P)=\VV(F_1)$. Namely, $H$ is a hyperplane cutting off $F_1$. It is easy to see that $P_+:=P\cap H_+$ is a cube and $P_-:=P\cap H_-\approx P$.
Denote $F_0^+$ and $F_0^-$ the new facets of $P_+$ and $P_-$, respectively, i.e. the facets corresponding to $P\cap H$.
Define two characteristic matrices $\mathit{\Lambda}_+$ and $\mathit{\Lambda}_+'$ on $P_+$ (resp. $\mathit{\Lambda}_-$  and $\mathit{\Lambda}_-'$ on $P_-$) by
\begin{gather*}
\mathit{\Lambda}_+(F_0^+)=\mathit{\Lambda}_+'(F_0^+)={(0,0,-1)},\text{ and }\\
\mathit{\Lambda}_+(F_i\cap H_+)=\mathit{\Lambda}(F_i),\  \mathit{\Lambda}_+'(F_i\cap H_+)=\mathit{\Lambda'}(F_i)\ \text{ for } 1\leqslant i\leqslant 5;\\
\text{resp.}\ \mathit{\Lambda}_-(F_0^-)=\mathit{\Lambda}_-'(F_0^-)={(0,0,-1)},\text{ and }\\
\mathit{\Lambda}_-(F_i\cap H_-)=\mathit{\Lambda}_-'(F_i\cap H_-)=\mathit{\Lambda}(F_i)\ \text{ for } 2\leqslant i\leqslant m.
\end{gather*}

Let $Y,Z$ be the quasitoric manifolds defined by the characteristic pair $(P_+,\mathit{\Lambda}_+)$ and $(P_-,\mathit{\Lambda}_-)$, respectively, and $\pi_+:Y\to P_+$, $\pi_-:Z\to P_-$ be the projections.
It is easy to see that the characteristic submanifolds $\pi_+^{-1}(F_0^+)=\pi_-^{-1}(F_0^-)=S^2\times S^2$, and their tubular neighborhoods in $Y$ and $Z$ are both $S^2\times S^2\times D^2$. Then by formula \eqref{eq:quasitoric} we have $M=Y\#_{S^2\times S^2}Z$.
Similarly let $Y',Z'$ be the quasitoric manifolds defined by $(P_+,\mathit{\Lambda}_+')$ and $(P_-,\mathit{\Lambda}_-')$, and $\pi_+':Y'\to P_+$, $\pi_-':Z'\to P_-$ be the projections.
Then $M'=Y'\#_{S^2\times S^2}Z'$. Since $\mathit{\Lambda}_-=\mathit{\Lambda}_-'$, we have $Z=Z'$.
{Example \ref{exmp:counterexmp} says that $Y\cong Y'$ if and only if $p\equiv p'$ mod $2$. 
Hence, if there is a homeomorphism
	$\varphi:Y\to Y'$ for $p\equiv p'$ mod $2$ and a commuatative diagram:
\[
\xymatrix{S^2\times S^2\times D^2\ar[d]^-{id}\ar[r]^-{i_0}&Y=S^2\times H_p\ar[d]^-{\varphi}\\
S^2\times S^2\times D^2\ar[r]^-{i_0'}& Y'=S^2\times H_{p'}}
\]
where $i_0$ and $i_0'$ are the embeddings of the tubular neighborhoods of $\pi_+^{-1}(F_0^+)$ and $\pi_+'^{-1}(F_0^+)$ respectively, then $M\cong M'$ follows immedietely. This commutative diagram does exist. To see this, note that a Hirzebruch surface $H_p$ is topologically obtained by gluing two copies of $S^2\times D^2$ along their boundary via an automorphism $f_p$ of $S^2\times S^1$, where $f_p(y,x)=(x^p\cdot y,x^{-1})$ by seeing $S^2$ as $\mathbb{C}\cup\{\infty\}$ and $S^1$ as the unit sphere in $\mathbb{C}$ (see, for example, \cite[Example 3.1.16]{CLS11}, where one easily sees that $H_p=(S^2\times D^2)_1\cup_{f_p}(S^2\times D^2)_2$ for $(S^2\times D^2)_1\subset U_{\sigma_1}\cup U_{\sigma_2}\cong S^2\times\Cc$ and $(S^2\times D^2)_2\subset U_{\sigma_3}\cup U_{\sigma_4}\cong S^2\times\Cc$). Since $f_p$ is the composition $h\circ g^p$, where $g$ and $h$ are automorphisms of $S^2\times S^1$ defined by $g(y,x)=(xy,x)$ and $h(y,x)=(y,x^{-1})$, it follows from \cite[Theorem 5.1]{Glu62} that $f_p$ is isotopic to $f_{p'}$ if and only if $p\equiv p'$ mod $2$. Now we construct the map $\varphi$ as follows. The above analysis shows that a homeomorphism between $H_p$ and $H_{p'}$ ($p\equiv p'$ mod $2$) can be decomposed as
\[
\begin{split}
H_p=&(S^2\times D^2)_1\cup_{id_0}S^2\times S^1\times I\cup_{f_p}(S^2\times D^2)_2\xr{\psi:=id\cup\Psi\cup id}\\
&(S^2\times D^2)_1\cup_{f_{p'}}S^2\times S^1\times I\cup_{id_1}(S^2\times D^2)_2=H_{p'},
\end{split}\]
where $I=[0,1]$ and the gluing maps along boundaries are defined by 
\begin{gather*}
id_0=id:\partial(S^2\times D^2)_1\to S^2\times S^1\times \{0\},\\
f_p:S^2\times S^1\times \{1\}\to\partial(S^2\times D^2)_2,\\
id_1=id:\partial(S^2\times D^2)_2\to S^2\times S^1\times \{1\},\\
f_{p'}:\partial(S^2\times D^2)_1\to S^2\times S^1\times \{0\},
\end{gather*}
and $\Psi$ is an isotopy from $f_{p'}$ to $f_p$. Let $\varphi=id\times\psi:S^2\times H_p\to S^2\times H_{p'}$, and let $i_0$, $i_0'$ be the embeddings $S^2\times S^2\times D^2=S^2\times (S^2\times D^2)_1\ha H_p,H_{p'}$ respectively. Then we get the desired commutative diagram.

On the other hand, $M$ is not weakly equivariantly homeomorphic to $M'$ in general if $p\neq p'$. For instance, if $F_1$ is the only quadrangular face of $P$, and $p'=0$, $p>2\max\{|\lambda_{ij}|:(i,j)\neq (2,1)\}$, then for any automorphism $\psi$ of $P$, the characteristic pairs $(P,\mathit{\Lambda})$ and $(\psi(P),\mathit{\Lambda'})$ defined above are not equivalent. Indeed, we have $\psi(F_1)=F_1$. Furthermore, one easily sees from the local combinatorial structure of $F_1$ that $\{\psi(F_2),\psi(F_4)\}=\{F_2, F_4\}$ and $\{\psi(F_3),\psi(F_5)\}=\{F_3, F_5\}$, or $\{\psi(F_2),\psi(F_4)\}=\{F_3, F_5\}$ and $\{\psi(F_3),\psi(F_5)\}=\{F_2, F_4\}$. In the first case, restricting to the $3\times 3$ matrices $\mathit{\Lambda}(F_2,F_3,F_1)$ and $\mathit{\Lambda}'(\psi(F_2),\psi(F_3),F_1)$, we see that in the equation $\mathit{\Lambda}=A\cdot\mathit{\Lambda}'\cdot B$, 
\[A=\begin{pmatrix}
	\pm 1&0&0\\
	0&\pm 1&\pm p\\
	0&0&\pm 1
\end{pmatrix}.\]
Let $F_6$ be the facet different from $F_1$ such that $F_2\cap F_3\cap F_6\neq\emptyset$. Then $\lambda_{36}=\pm1$ since $\det(\Lambda(F_2,F_3,F_6))=\pm1$. 
Similarly, $\lambda_{3\psi(6)}'=\pm1$. Hence $\lambda_{26}=\pm \lambda_{2\psi(6)}'\pm p$, and so $p\leqslant2\max\{|\lambda_{26}|,|\lambda_{2\psi(6)}'|\}$, which is a contradiction because $\lambda_{2\psi(6)}'=\lambda_{2j}$ for some $j\neq 1,\dots,5$ by the construction of $\mathit{\Lambda}'$. In the second case, the matrix $A$ should be \[A=\begin{pmatrix}
		 0&\pm 1&0\\
		\pm 1&0&\pm p\\
		0&0&\pm 1
	\end{pmatrix},\]
and we should have $\lambda_{26}=\pm \lambda_{1\psi(6)}'\pm p$, which also leads to a contradicition by similar reasoning.

\begin{rem}
Although Theorem \ref{thm:rigidity of toric manifold} does not hold for all quasitoric $6$-manifolds over polytopes from $\QQ$, it can hold if the characteristic matrix $\mathit{\Lambda}$ satisfies some mild conditions (see \cite{FMW16} for details). On the other hand, we may still expect that the cohomological rigidity holds for all quasitoric $6$-manifolds over polytopes from $\QQ$, and indeed this is the case, as we will show in \cite{FMW16}.
\end{rem}

\appendix
\section{Cohomology of real moment-angle complexes and the second Baskakov-Hochster ring}\label{appdx:BHR2}
In this appendix, first we briefly review the work of Cai \cite{C17} on describing the cohomology ring of a real moment-angle complex $\mathbb{R}\ZZ_K$.
Let $K$ be a simplicial complex with $m$ vertices. As we know, the cohomology group of $\mathbb{R}\ZZ_K$ is isomorphic to the BHR as $\kk$-module, in the following way (see \cite{BBCG1}):
\[H^p(\mathbb{R}\ZZ_K;\kk)\cong\bigoplus_{J\subset[m]}\w H^{p-1}(K_J;\kk).\]

To describe the ring structure of $H^*(\mathbb{R}\ZZ_K)$, let $T(u_1,\dots,u_m;t_1,\dots,t_m)$ be the differential graded tensor $\kk$-algebra  with $2m$ generators such that
\[\deg u_i=1,\ \deg t_i=0;\ d(t_i)=u_i,\ d(u_i)=0.\]
Let $R$ be the quotient of $T(u_1,\dots,u_m;t_1,\dots,t_m)$ under the following relations:
\[
\begin{split}
u_it_i=u_i,\ \ t_iu_i=0,&\ \ u_it_j=t_ju_i,\ \ t_it_i=t_i,\\
&u_iu_i=0,\ \ u_iu_j=-u_ju_i,\ \ t_it_j=t_jt_i,
\end{split}
\]
for $i,j=1,\dots,m$ and $i\neq j$. Denote by $u_I$ (resp. $t_I$) the monomial $u_{i_1}\cdots u_{i_k}$ (resp. $t_{i_1}\cdots t_{i_k}$) for an ordered subset $I=\{i_1,\dots,i_k\}\subset[m]$. ($t_I$ dose not depend on the order of $I$ since $t_it_j=t_jt_i$.) Let $\mathcal{I}$ be the ideal of $R$ generated by all square-free monomials $u_\sigma$ such that $\sigma$ is not a face of $K$. Clearly, as $\kk$-module, $R/\II$ is freely generated by the square-free
monomials $u_\sigma t_{I\setminus\sigma}$, where $\sigma\subset I\subset [m]$ and $\sigma \in K$.

\begin{thm}[\cite{C17}]\label{thm:ring of real m-a}
There is a graded ring isomorphism
\[
H(R/\mathcal{I}, d) \cong H^*(\mathbb{R}\ZZ_K).
\]
\end{thm}

In view of Theorem \ref{thm:ring of real m-a}, a natural question is that how to describe the multiplicative structure of $H^*(\mathbb{R}\ZZ_K)$ in terms of the full subcomplex of $K$? We answer this question by slightly modifying the rules for multiplication in the BHR.

For subsets $I'\subset I\subset[m]$, define $\varphi^{I'}_I: K_{I'}\ha K_I$.
Given two elements $\alpha\in \w {H}^{p-1}(K_I)$ and $\beta\in \w {H}^{q-1}(K_J)$ in the BHR of $K$,
define the \emph{$\star $-product} of $\alpha$ and $\beta$ to be
\[\alpha\star\beta:=\alpha\cdot(\varphi^{J\setminus(I\cap J)}_J)^*(\beta)\in \w H^{p+q-1}(K_{I\cup J}).\]
Here ``{\Large$\cdot$}'' means the product in the BHR of $K$.

We call the $\kk$-module $\bigoplus_{I\subseteq [m]} \w {H}^*(K_I;\kk)$ with the $\star$-product the \emph{second Baskakov-Hochster ring} (BHR$\mathrm{^{II}}$ for abbreviation) of $K$.

\begin{prop}\label{prop:bhr}
Let $K$ be a simplicial complex. Then the BHR$\mathrm{^{II}}$ of $K$ is isomorphic to $H^*(\mathbb{R}\ZZ_K)$.
\end{prop}

\begin{proof}
First we consider a special case when $\alpha\in \w {H}^{p-1}(K_{I})$ and $\beta\in \w {H}^{q-1}(K_{J})$ are represented by monic monomials in $R/\II$, i.e., $\alpha=u_{\sigma_1} t_{I\setminus\sigma_1}$, $\beta=u_{\sigma_2} t_{J\setminus\sigma_2}$. Thus $\alpha=[\sigma_1^*]\in \w {H}^{p-1}(K_{I})$ and $\beta=[\sigma_2^*]\in \w {H}^{q-1}(K_{J})$. The ring structure of $R/\II$ shows that
\[u_{\sigma_1} t_{I\setminus\sigma_1}\cdot u_{\sigma_2} t_{J\setminus\sigma_2}=
\begin{cases}
0&\text{\quad if }\sigma_2\cap I\neq\emptyset,\\
u_{\sigma_1\cup\sigma_2}t_{(I\cup J)\setminus(\sigma_1\cup\sigma_2)}&\text{\quad otherwise}.
\end{cases}
\]
Here $\sigma\cup \tau$ means the juxtaposition of $\sigma$ and $\tau$. An easy calculation shows that this is just $\alpha\star\beta\in \w H^{p+q-1}(K_{I\cup J})$.

For the general case, in which $\alpha$ and $\beta$ are polynomials, we can get the desired formula by applying the argument above to each monomial term.
\end{proof}

\begin{proof}[Proof of Lemma \ref{lem:hoch}]
According to the definition of $\star$-product, $\alpha'\star\beta=\alpha'\cdot\beta'$.
On the other hand, since BHR$\mathrm{^{II}}$ is isomorphic to the cohomology ring $H^*(\Rr\ZZ_K)$, which is a skew commutative ring, it follows that $\alpha'\star\beta=(-1)^{pq}\beta\star\alpha'=(-1)^{pq}\beta\cdot\alpha=\alpha\cdot\beta$.
\end{proof}

\section{Factor index and its application}
To prepare for the proof of Propsition \ref{prop:scc}, we introduce an algebraic concept, and use it to get some useful algebraic-combinatorics results in this appendix.
\begin{Def}\label{def:ind}
Let $A$ be an algebra over a field $\kk$. Given a nonzero element $\alpha\in A$, if a $\kk$-subspace $V\subset A$ satisfies the condition that for any nonzero element $v\in V$, $v$ is a factor of $\alpha$ (i.e., there exists $u\in A$ such that $vu=\alpha$), then $V$ is called a \emph{factor space} of $\alpha$ in $A$. Denote by
$\mathcal {F}_\alpha$ the set of all factor spaces of $\alpha$, and define the \emph{factor index} of $\alpha$ in $A$ to be
\[\mathrm{ind}(\alpha)=\mathrm{max}\{\dim_\kk(V): V\in \mathcal {F}_\alpha\}.\]

Let $A=\bigoplus_iA^i$ be a graded $\kk$-algebra, $0\neq\alpha\in A$. If $V\subset A^k$ is a factor space of $\alpha$ in $A$, then $V$ is called a \emph{$k$-factor space} of $\alpha$ in $A$. Denote by $\mathcal {F}_\alpha^k$ the set of all $k$-factor spaces of $\alpha$. The \emph{$k$-factor index} of $\alpha$ in $A$ is defined to be
\[\mathrm{ind}_k(\alpha)=\mathrm{max}\{\dim_\kk(V): V\in \mathcal {F}_\alpha^k\}\]
\end{Def}
\begin{exmp}
(i) If $A$ is the polynomial algebra $\kk[x]$ with $\kk=\Cc$ and $\deg x=1$, then for any $0\neq f\in A$ with $\deg f=n$, $\mathrm{ind}(f)=\mathrm{ind}_k(f)=1$ for $k\leqslant n$ and $\mathrm{ind}_k(f)=0$ for $k>n$. This follows from the fundamental theorem of algebra.

(ii) If $A=\bigoplus_{i=0}^d A^i$ is a Poincar\'e duality $\kk$-algebra, then for $0\neq\alpha\in A^d$ and $k\leqslant d$, $\mathrm{ind}_k(\alpha)=\dim_\kk A^k$.
\end{exmp}
In the following, we always assume $\kk$ to be a field, and the cohomology with coefficients in $\kk$ will be implicit.
\begin{lem}\label{lem:ind}
Let $K$ be a simplicial complex with $m$ ($m\geqslant3$) vertices. Then for any cohomology class $\xi\in H^{p}(\mathcal {Z}_K)$, $p\geqslant4$, we have
\begin{enumerate}[(i)]
\item\label{case:i} $\mathrm{ind}_3(\xi)\leqslant \binom{m}{2}-m$.\vspace{5pt}
\item\label{case:ii} For $m\geqslant 4$, $\mathrm{ind}_3(\xi)=\binom{m}{2}-m$ if and only if $K$ is the boundary of an $m$-gon and $\xi\in H^{m+2}(\mathcal {Z}_K)$.
\end{enumerate}
\end{lem}
\begin{proof}
We prove \eqref{case:i} by induction on $m$. For the base case $m=3$, $K$ is one of $\Delta^2$, $\partial \Delta^2$, $\Delta^1\vee\Delta^1$, $\Delta^0\sqcup\Delta^1$, $\Delta^0\sqcup\Delta^0\sqcup\Delta^0$, and the statement can be veified case by case. For the induction step $m>3$, let $e$ be the number of edges of $K$, and let $e_i$ be the number of edges containing $i\in\VV(K)$ as a vertex for $1\leqslant i\leqslant m$.  Suppose on the contrary that $\mathrm{ind}_3(\xi)>{m\choose 2}-m$ for some $\xi\in H^p(\mathcal {Z}_K)$ with $p\geqslant 4$. Since $\dim_\kk H^3(\mathcal {Z}_K)$ equals the number of two-element missing faces of $K$, we have the following inequality
\[\binom{m}{2}-m<\mathrm{ind}_3(\xi)\leqslant\dim_\kk H^3(\mathcal {Z}_K)=\binom{m}{2}-e.\]
It follows that $e<m$. Combining this with the relation $2e=\sum_{i=1}^m e_m$ we see that $e_i<2$ for some $1\leqslant i\leqslant m$.

Without loss of generality, assume $e_m<2$, and let $I=\{1,\dots,m-1\}$.
{Then $K=K_I\vee\Delta^1$ or $K=K_I\sqcup\Delta^0$. Hence by \cite[Lemma 4.4 and Lemma 4.5]{FW21} we have a ring isomorphism}
\[
\w H^*(\mathcal {Z}_K)\cong (\w H^*(\mathcal {Z}_{K_I})\otimes\Lambda[y])\oplus B,\quad\deg y=1,
\]
where $B$ is a $\kk$-algebra with a trivial multiplicative structure. Moreover, let $\phi:\ZZ_{K_I}\to \ZZ_K$ be the inclusion map.
Then $\phi^*:\w H^*(\mathcal {Z}_K)\to \w H^*(\mathcal {Z}_{K_I})$ is the projection to $(\w H^*(\mathcal {Z}_{K_I})\otimes 1,0)$.
Suppose $\alpha\in H^3(\mathcal Z_K)$ is a factor of $\xi$. Then $\alpha=\alpha_1\otimes1+\alpha_2$ for some $0\neq\alpha_1\in H^3(\mathcal {Z}_{K_I})$ and $\alpha_2\in B$, because $H^{2}(\ZZ_K)=0$ and $B\cdot B=0$. It follows that $\dim_\kk\,\phi^*(V)=\dim_\kk V$ for any $V\in \mathcal {F}_\xi^3$. We consider two cases:
\begin{enumerate}[$\bullet$]
\item $\phi^*(\xi)\neq0$. Then from the functorial property of cohomology ring we know that $\phi^*(V)$ is a factor space of $\phi^*(\xi)\in H^p(\mathcal {Z}_{K_I})$.

\item $\phi^*(\xi)=0$. Then we can write $\xi=\xi_0\otimes y$ with $0\neq\xi_0\in H^{p-1}(K_I)$, since $\xi$ can be written as the product of two elements of nonzero degree by the inequality $\mathrm{ind}_3(\xi)>0$, while $B\cdot B=0$. For any nonzero element $\alpha\in V\in \mathcal {F}_\xi^3$ (write $\alpha=\alpha_1\otimes1+\alpha_2$ as we have seen before), suppose that $\alpha\beta=\xi$ for some $\beta\in H^{p-3}(\mathcal Z_K)$, and write $\beta=\beta_0\otimes y+\beta_1\otimes 1+\beta_2$ with $\beta_0,\beta_1\in \w H^*(\mathcal {Z}_{K_I})$ and $\beta_2\in B$. Since $\alpha\beta=\alpha_1\beta_0\otimes y+\alpha_1\beta_1\otimes 1=\xi=\xi_0\otimes y$, it follows that $\alpha_1\beta_0=\xi_0$, and so $\alpha_1=\phi^*(\alpha)$ is a factor of $\xi_0$.
Thus $\phi^*(V)$ is a factor space of $\xi_0\in H^{p-1}(\mathcal {Z}_{K_I})$ by definition. Note that $\deg(\beta_0)\neq0$, so $\deg(\xi_0)\geqslant 4$.
\end{enumerate}
So in either case, there is a nonzero element $\xi'\in H^{\geqslant 4}(\mathcal {Z}_{K_I})$ such that \[\mathrm{ind}_3(\xi')\geqslant\mathrm{ind}_3(\xi)>\binom{m}{2}-m>\binom{m-1}{2}-(m-1).\] By induction, this is a contradiction. 

\emph{Remark}: If $K$ has more than two connected components, say $K=K_{I}\sqcup K_{J}$ with $I,J\neq\emptyset$, then by \cite[Proposition 4.6]{FW21} there is a ring isomorphism:
\[\w H^*(\ZZ_K)\cong (\w H^*(\mathcal {Z}_{K_I})\otimes\Lambda[y_j:j\in J])\oplus (\w H^*(\mathcal {Z}_{K_J})\otimes\Lambda[y_i:i\in I])\oplus B,\]
where $B$ has trivial multiplication. See also \cite[Theorem 6.12]{GT16} for a homotopy exposition of this isomorphism. Let $\phi_I:\ZZ_{K_I}\to \ZZ_K$ and $\phi_J:\ZZ_{K_J}\to \ZZ_K$ be the inclusions. By considering the two cases: (i) $\phi_I^*(\xi)\neq 0$ or $\phi_J^*(\xi)\neq 0$; (ii) $\phi_I^*(\xi)=\phi_J^*(\xi)=0$, one can  similarly show that $\mathrm{ind}_3(\xi)\leq \mathrm{ind}_3(\xi')$ for some $\xi'\in H^{\geqslant 4}(\mathcal {Z}_{K_I})$ or $H^{\geqslant 4}(\mathcal {Z}_{K_J})$. For the second case, one may assume that $\xi=\sum_{i=0}^r\xi_i\otimes y_{J_i}+\sum_{j=0}^s\eta_j\otimes y_{I_j}$, where $0\neq \xi_i\in \w H^*(\mathcal {Z}_{K_I})$, $0\neq \eta_j\in \w H^*(\mathcal {Z}_{K_J})$, $\emptyset\neq I_j\subset I$, $\emptyset\neq J_i\subset J$, and make argument on $\xi_0\otimes y_{J_0}$ as before to get the desired $\xi'=\xi_0$. Then by the part \eqref{case:i} of the lemma, $\mathrm{ind}_3(\xi)<\binom{m}{2}-m$ because $|I|,|J|<m$.

One direction of \eqref{case:ii} is an immediate consequence of the fact that $H^*(\mathcal {Z}_K)$ is a Poincar\'e duality $\kk$-algebra and the equality
\[\dim_\kk H^3(\ZZ_K)=|MF(K)|=\binom{m}{2}-m.\]
To see the other direction of \eqref{case:ii}, let $e$, $e_i$ be as in the proof of \eqref{case:i}. Then $e_i\geqslant 2$ for each $1\leqslant i \leqslant m$, for otherwise we would have $\mathrm{ind}_3(\xi)\leqslant\binom{m-1}{2}-(m-1)<\binom{m}{2}-m$ by the proof of \eqref{case:i}. Note that ${\dim_\kk H^3(\mathcal {Z}_K)}\geqslant\mathrm{ind}_3(\xi)=\binom{m}{2}-m$. So there are at least $\binom{m}{2}-m$ two-element missing faces of $K$, which implies that $K$ has at most $m$ edges, and so $2e=\sum_{i=1}^m e_i\leqslant 2m$. Combining this with $e_i\geqslant 2$ gives $e_i=2$ for each $1\leqslant i\leqslant m$.
{Hence by an easy combinatorial argument we conclude that $K$ is the disjoint union of several circles. But the remark at the end of the proof of \eqref{case:i} shows that $K$ is path-connected, and the result follows.}
\end{proof}
\begin{cor}\label{cor:ind}
Let $\xi$ be a nonzero element of $H^i(\mathcal {Z}_{K})$ with $i\geqslant6$ and $\mathrm{ind}_3(\xi)=\binom{i-2}{2}-(i-2)$. Suppose $\xi={\sum_{k=1}^l\xi_k},\,0\neq\xi_k\in \w H^{i-|J_k|-1}(K_{J_k})$. Then $K_{J_k}$ is the boundary of an $(i-2)$-gon for each $J_k$.
\end{cor}
\begin{proof}
First note that $\xi$ can be written as the product of two elements of nonzero degree, since $\binom{i-2}{2}-(i-2)>0$ for $i\geqslant 6$. It follows that $i-|J_k|-1\geqslant1$, i.e., $|J_k|\leqslant i-2$ for all $J_k$, by the degree rule of BHR. Choose a $J_q$ such that $|J_q|$ is minimal among all $J_k$, and consider the ring homomorphism
$\phi_q^*:H^*(\mathcal {Z}_{K})\to H^*(\mathcal {Z}_{K_{J_q}})$, which is induced by the
inclusion $\phi_q:\ZZ_{K_{J_q}}\to \ZZ_K$. 
By abuse of notation, we also regard $\xi_q$ as an element of $H^*(\mathcal {Z}_{K_{J_q}})$. Then it is easy to see that $\phi_q^*(\xi)=\xi_q$ because $|J_q|$ is mininal. The functorial property shows that if $\alpha$ is a factor of $\xi$ in $H^*(\mathcal {Z}_{K})$, then $\phi_q^*(\alpha)$ is a factor of $\xi_q$ in $H^*(\mathcal {Z}_{K_{J_q}})$.
Thus, if $V\in \mathcal {F}^3_\xi$, then $\phi_q^*(V)\in \mathcal {F}^3_{\xi_q}$ and $\dim_\kk \phi_q^*(V)=\dim_\kk V$, which implies that
\begin{equation}\label{eq:inequality1}
\mathrm{ind}_3(\xi_q)\geqslant\binom{i-2}{2}-(i-2).
\end{equation}
On the other hand, applying Lemma \ref{lem:ind} \eqref{case:i} to $\xi_q\in H^{\geqslant 6}(\ZZ_{K_{J_q}})$ we have
\begin{equation}\label{eq:inequality2}
\mathrm{ind}_3(\xi_q)\leqslant\binom{|J_q|}{2}-|J_q|\leqslant\tbinom{i-2}{2}-(i-2).
\end{equation}
Combining \eqref{eq:inequality2} and \eqref{eq:inequality1}, we obtain $|J_q|=i-2$ and $\mathrm{ind}_3(\xi_q)=\binom{i-2}{2}-(i-2)$. Since $|J_k|\leqslant i-2$ for all $J_k$ and $|J_q|=i-2$ is minimal, it follows that $|J_k|=i-2$ for all $J_k$, and then $\mathrm{ind}_3(\xi_k)=\binom{i-2}{2}-(i-2)$ for all $J_k$, by the same reasoning. 
Lemma \ref{lem:ind} \eqref{case:ii} then shows that $K_{J_k}$ is the boundary of an $(i-2)$-gon for all $J_k$.
\end{proof}

\section{Proof of Propsition \ref{prop:scc}}\label{app:proof of prop:scc}
\begin{lem}\label{lem:mf}
Let $K$ be a simplicial complex. Assume that $K_I$ is a full subcomplex that satisfies $\w H^0(K_I)\neq 0$.
Then for any vertex $i\in I$ and $0\neq\xi_I\in \w H^0(K_I)$, there is a two-element missing face $\omega\in MF(K_I)$ containing $i$ such that $\phi^*(\xi_I)\neq0$,
where $\phi:K_\omega\ha K_I$ is the inclusion.
\end{lem}
\begin{proof}
Suppose $K_{I_1},\dots,K_{I_k}$ are the path-components of $K_I$. Then we can write $\xi_I=\sum_{j=1}^k a_j\cdot\sum_{l\in I_j}\{l\}$, $a_j\in\kk$.
The assumption that $\xi_I\neq0$ implies that the coefficients $a_j$'s are not all equal. Thus if we assume $i\in I_s$, then we can take a component
$K_{I_t}$ with $a_s\neq a_t$. Pick a vertex $j\in I_t$, then $\omega=\{i,j\}\in MF(K_I)$ is the desired missing face.
\end{proof}

\begin{lem}\label{lem:ann}
Let $K$ be a flag complex which satisfies the SCC.
Suppose there is a sequence of different subsets $I_1,I_2,\dots,I_k\subset\VV(K)$ ($k\geqslant2$) satisfying $|I_1|=\cdots=|I_k|\leqslant 3$ and $\w H^0(K_{I_j})\neq0$ for $1\leqslant j\leqslant k$. Take a nonzero element $\xi_j\in\w H^0(K_{I_j})\subset H^*(\ZZ_K)$ for each $1\leqslant j\leqslant k$.
Then for $\mathbf{x}=\xi_1+\xi_2+\cdots+\xi_k\in H^*(\ZZ_K)$, we have
\[\dim_\kk(\mathrm{ann}(\xx))< \min\{\dim_\kk(\mathrm{ann}(\xi_j)): 1\leqslant j\leqslant k\}.\]
\end{lem}
\begin{proof}
We need to show that $\dim_\kk(\mathrm{ann}(\xx))<\dim_\kk(\mathrm{ann}(\xi_j))$ for all $j$, and without loss of generality we only prove the inequality for $\xi_1$.

First we define a subspace $V_{\xi_1}\subset H^*(\ZZ_K)$ as follows.
For each subset $J\subset[m]$, let $A_J=\mathrm{ann}(\xi_1)\cap\w H^*(K_J)$, and take a complementary subspace $V_{\xi_1,J}$ of $A_J$ in $\w H^*(K_J)$.
Then let $V_{\xi_1}=\bigoplus_{J\subset[m]}V_{\xi_1,J}$. We claim that $V_{\xi_1}\cap \mathrm{ann}(\xi_1)=0$ and $V_{\xi_1}\oplus\mathrm{ann}(\xi_1)=H^*(\ZZ_K)$.
To see this, for any $0\neq\eta\in V_{\xi_1}$, write $\eta=\sum\eta_{i}$, $0\neq\eta_{i}\in V_{\xi_1,J_i}$, $J_p\neq J_q$ for $p\neq q$. By the construction of $V_{\xi_1,J_i}$, we have $0\neq\xi_1\cdot\eta_i\in\w H^*(K_{I_1\cup J_i})$.
The {intersection rule of BHR} shows that $I_1\cap J_i=\emptyset$ for all such $J_i$, so $I_1\cup J_p\neq I_1\cup J_q$ for $p\neq q$, and this implies that $\xi_1\cdot\eta\neq0$. So  $V_{\xi_1}\cap \mathrm{ann}(\xi_1)=0$. The rest of the claim can be seen from the fact that
\[
\w H^*(\ZZ_K)=\bigoplus_{J\subset[m]}(V_{\xi_1,J}\oplus A_J)\ \text{ and }\ \bigoplus_{J\subset[m]}A_J\subset\mathrm{ann}(\xi_1).
\]

Next we show that for any $0\neq\eta\in V_{\xi_1}$, $\xx\cdot\eta\neq 0$. As before write $\eta=\sum\eta_{i}$, $0\neq\eta_{i}\in V_{\xi_1,J_i}$.
We will prove that $p_{I_1\cup J_1}(\xx\cdot\eta)\neq0$, where $p_{I_1\cup J_1}$ is the projection map in \eqref{eq:projection}.
Since $0\neq\xi_1\cdot\eta_1\in \w H^*(K_{I_1\cup J_1})$ and $I_1\cup J_1\neq I_1\cup J_q$ for $q\neq1$, if $p_{I_1\cup J_1}(\xx\cdot\eta)=0$ there would be $p,q\neq 1$ such that $I_p\cup J_q=I_{1}\cup J_1$ and $\xi_p\cdot\eta_q\neq0$. Combining this with the fact that $I_1\neq I_p$, we get $I_1\cap J_q\neq\emptyset$.
However, this implies that $\xi_1\cdot\eta_q=0$, contradicting the fact that $0\neq\eta_{q}\in V_{\xi_1,J_q}$.

\begin{rem}\label{rem:ann general}
In fact, the arguments above show a general fact. That is, for any simplicial complex $K$,
an element $\mathbf{x}=\xi_1+\xi_2+\cdots+\xi_k\in H^*(\ZZ_K)$ with $0\neq\xi_j\in\w H^*(K_{I_j})$, {$I_i\not\subset I_j$ for $i\neq j$}, has the property that
\[
\dim_\kk(\mathrm{ann}(\xx))\leqslant \mathrm{min}\{\dim_\kk(\mathrm{ann}(\xi_i)):1\leqslant i\leqslant k\}.
\]
\end{rem}

Now suppose the hypotheses of the lemma is satisfied. If we can find an element $\lambda\in \mathrm{ann}(\xi_1)$,
such that for any nonzero element $\eta\in V_{\xi_1}\oplus\kk\cdot\lambda$, $\xx\cdot\eta\neq 0$, then we can get the desired inequality
\[\dim_\kk(\mathrm{ann}(\xx))<\dim_\kk(\mathrm{ann}(\xi_1)).\]
The rest of the proof is devoted to a search of such an element.

To do this, we start with establishing a few preliminary facts. Choose a vertex $i_2\in I_2$, $i_2\not\in I_1$. By Lemma \ref{lem:mf} there is a missing face $\omega\in MF(K_{I_2})$ containing $i_2$ such that $0\neq f^*(\xi_2)\in\w H^0(K_{\omega})$ for $f:K_\omega\ha K_{I_2}$. Choosing a vertex $i_1\in I_1\setminus I_2$, the assumption that $K$ satisfies the SCC shows that there is an subset $I\subset[m]$ such that $\omega\subset I$, $i_1\notin I$, $|K_I|\cong S^1$ and $\w H^0(K_{J})\neq0$ for $J=\{i_1\}\cup I\setminus\omega$.
There are simplicial inclusions:
\[\begin{split}
K_\omega\stackrel{\phi_1}{\ha} K_{I_2\setminus J}\stackrel{\phi_2}{\ha} K_{I_2},\ K_{I_2\cap I}\stackrel{\phi_3}{\ha} K_{I_2};\\
K_{I\setminus I_2}\stackrel{\psi_1}{\ha} K_{J\setminus I_2}\stackrel{\psi_2}{\ha} K_J,\  K_{I\setminus\omega}\stackrel{\psi_3}{\ha} K_J.
\end{split}\]
Let $\xi_J\in \w H^0(K_{J})$ be an element such that $\psi^*_3(\xi_J)\neq0$.
We claim that
\begin{equation}\label{eq:ann-1}
\xi_2\cdot\psi^*_2(\xi_J),\quad \xi_2\cdot\psi^*_1\circ\psi^*_2(\xi_J)\neq0.
\end{equation}
To deduce \eqref{eq:ann-1}, let $h: K_I\ha K_{I_2\cup J}$.
Then from the following commutative diagramm:
\[\begin{gathered}\xymatrix{
	\w H^*(K_{I_2\setminus J})\otimes\w H^*(K_{J})\ar[r]\ar[d]^{\phi_1^*\otimes\psi_3^*}&\w H^*(K_{I_2\cup J})\ar[d]^{h^*}\\
	\w H^*(K_{\omega})\otimes\w H^*(K_{I\setminus\omega})\ar[r]&\w H^*(K_{I})
}\end{gathered},\]
where the horizontal maps are the product maps in BHR, we have \[h^*(\phi_2^*(\xi_2)\cdot\xi_J)=\phi_1^*\circ\phi_2^*(\xi_2)\cdot\psi_3^*(\xi_J)\] 
Since $0\neq f^*(\xi_2)=\phi_1^*\circ\phi_2^*(\xi_2)\in\w H^0(K_\omega)\cong\kk$, $0\neq\psi_3^*(\xi_J)\in\w H^0(K_{I\setminus\omega})\cong\kk$ and $H^*(\ZZ_{K_I})$ is a Poincar\'e duality algebra, we have $\phi_1^*\circ\phi_2^*(\xi_2)\cdot\psi_3^*(\xi_J)\neq0$.
It follows that $\phi_2^*(\xi_2)\cdot\xi_J\neq0$, and so $\xi_2\cdot\psi^*_2(\xi_J)\neq0$ since $\xi_2\cdot\psi^*_2(\xi_J)=\phi_2^*(\xi_2)\cdot\xi_J$ by Lemma \ref{lem:hoch}.
Furthermore, note that $h$ is the composition
$K_I\ha K_{I_2\cup I}\stackrel{g}{\ha} K_{I_2\cup J}$, so $h^*(\xi_2\cdot\psi^*_2(\xi_J))\neq0$ implies $g^*(\xi_2\cdot\psi^*_2(\xi_J))=\xi_2\cdot\psi^*_1\circ\psi^*_2(\xi_J)\neq0$.
Thus \eqref{eq:ann-1} holds.

Now we can choose the desired element $\lambda\in \mathrm{ann}(\xi_1)$ in different cases.

(i) $|I_1|=2$. In this case, $I_2=\omega$. Let $\lambda=\xi_J\in\w H^0(K_J)$.
Since $I_1\cap J\neq\emptyset$, $\lambda\in \mathrm{ann}(\xi_1)$.
\eqref{eq:ann-1} says that $0\neq\xi_2\cdot\lambda\in\w H^1(K_{I_2\cup J})$. Then from the fact that $I_p\neq I_q$ for $p\neq q$ we deduce that $p_{I_2\cup J}(\xx\cdot\lambda)\neq0$, and so $\xx\cdot\lambda\neq0$.

It remains to show that $\xx\cdot(\eta+\lambda)\neq0$ for any $0\neq\eta\in V_{\xi_1}$.
To see this, write
\[\eta=\sum_{i=1}^l\eta_{i},\quad 0\neq\eta_{i}\in V_{\xi_1,J_i}.\]
 It suffices to show that $p_{I_2\cup J}(\xi_j\cdot\eta_i)=0$ for any $1\leqslant j\leqslant k,1\leqslant i\leqslant l$. Indeed, $\xi_1\cdot\eta_i\neq0$ gives $I_1\cap J_i=\emptyset$, so $i_1\not\in J_i$.
This, together with the fact that $I_2\cup J=\{i_1\}\cup I$ and $|I_j|=2$, implies that if $p_{I_2\cup J}(\xi_j\cdot\eta_i)\neq0$ then $J_i\subset I$ and $|I\setminus J_i|=1$. But this would give $|K_{J_i}|\cong D^1$ since $|K_I|\cong S^1$, contradicting $0\neq\eta_i\in\w H^*(K_{J_i})$.
Hence $\lambda$ is the desired element.

(ii) $|I_1|=3$ and $|I_1\cap I_2|=2$, say $I_1=\{i_1,i_3,i_4\}$, $I_2=\{i_2,i_3,i_4\}$, and  $\omega=\{i_2,i_3\}$.
First we consider the case $i_4\not\in I$, i.e., $I_2\cap J=\emptyset$. In this case, take $\lambda=\xi_J$, then we have $\lambda\in \mathrm{ann}(\xi_1)$ and $\xx\cdot\lambda\neq0$ as in the first paragraph of  case (i). Also, for any $0\neq\eta=\sum_{i=1}^l\eta_{i}\in V_{\xi_1}$,  $0\neq\eta_{i}\in V_{\xi_1,J_i}$, we have $p_{I_2\cup J}(\xi_j\cdot\eta_i)=0$, $1\leqslant j\leqslant k,1\leqslant i\leqslant l$. This is because if $p_{I_2\cup J}(\xi_j\cdot\eta_i)\neq 0$ so that $I_j\sqcup J_i=I_2\cup J$, then  $J_i\subset I$ and $|I\setminus J_i|=1$ by the fact that $J_i\cap I_1=\emptyset$, $I_2\cup J=\{i_1,i_4\}\sqcup I=I_1\cup I$,  $|I_j|=3$, and we reach a contradiction as in the second paragraph of  case (i). So  $\lambda$ is the desired element.

For the case $i_4\in I$, i.e., $I_2\subset I$, let $\lambda=\psi^*_1\circ\psi^*_2(\xi_J)\in\w H^0(K_{I\setminus I_2})$.
First let us show that $\lambda\in\mathrm{ann}(\xi_1)$.
Note that $K_{I\setminus\omega}$ has two components, say $K_1$ and $K_2$.
Since $0\neq\psi^*_3(\xi_J)\in\w H^0(K_{I\setminus\omega})$, we may assume $\{i_1,v\}\in MF(K)$ for all $v\in\VV(K_1)$, and $\sum_{v\in\VV(K_1)}\{v\}^*$ represents the cohomology class $\xi_J\in\w H^0(K_J)$, so that $\lambda=\sum_{v\in\VV(K_1)\setminus\{i_4\}}\{v\}^*$.
From the intersection rule of BHR it follows that $\xi_1\cdot\lambda\in \w H^1(K_{J\cup\{i_3\}})$ is represented by a cocycle of the form
\[[\alpha]=\sum_{\sigma_i\in K_U,|\sigma_i|=2} a_i\cdot\sigma_i^*,\ \text{ where }a_i\in\kk,\ U=\VV(K_1)\cup\{i_3\}.\]
{Since $\{i_1,v\}\in MF(K)$ for all $v\in\VV(K_1)$,  $K_{J\cup\{i_3\}}$ deformation retracts to $K_{J\setminus\VV(K_1)}$ by retracting $K_U$, which is a triangulation of $D^1$, to the endpoint $\{v_3\}$.}
This implies that $\xi_1\cdot\lambda$ is a zero class in $\w H^1(K_{J\cup\{i_3\}})$ since $i^*([\alpha])=0$ for $i:K_{J\setminus\VV(K_1)}\to K_{J\cup\{i_3\}}$. So $\lambda\in\mathrm{ann}(\xi_1)$.

Since $0\neq\xi_2\cdot\lambda\in\w H^1(K_I)$ by \eqref{eq:ann-1} and $I_p\neq I_q$ for $p\neq q$, it follows that $p_I(\xx\cdot\lambda)\neq0$. Hence $\xx\cdot\lambda\neq0$, and it remains to show that $\xx\cdot(\eta+\lambda)\neq0$ for any $0\neq\eta\in V_{\xi_1}$.
Suppose on the contrary that $\xx\cdot(\eta+\lambda)=0$ for some $\eta=\sum\eta_{i}$, $0\neq\eta_{i}\in V_{\xi_1,J_i}$.
 From the discussion in the third paragraph of the proof of this lemma we see that $p_{I_1\cup J_1}(\xx\cdot\eta)\neq0$, so the fact that $\xx\cdot(\eta+\lambda)=0$ and $\lambda\in\mathrm{ann}(\xi_1)$ implies that there exists $\xi_p$ ($p\neq1$) such that $p_{I_1\cup J_1}(\xi_p\cdot\lambda)\neq0$, and therefore $I_1\cup J_1=I_p\cup (I\setminus I_2)$. However this, together with the fact that $I_1\cap (I\setminus I_2)=\emptyset$, shows that $I_1\subset I_p$, and so $I_1=I_p$ since $|I_p|=|I_1|$, contradicting the assumption of the lemma.
So $\lambda$ is the desired element.

(iii) $|I_1|=3$ and $|I_1\cap I_2|\leqslant1$. Let
\begin{equation}\label{eq:cases}
\lambda=
\begin{cases}\psi^*_2(\xi_J)\in\w H^0(K_{J\setminus I_2})\ &\text{ if } I_1\cap (I\setminus I_2)=\emptyset;\\
\psi^*_1\circ\psi^*_2(\xi_J)\in\w H^0(K_{I\setminus I_2})\ &\text{ otherwise. }
\end{cases}\end{equation}
Then in both cases $\lambda\in\mathrm{ann}(\xi_1)$. If $I_1\cap (I\setminus I_2)=\emptyset$ (resp. $I_1\cap (I\setminus I_2)\neq\emptyset$), we have $p_{I_2\cup J}(\xx\cdot\lambda)=\xi_2\cdot\lambda\neq0$ (resp. $p_{I_2\cup I}(\xx\cdot\lambda)=\xi_2\cdot\lambda\neq0$) by \eqref{eq:ann-1} and the fact that $I_p\neq I_q$ for $p\neq q$.
So if $\xx\cdot(\eta+\lambda)\neq0$ for any $0\neq\eta\in V_{\xi_1}$, then $\lambda$ is the desired element.
We will show this by reducing to case (ii).
The argument will be carried out for both cases of \eqref{eq:cases} simultaneously.

In the first (resp. second) case, the condition $p_{I_2\cup J}(\xx\cdot\lambda)\neq0$ (resp. $p_{I_2\cup I}(\xx\cdot\lambda)\neq0$) implies that if
$\xx\cdot(\eta+\lambda)=0$ for some  $\eta=\sum_{i=1}^l\eta_{i}$ with $0\neq\eta_{i}\in V_{\xi_1,J_i}$, then there would be $\xi_p$ and $\eta_q$ such that $p_{I_2\cup J}(\xi_p\cdot\eta_q)\neq0$  (resp. $p_{I_2\cup I}(\xi_p\cdot\eta_q)\neq0$), so that $I_2\cup J=I_p\cup J_q$ (resp. $I_2\cup I=I_p\cup J_q$) and $I_p\cap J_q=\emptyset$. This implies that
\begin{equation}\label{eq:ann-3}
	J\setminus (I_2\cup J_q)=I_p\setminus I_2\ \text{ (resp. $I\setminus (I_2\cup J_q)=I_p\setminus I_2$)}.
\end{equation}
Furthermore, since $p_{I_1\cup J_q}(\xx\cdot\eta)\neq0$, there would be $\xi_r$ such that $p_{I_1\cup J_q}(\xi_r\cdot\lambda)\neq0$. 
Therefore we have
\begin{gather*}
	I_r\cap (J\setminus I_2)=\emptyset\ \text{ and }\ I_r\cup (J\setminus I_2)=I_1\cup J_q \\
	(\text{resp.}\ I_r\cap (I\setminus I_2)=\emptyset\  \text{ and}\ I_r\cup (I\setminus I_2)=I_1\cup J_q).
\end{gather*}
It follows that
\begin{equation}\label{eq:ann-4}
	J\setminus (I_2\cup J_q)=I_1\setminus I_r\ \text{ (resp. $I\setminus (I_2\cup J_q)=I_1\setminus I_r$).}
\end{equation}

It is clear that $I_1\not\subset I_2\cup J$ since $|I_1\cap I_2|\leqslant1$ and $I_1\cap (I\setminus I_2)=\emptyset$ (resp. $I_1\not\subset I_2\cup I$ since $i_1\not\in I_2\cup I$), so $p\neq1$. Moreover, since $i_1\in I_2\cup J$ and $i_1\not\in I_2\cup J_q$ (resp. $I_1\cap (I\setminus I_2)\neq\emptyset$ and $I_1\cap J_q=\emptyset$), $p\neq2$. 
By \eqref{eq:ann-3} and \eqref{eq:ann-4}, $I_p\setminus I_2\subset I_1$, and so $1\leqslant|I_p\setminus I_2|\leqslant2$ since $p\neq1,2$.
If $|I_p\setminus I_2|=2$, then $|I_1\cap I_p|=2$, since $I_p\setminus I_2\subset I_1$ and $p\neq1$. So we can reduce to case (ii) by substituting $I_p$ for $I_2$. If $|I_p\setminus I_2|=1$, then $|I_1\setminus I_r|=1$  by \eqref{eq:ann-3} and \eqref{eq:ann-4}. Hence $|I_1\cap I_r|=2$, and we can reduce to case (ii) again  by substituting $I_r$ for $I_2$.

The cases above exhaust all possibilities. Thus we complete the proof of Lemma \ref{lem:ann}.
\end{proof}

Now we start the proof of Propsition \ref{prop:scc}. For Propsition \ref{prop:scc} \eqref{st:a}, we separate the proof into two cases: $|I|=2$ and  $|I|=3$. The first case is easier.
\begin{proof}[Proof of Propsition \ref{prop:scc} \eqref{st:a} for $|I|=2$]
Note that $\bigoplus_{|I|=2}\w H^0(K_I)=H^3(\ZZ_K)$, so since $\phi$ is a graded isomorphism, for a nonzero element $\xi_I\in\w H^0(K_I)$ we have \[\phi(\xi_I)=\sum_{|J|=2}\xi'_J,\ \xi'_J\in\w H^0(K'_J).\]
The first part of statement \eqref{st:a} says that there is exactly one term $\xi'_J\neq0$ in the formula above.
Suppose on the contrary that $\phi(\xi_I)=\sum_{i=1}^k\xi'_{J_i}$ with $k>1$ and $\xi'_{J_i}\neq0$ for $1\leqslant i\leqslant k$. Since $\w H^0(K_I)\cong\kk$,
one of $\xi'_{J_i}$, say $\xi'_{J_1}$, satisfies $\phi^{-1}(\xi'_{J_1})=\xi_I+\sum_{j=1}^s\xi_{I_j}$ ($s\geqslant 1$) up to a multiplication by a nonzero element of $\kk$,
in which $I\neq I_j\in MF(K)$  and $\xi_{I_j}\neq0$ for $1\leqslant j\leqslant s$.
Since $H^*(\ZZ_K)\cong H^*(\ZZ_{K'})$, according to Lemma \ref{lem:ann} and Remark \ref{rem:ann general} we have
\begin{equation}\label{eq:1}
\begin{split}
\dim_\kk(\mathrm{ann}(\xi_I))=\dim_\kk(\mathrm{ann}(\sum_{i=1}^k\xi'_{J_i}))\leqslant\dim_\kk(\mathrm{ann}(\xi'_{J_1}))\\
=\dim_\kk(\mathrm{ann}(\xi_I+\sum_{j=1}^s\xi_{I_j}))<\dim_\kk(\mathrm{ann}(\xi_I)),
\end{split}
\end{equation}
This is a contradiction, so the first part of statement \eqref{st:a} holds for $|I|=2$. Since $|MF(K)|=\dim_\kk H^3(\ZZ_K)=\dim_\kk H^3(\ZZ_{K'})=|MF(K')|$, the second part follows.
\end{proof}

We can use the conclusion of \eqref{st:a} for $|I|=2$ to prove Propsition \ref{prop:scc} \eqref{st:c}, which is an essential ingredient in the proof of \eqref{st:a} for the case $|I|=3$.

\begin{proof}[Proof of Propsition \ref{prop:scc} \eqref{st:c}]
Suppose $K_I$ is an $n$-circuit and $MF(K_I)=\{\omega_1,\dots,\omega_t\}$ (clearly $t=\binom{n}{2}-n$). Let $c_I$ be a generator of $\w H^1(K_I)=\kk$. According to Corollary \ref{cor:ind}, $\phi(c_I)=\sum c_{J_i}'$, where $0\neq c_{J_i}'\in\w H^1(K_{J_i}')$ and $K_{J_i}'$ is an $n$-circuit of $K'$.
It is easy to see that $\w\omega_j$ ($1\leqslant j\leqslant t$) is a factor of $c_I$.
The conclusion of \eqref{st:a} for $|I|=2$ shows that $\phi(\w\omega_j)=\w\omega'_j$  for some $\omega_j'\in MF(K')$.
It follows that $\w\omega'_j$ is a factor of $c_{J_i}'$ for each $c_{J_i}'$, and therefore $\omega_j'\in MF(K'_{J_i})$.
Since $K'_{J_i}$ is an $n$-circuit, $|MF(K'_{J_i})|=\binom{n}{2}-n=t$, thus we have $MF(K'_{J_i})=\{\omega_1'\dots,\omega_t'\}$ for each $K'_{J_i}$.
The fact that a simplicial complex is uniquely determined by its missing face set shows that there is exactly one $c_{J_i}'$ in the formula for $\phi(c_I)$, which is just the statement of \eqref{st:c}.
\end{proof}

The proof of Propsition \ref{prop:scc} \eqref{st:a} for $|I|=3$ will consist of the following two lemmas.

\begin{lem}\label{lem:|I|=3}
Under the hypotheses of Proposition \ref{prop:scc}, if $\w H^0(K_I)\neq0$ for some $I\subset\VV(K)$, $|I|=3$, then $\phi(\w H^0(K_I))\subset\w H^0(K'_{J})$ for some $J\subset\VV(K')$, $|J|=3$.
\end{lem}
\begin{proof}
It is easy to see that for any simplicial complex $\Gamma$, $H^4(\ZZ_\Gamma)$ is isomorphic to the subgroup $\bigoplus_{|I|=3}\w H^0(\Gamma_I)$ by Hochster's formula. Hence for a nonzero element $\xi_I\in\w H^0(K_I)$, we can write $\phi(\xi_I)=\sum_{i=1}^k\xi'_{J_i}$, in which $0\neq\xi'_{J_i}\in\w H^0(K'_{J_i})$, $J_i\neq J_j$ for $i\neq j$, $|J_i|=3$ for $1\leqslant i\leqslant k$. The statement of the lemma is equivalent to saying that $k=1$ in the formula above.
Suppose on the contrary that $k\geqslant2$. Note that $\dim_\kk\w H^0(K_I)\leqslant 2$. If $\w H^0(K_I)=\kk$, using \eqref{eq:1} we can get a contradiction as in the proof of Proposition \textbf{\ref{prop:scc} \eqref{st:a} for $|I|=2$}.
So we only consider the case $\w H^0(K_I)=\kk\oplus\kk$.
In this case, first let us prove the following statement:
There exists a basis $\{\xi_1,\xi_2\}$ of $\w H^0(K_I)$ such that
$\phi(\xi_i)\in\w H^0(K_{J_i}')$ for some $J_i\subset\VV(K')$, $i=1,2$. Here we allow $J_1=J_2$.

To see this, take $\theta_1\in\w H^0(K_I)$ to be a nonzero element such that
\[\dim_\kk(\mathrm{ann}(\theta_1))=m_1:=\max\{\dim_\kk(\mathrm{ann}(\theta))\mid0\neq\theta\in\w H^0(K_I)\},\]
and take $\theta_2\in\w H^0(K_I)$ to be a nonzero element such that
\[\dim_\kk(\mathrm{ann}(\theta_2))=m_2:=\max\{\dim_\kk(\mathrm{ann}(\theta))\mid\theta\in\w H^0(K_I),\,\theta\not\in\kk\cdot\theta_1\}.\]
If there exists $J_i\subset\VV(K')$ such that $\phi(\theta_i)\in\w H^0(K_{J_i}')$, $i=1,2$, then $\{\theta_1,\theta_2\}$ is the desired basis.
If $\phi(\theta_1)=\sum_{i=1}^p\theta_i'$ ($p\geqslant2$), where $0\neq\theta'_i\in \w H^0(K_{J_i}')$, $J_i\neq J_j$ for $i\neq j$,
then either there exists a $\theta'_i$ such that $\phi^{-1}(\theta'_i)=\theta_1+\sum_{j=1}^q\xi_{I_j}$ ($q\geqslant 1$, $I_j\neq I$) up to multiplication by a nonzero element of $\kk$, which is impossible because of the inequality \eqref{eq:1};
or there are two $\theta'_{i_1}$ and $\theta'_{i_2}$ such that $\xi_1=p_I\circ\phi^{-1}(\theta'_{i_1})$ and $\xi_2=p_I\circ\phi^{-1}(\theta'_{i_2})$ span $\w H^0(K_I)$, and $\xi_1,\xi_2\not\in\kk\cdot\theta_1$.
Actually it can be shown that $p_J\circ\phi^{-1}(\theta'_{i_j})=0$ ($j=1,2$) for any $J\neq I$, since otherwise by applying \eqref{eq:1} again we would have
\[\dim_\kk(\mathrm{ann}(\theta_1))<\dim_\kk(\mathrm{ann}(\xi_1)),\,\dim_\kk(\mathrm{ann}(\xi_2)),\]
contradicting the maximality of $m_1$. Thus $\{\xi_1,\xi_2\}$ is the desired basis.

The remaining case is that $\phi(\theta_1)\in\w H^0(K_J')$ for some $J\subset\VV(K')$,
and $\phi(\theta_2)=\sum_{i=1}^p\theta_i'$ ($p\geqslant2$) with $0\neq\theta'_i\in \w H^0(K_{J_i}')$, $J_i\neq J_j$ for $i\neq j$.
By the same reasoning there are two $\theta'_{i_1}$ and $\theta'_{i_2}$ such that $\xi_1=p_I\circ\phi^{-1}(\theta'_{i_1})$ and $\xi_2=p_I\circ\phi^{-1}(\theta'_{i_2})$ span $\w H^0(K_I)$.
Since $\xi_1$ and $\xi_2$ are linearly independent, one of them is not in $\kk\cdot\theta_1$, say $\xi_1\not\in\kk\cdot\theta_1$.
Then $p_J\circ\phi^{-1}(\theta'_{i_1})=0$ for any $J\neq I$ since otherwise by \eqref{eq:1} we would have $\dim_\kk(\mathrm{ann}(\theta_2))<\dim_\kk(\mathrm{ann}(\xi_1))$,
contradicting the maximality of $m_2$. Hence we can choose $\{\theta_1,\xi_1\}$ as the desired basis. So the statement is true.

Returning now to the proof of the lemma for $\w H^0(K_I)=\kk\oplus\kk$, it is shown above that $k\leqslant2$ in the expression $\phi(\xi_I)=\sum_{i=1}^k\xi'_{J_i}$.  So we only need to exclude the possibility of $k=2$.
Suppose $I=\{i_1,i_2,i_3\}$ and $k=2$ is the case in the formula above. Then without loss of generality we may assume that the cohomology classes
$\phi^{-1}(\xi_{J_i}')$ have the following form:
\[\phi^{-1}(\xi_{J_i}')=a_{i,1}\{i_1\}^*+a_{i,2}\{i_2\}^*\in\w H^0(K_I),\ a_{i,1}\neq a_{i,2}\in\kk,\ i=1,2.\]
This assumption is reasonable because if we write $\phi^{-1}(\xi_{J_i}')=\sum_{k=1}^3 a_{i,k}\{i_k\}^*$, and set $S=\{(1,2),(1,3),(2,3)\}$, $S_i=\{(j,k)\in S:a_{i,j}\neq a_{i,k}\}$ for $i=1,2$, then $|S_i|\geqslant 2$ since $\phi^{-1}(\xi_{J_i}')$, $i=1,2$, are nonzero reduced cohomology classes, and so $S_1\cap S_2\neq\emptyset$, which we may assume to be $(1,2)$, hence substracting $a_{i,3}(\{i_1\}^*+\{i_2\}^*+\{i_3\}^*)$ gives the desired expression of $\phi^{-1}(\xi_{J_i}')$.

Let $\omega=\{i_1,i_2\}$.
Since $K$ satisfies the SCC, there is a subset $U\subset \VV(K)$ such that $\omega\subset U$, $i_3\notin U$, $|K_U|\cong S^1$ and $\w H^0(K_{W})\neq0$ for $W=\{i_3\}\cup U\setminus\omega$.
Let $c_U$ be a generator of $\w H^1(K_U)\cong\kk$. We have shown in \eqref{st:c} that $\phi(c_U)=c_V'\in\w H^1(K'_V)$
for some circuit $K'_V\subset K'$.
Note that $c_U=\w\omega\cdot\xi_{U\setminus\omega}$ for some $\xi_{U\setminus\omega}\in\w H^0(K_{U\setminus\omega})\cong\kk$.
So $c_V'=\phi(\w\omega)\cdot\phi(\xi_{U\setminus\omega})$. According to the conclusion of (i), $\phi({\w\omega})=\w\omega'$
for some $\omega'\in MF(K'_V)$,
therefore $0\neq\ p_{V\setminus\omega'}\circ \phi({\xi_{U\setminus\omega}})\in \w H^0(K'_{V\setminus\omega'}).$
Let $\xi_{V\setminus\omega'}'=p_{V\setminus\omega'}\circ \phi({\xi_{U\setminus\omega}})$.
The same reasoning shows that $p_{U\setminus\omega}\circ\phi^{-1}(\xi_{V\setminus\omega'}')=\xi_{U\setminus\omega}$.

Let $f: K_\omega\to K_I$ and $h:K_U\to K_{I\cup U}$ be the simplicial inclusions. The fact that $a_{i,1}\neq a_{i,2}$ implies that $f^*\circ \phi^{-1}(\xi_{J_i}')=a\w\omega$ for some $0\neq a\in\kk$. Thus $h^*(\phi^{-1}(\xi_{J_i}')\cdot\xi_{U\setminus\omega})=a\w\omega\cdot\xi_{U\setminus\omega}\neq0$, and so
$\phi^{-1}(\xi_{J_i}')\cdot\xi_{U\setminus\omega}\neq0$ for $i=1,2$.
It follows that $\phi^{-1}(\xi_{J_i}'\cdot\xi_{V\setminus\omega'}')\neq0$
since \[p_{I\cup U}\circ\phi^{-1}(\xi_{J_i}'\cdot\xi_{V\setminus\omega'}')=p_{I\cup U}(\phi^{-1}(\xi_{J_i}')\cdot\phi^{-1}(\xi_{V\setminus\omega'}'))=\phi^{-1}(\xi_{J_i}')\cdot\xi_{U\setminus\omega}.\]
The second equality comes from the fact that $\phi^{-1}(\xi_{J_i}')\in\w H^0(K_I)$.
Hence we have $\xi_{J_i}'\cdot\xi_{V\setminus\omega'}'\neq0$, $J_i\cap (V\setminus\omega')=\emptyset$ for $i=1,2$. Let
\begin{align*}
\theta_I&=(a_{2,1}-a_{2,2})\phi^{-1}(\xi_{J_1}')-(a_{1,1}-a_{1,2})\phi^{-1}(\xi_{J_2}')\\
&=(a_{1,2}a_{2,1}-a_{1,1}a_{2,2})(\{i_1\}^*+\{i_2\}^*)\sim (a_{1,1}a_{2,2}-a_{1,2}a_{2,1})\{i_3\}^*\in\w H^0(K_I).
\end{align*}
Since $\w H^0(K_{W})\neq0$ implies that one of the two components of $K_{U\setminus\omega}$ is disjoint from $\{i_3\}$, it is easy to see that $\theta_I\cdot\xi_{U\setminus\omega}=0$ by choosing the representation of $\xi_{U\setminus\omega}$ to be a linear combination of dual simplices in this component of $K_{U\setminus\omega}$.
Combining this with the fact that $J_1\neq J_2$ and $J_i\cap (V\setminus\omega')=\emptyset$ for $i=1,2$, we obtain
\[p_{J_1\cup(V\setminus\omega')}\circ\phi(\theta_I\cdot\xi_{U\setminus\omega})=(a_{2,1}-a_{2,2})\cdot\xi_{J_1}'\cdot\xi_{V\setminus\omega'}'=0.\]
This is a contradiction since $a_{2,1}-a_{2,2}\neq0$ by assumption and $\xi_{J_1}'\cdot\xi_{V\setminus\omega'}'\neq0$.
Thus $k$ must equal to one and the proof is finished.
\end{proof}

Before giving the next lemma, we use Lemma \ref{lem:|I|=3} to prove Propsition \ref{prop:scc} \eqref{st:b}.
\begin{proof}[Proof of Propsition \ref{prop:scc} \eqref{st:b}]
Let $\xi_I\in \w H^0(K_I)$ be an element such that $f^*(\xi_I)=\w\omega$ for $f:K_\omega\to K_I$, and set $\xi_J'=\phi(\xi_I)\in\w H^0(K_J')$.
Suppose $I\setminus\omega=\{i_k\}$. Then by the assumption that $K$ satisfies the SCC there is a subset $U\subset \VV(K)$ such that $\omega\subset U$, $i_k\notin U$, $|K_U|\cong S^1$ and $\w H^0(K_{W})\neq0$ for $W=\{i_k\}\cup U\setminus\omega$.
Let $c_U$ be a generator of $\w H^1(K_U)\cong\kk$; $c_V'=\phi(c_U)\in\w H^1(K'_V)$ for some circuit $K'_V\subset K'$;
$\xi_{U\setminus\omega}$ be a generator of $\w H^0(K_{U\setminus\omega})\cong\kk$;
$\xi_W\in \w H^0(K_W)$ be an element such that $g^*(\xi_W)=\xi_{U\setminus\omega}$ for $g:K_{U\setminus\omega}\to K_W$.
Then $\xi_I\cdot\xi_{U\setminus\omega}=\w\omega\cdot\xi_W$ by Lemma \ref{lem:hoch}.

As we have seen in the proof of Lemma \ref{lem:|I|=3}, the element $\xi_{V\setminus\omega'}':=p_{V\setminus\omega'}\circ \phi({\xi_{U\setminus\omega}})$ is nonzero and $p_{U\setminus\omega}\circ\phi^{-1}(\xi_{V\setminus\omega'}')=\xi_{U\setminus\omega}$. Also, since $h^*(\xi_I\cdot\xi_{U\setminus\omega})=\w\omega\cdot\xi_{U\setminus\omega}\neq 0$ for $h:K_U\to K_{I\cup U}$, we have $\xi_I\cdot\xi_{U\setminus\omega}\neq 0$. Hence from
$p_{I\cup U}\circ\phi^{-1}(\xi_J'\cdot\xi_{V\setminus\omega'}')=\xi_I\cdot\xi_{U\setminus\omega}$,
it follows that $\xi_J'\cdot\xi_{V\setminus\omega'}'\neq0$.
Thus we have \[p_{J\cup(V\setminus\omega')}\circ\phi(\w\omega\cdot\xi_W)=p_{J\cup(V\setminus\omega')}\circ\phi(\xi_I\cdot\xi_{U\setminus\omega})=\xi_J'\cdot\xi_{V\setminus\omega'}'\neq0.\]
This implies that $\omega'\subset J$. Note that this argument is valid for all $\omega\in MF(K_I)$. Then we get the desired result.
\end{proof}
The following lemma is the inverese of Lemma \ref{lem:|I|=3}.
\begin{lem}\label{lem:inverse}
	Under the hypotheses of Proposition \ref{prop:scc}, if $\w H^0(K_J')\neq0$ for some $J\subset\VV(K')$, $|J|=3$, then $\phi^{-1}(\w H^0(K_J'))\subset\w H^0(K_{I})$ for some $I\subset\VV(K)$, $|I|=3$.
\end{lem}
\begin{proof}
Suppose $\phi^{-1}(\xi_J')=\sum_{j=1}^k\xi_{I_j}$, where $0\neq\xi_{I_j}\in\w H^0(K_{I_j})$, $I_i\neq I_j$ for $i\neq j$, $|I_j|=3$ for $1\leqslant j\leqslant k$.
Let $S=\{j\in[k]:\phi(\xi_{I_j})\not\in\w H^0(K_J')\}$. Then Lemma \ref{lem:|I|=3} says that $p_J\circ\phi(\xi_{I_j})=0$ for $j\in S$,  and $\phi(\xi_{I_j})\in \w H^0(K_J')$ for $j\in [k]\setminus S$. Hence $S=\emptyset$, and $\phi(\xi_{I_j})\in \w H^0(K_J')$ for all $1\leqslant j\leqslant k$.
It follows that $k\leqslant 2$ since $\dim_\kk\w H^0(K_J')\leqslant2$.
If $\w H^0(K_J')=\kk$ the statement  obviously holds.
So we only consider the case that 
$\w H^0(K_J')=\kk\oplus\kk$.
Suppose on the contrary that $\phi^{-1}(\xi_J')=\xi_{I_1}+\xi_{I_2}$ for some nonzero element $\xi'_J\in\w H^0(K_J')$. Since $I_1\neq I_2$, we have $|MF(K_{I_1})\cup MF(K_{I_2})|\geqslant3=|MF(K_J')|$.
On the other hand, \eqref{st:b} shows that  $\phi_\MM(MF(K_{I_1})\cup MF(K_{I_2}))\subset MF(K_J')$. It follows that $|MF(K_{I_1})|=|MF(K_{I_2})|=2$, hence $\w H^{0}(K_{I_1})=\w H^{0}(K_{I_2})=\kk$.

Suppose $J=\{j_1,j_2,j_3\}$. Then without loss of generality we may assume that  the cohomology classes $\phi(\xi_{I_i})$ have the form:
\[\phi(\xi_{I_i})=a_{i,1}\{j_1\}^*+a_{i,2}\{j_2\}^*\in\w H^0(K_J'),\ a_{i,1}\neq a_{i,2}\in\kk,\ i=1,2.\]
Let $\omega'=\{j_1,j_2\}$ and $\omega=\phi_\MM^{-1}(\omega')$. We may assume that $\omega\in  MF(K_{I_1})$.
Apply the SCC on $K$ to $\omega$ and $I_1\setminus\omega$, and use the same notations as in the proof of Lemma \ref{lem:|I|=3}, i.e., $|K_U|\cong S^1$, $\omega\subset U$, $I_1\cap U=\omega$, $W=(I_1\cup U)\setminus\omega$, $\phi(c_U)=c_V'\in\w H^1(K'_V)$, $\xi_{V\setminus\omega'}':=p_{V\setminus\omega'}\circ \phi({\xi_{U\setminus\omega}})$ etc. Note that for the missing face $\omega_0\in MF(K_{I_1})\setminus\{\omega\}$, $\w\omega_0$ is not a factor of $c_U$, so $\w\omega_0'=\phi(\w\omega_0)$ is not a factor of $c_V'$ either.
Hence $\omega_0'\not\subset V$, and it follows that $J\cap (V\setminus\omega')=\emptyset$ since $\omega_0'\in MF(K_J')$ by Proposition \eqref{st:b}. Let $f':K_{\omega'}'\to K'_J$ and $h':K_V'\to K_{J\cup V}'$ be the inclusions. Then $(f')^*(\phi(\xi_{I_1}))=a\w\omega'$ for some $0\neq a\in\kk$ since $a_{1,1}\neq a_{1,2}$. It follows that $(h')^*(\phi(\xi_{I_1})\cdot\xi_{V\setminus\omega'}')=a\w\omega'\cdot\xi_{V\setminus\omega'}'\neq0$, and so
$p_{J\cup V}\circ\phi(\xi_{I_1}\cdot\xi_{U\setminus\omega})=\phi(\xi_{I_1})\cdot\xi_{V\setminus\omega'}'\neq0$. Similarly, $p_{J\cup V}\circ\phi(\xi_{I_2}\cdot\xi_{U\setminus\omega})\neq 0$.
Thus $\xi_{I_i}\cdot\xi_{U\setminus\omega}\neq0$, $I_i\cap(U\setminus\omega)=\emptyset$ for $i=1,2$.

Take $\xi_W\in \w H^0(K_W)$ to be an element such that $g^*(\xi_W)=\xi_{U\setminus\omega}$ for $g:K_{U\setminus\omega}\to K_W$. Note that $f^*(\xi_{I_1})=\w\omega$ for $f:K_\omega\to K_{I_1}$.
Thus we have \[\xi_{I_1}\cdot\xi_{U\setminus\omega}=\xi_{I_1}\cdot g^*(\xi_W)=f^*(\xi_{I_1})\cdot\xi_W=\w\omega\cdot\xi_W\]
by Lemma \ref{lem:hoch}. So
\[0\neq p_{J\cup V}\circ\phi(\xi_{I_1}\cdot\xi_{U\setminus\omega})=p_{J\cup V}\circ\phi(\w\omega\cdot\xi_W)=\w\omega'\cdot p_{(J\cup V)\setminus\omega'}\circ\phi(\xi_W).\]
It follows that $\w H^0(K'_{(J\cup V)\setminus\omega'})\neq0$. This means that $\{j_3,v\}\in MF(K)$ for any vertex $v$ of one of the two components of $K_{V\setminus\omega'}'$.
Let
\begin{align*}
\theta_J'&=(a_{2,1}-a_{2,2})\phi(\xi_{I_1})-(a_{1,1}-a_{1,2})\phi(\xi_{I_2})\\
&=(a_{1,2}a_{2,1}-a_{1,1}a_{2,2})(\{j_1\}^*+\{j_2\}^*)\sim (a_{1,1}a_{2,2}-a_{1,2}a_{2,1})\{j_3\}^*\in\w H^0(K'_J).
\end{align*}
It follows by an easy calculation that $\theta_J'\cdot\xi'_{V\setminus\omega'}=0$.
Since $I_1\neq I_2$ and $I_2\cap(U\setminus\omega)=\emptyset$, we have $I_2\not\subset I_1\cup U$, and then
\[p_{I_1\cup U}\circ\phi^{-1}(\theta_J'\cdot\xi_{V\setminus\omega'}')=(a_{2,1}-a_{2,2})\xi_{I_1}\cdot\xi_{U\setminus\omega}=0.\]
The first equality comes from the fact that $p_{U\setminus\omega}\circ\phi^{-1}(\xi_{V\setminus\omega'}')=\xi_{U\setminus\omega}$. But $a_{2,1}-a_{2,2}$ and $\xi_{I_1}\cdot\xi_{U\setminus\omega}$ are both nonzero, a contradiction. The proof is finished.
\end{proof}
Note that if $\w H^0(K_I)\neq0$ and $|I|=3$, then $K_I$ is either $\Delta^0\sqcup\Delta^1$ or $\Delta^0\sqcup\Delta^0\sqcup\Delta^0$, depending on the rank of $\w H^0(K_I)$ is one or two. Hence combining Lemma \ref{lem:|I|=3} and Lemma \ref{lem:inverse} gives Propsition \ref{prop:scc} \eqref{st:a} for  $|I|=3$, completing the proof of Propsition \ref{prop:scc}.

\section{Proof of Lemma \ref{lem:1-1} and Lemma \ref{lem:link}}\label{app:proof of two lemmata}
\begin{proof}[Proof of Lemma \ref{lem:1-1}]
It is clear that $\phi_\MM\mid_{MF(K_I)}$ is an injection, so it remains to prove that it is surjective.
Assuming $MF(K)=\{\omega_1,\dots,\omega_k\}$ and $\phi_\MM(\omega_i)=\omega_i'$ for $1\leqslant i\leqslant k$, first we claim that for any $\omega_{i}'\in\phi_\MM(MF(K_I))$,
if there is an $\omega_{j}'\in MF(K_J')$ ($\omega_j\neq\omega_i$) such that $\omega_{i}'\cap\omega_{j}'\neq\emptyset$,
then $\omega_{j}'\in\phi_\MM(MF(K_I))$.

Clearly, $\w\omega_i$ is a factor of $\xi_I\in \w H^p(K_I)=\kk$,
so $\w\omega_{i}'$ is a factor of $\xi_J'=\phi(\xi_I)$ and  $\xi_J'=\w\omega_{i}'\cdot\xi_{J\setminus\omega_i'}'$ for some $\xi_{J\setminus\omega_i'}'\in\w H^*(K_{J\setminus\omega_{i}'}')$.
Set $W=\omega_{i}'\cup\omega_{j}'$.
It is easy to see that there is an element $\xi_W'\in\w H^0(K_W')$ such that $f^*(\xi_W')=\w\omega_{i}'$ for $f:K_{\omega_i'}'\to K_W'$.
Let $U=J\setminus W$ and let $\xi_U'=g^*(\xi_{J\setminus\omega_i'}')$ for $g:K_U'\to  K_{J\setminus\omega_{i}'}'$.
By Lemma \ref{lem:hoch}, $\xi_W'\cdot \xi_U'=\w\omega_i'\cdot\xi_{J\setminus\omega_i'}'=\xi_J'$. Thus $\xi_I=\phi^{-1}(\xi_W')\cdot\phi^{-1}(\xi_U')$.
The assumption that $\omega_{i}'\neq\omega_{j}'$ and $\omega_{i}'\cap\omega_{j}'\neq\emptyset$ implies that $|W|=3$, so according to Proposition \ref{prop:scc} \eqref{st:a},
$\phi^{-1}(\xi_W')\in\w H^0(K_{I_0})$ for some $I_0\subset I$, $|I_0|=3$. Then by Proposition \ref{prop:scc} \eqref{st:a} and \eqref{st:b} we have \[\omega_{j}'\in\phi_\MM(MF(K_{I_0}))\subset\phi_\MM(MF(K_I)).\]
Hence the claim is true. As a consequence, if we define two subsets of $\VV(K')$ as
\[V_1:=\bigcup_{\omega_i'\in\phi_\MM(MF(K_I))}\omega_i',\quad \ V_2:=\bigcup_{\omega_i'\in MF(K_J')\setminus\phi_\MM(MF(K_I))}\omega_i',\]
then $V_1\cap V_2=\emptyset$, and $K'_J=K_{V_1}'*K_{V_2}'$ since $K'$ is flag and $K_J'=\mathrm{core}\,K_J'$.

Next we show that for any $\omega_i\in MF(K)$ such that $\omega_i'\in MF(K'_{V_2})$, we have $\omega_i\cap I=\emptyset$.
If not, $|\omega_i\cap I|=1$ since $\omega_i\not\in MF(K_I)$, and we may assume $\omega_i\cap I=\{i_1\}$.
Let $\omega_j\in MF(K_I)$ be such that $i_1\in\omega_j$ (such $\omega_j$ always exists since $K_I=\mathrm{core}\,K_I$). Then $|\omega_i\cup\omega_j|=3$ and $\w H^0(K_{\omega_i\cup\omega_j})\neq0$.
So by Corollary \ref{cor:intersect} we have $\omega_i'\cap\omega_j'\neq\emptyset$.
However the fact that $\omega_j'\subset V_1$, $\omega_i'\subset V_2$ and $V_1\cap V_2=\emptyset$ gives a contradiction.

Finally, we prove the lemma by showing that $V_2=\emptyset$.
Suppose on the contrary that there is an $\omega_i'\in MF(K'_{V_2})$.
Then the missing face $\omega_i\in MF(K)$ satisfies $\omega_i\cap I=\emptyset$ as we showed in the previous paragraph.
Suppose $\omega_i=\{i_1,i_2\}$, then $K_I\not\subset\mathrm{lk}_K\{i_1\}\cap\mathrm{lk}_K\{i_2\}$ by Proposition \ref{prop:nsc} since $\w H^p(K_I)\neq0$ and $p\geqslant n-2$.
Thus without loss of generality we may assume that $K_I\not\subset \mathrm{lk}_K\{i_1\}$.
This implies that there is a vertex $i_3\in I$ such that
$\omega_k:=\{i_1,i_3\}\in MF(K)$. Clealy $|\omega_i\cup\omega_k|=3$ and $\w H^0(K_{\omega_i\cup\omega_k})\neq0$.
Hence it follows by Proposition \ref{prop:scc} \eqref{st:b} that $|\omega_i'\cup\omega_k'|=3$.
This implies that $\emptyset\neq\omega_k'\cap\omega_i'\subset V_2$, and then $\omega_k'\cap V_1=\emptyset$ since $K'_J=K_{V_1}'*K_{V_2}'$.
Letting $\omega_l\in MF(K_I)$  be such that $i_3\in\omega_l$
and applying Proposition \ref{prop:scc} \eqref{st:b} again for $K_{\omega_k\cup\omega_l}$ we have $|\omega_k'\cup\omega_l'|=3$, i.e., $\omega_k'\cap\omega_l'\neq\emptyset$.
However this contradicts the fact that $\omega_l'\subset V_1$ and $\omega_k'\cap V_1=\emptyset$. Thus the proof is finished.
\end{proof}

\begin{proof}[Proof of Lemma \ref{lem:link}]
Suppose $\VV(K)=[m]$, $\dim K=n$, and suppose $L=\mathrm{lk}_K\{m\}$, $\VV(L)=[l]$.
$L$ is clearly a $\kk$-homology sphere of dimension $n-1$.
So from Theorem \ref{thm:union product} we have $[L]\in H^{l+n}(\ZZ_K)$.
Let $R=H^*(\ZZ_K)$, $R'=H^*(\ZZ_{K'})$.
For the first step, we will show that there are at most  $m-l-1$ linear independent elements $\xi_i$ in $ R^{l+n+1}\cap(R^+)^{*n}$ such that for every $\omega\in MF(L)$, $\w\omega$ is a factor of each $\xi_i$. Moreover, the upper bound $m-l-1$ can be reached.

Suppose $0\neq\xi\in R^{l+n+1}\cap (R^+)^{*n}$. If there is a subset $I\subset[m]$ such that $0\neq p_{I}(\xi)\in\w H^{p}(K_{I})$, then $p\geqslant n-1$ by Lemma \ref{lem:upper bound of nil}. But $\w H^n(K_I)\neq0$ if and only if $I=[m]$, so $p=n-1$ since $\w H^n(K)=(R^+)^{*(n+1)}$, hence $|I|=l+1$.
If in addition $\w\omega$ is a factor of $\xi$ for every $\omega\in MF(L)$, then $\VV(L)\subset I$.
Actually there are exactly $m-l-1$ subsets satisfying these conditions:
\[I_j=\VV(L)\cup\{j\},\ l+1\leqslant j\leqslant m-1.\]
It is easily verified that $L\cap \mathrm{lk}_K\{j\}=\mathrm{lk}_{K_{I_j}}\{j\}$ and $\{j,m\}\in MF(K)$ for $l+1\leqslant j\leqslant m-1$.
Since $\mathrm{lk}_{K_{I_j}}\{j\}\subset L$, $\w H^{n-1}(\mathrm{lk}_{K_{I_j}}\{j\})\neq 0$ only if $\mathrm{lk}_{K_{I_j}}\{j\}=L$, but this means that the subcomplex $\{\{j\},\{m\}\}*L\subset K$ is already an $n$-dimensional homology sphere, and then $K=\{\{j\},\{m\}\}*L$ would be a suspension.
Thus by Proposition \ref{prop:nsc}
\[\w H^{n-2}(\mathrm{lk}_{K_{I_j}}\{j\})=\w H^{n-1}(\mathrm{lk}_{K_{I_j}}\{j\})=0\quad \text{for}\quad l+1\leqslant j\leqslant m-1.\]
 From the Mayer-Vietoris sequence
\begin{align*}
\cdots\to\w H^{n-2}(\mathrm{lk}_{K_{I_j}}\{j\})\to\w H^{n-1}(K_{I_j})&\to\w H^{n-1}(\mathrm{st}_{K_{I_j}}\{j\})\oplus\w H^{n-1}(L)\\
&\to\w H^{n-1}(\mathrm{lk}_{K_{I_j}}\{j\})\to\cdots
\end{align*}
we immediately obtain that $\w H^{n-1}(K_{I_j})\cong\w H^{n-1}(L)=\kk$, and the upper bound follows since $\xi\in\bigoplus_{l+1\leqslant j\leqslant m-1} \w H^{n-1}(K_{I_j})$.

Take a nonzero element $\xi_{I_j}\in\w H^{n-1}(K_{I_j})$ for each $l+1\leqslant  j\leqslant m-1$.
We will show that each $\xi_{I_j}$ satisfies the condition in the first paragraph.
For any $\omega\in MF(L)$, let $U=[l]\setminus\omega$ and let $W_j=U\cup\{j\}$ for $l+1\leqslant j\leqslant m-1$.
Since $\mathrm{lk}_{K_{W_j}}\{j\}\subset\mathrm{lk}_K\{j\}\cap L$,
then by applying Proposition \ref{prop:nsc} again, we have $\w H^{n-2}(\mathrm{lk}_{K_{W_j}}\{j\})=0$.
Now from the  Mayer-Vietoris sequence
\[
\cdots\to\w H^{n-2}(K_{W_j})\to\w H^{n-2}(\mathrm{st}_{K_{W_j}}\{j\})\oplus\w H^{n-2}(K_{U})\to\w H^{n-2}(\mathrm{lk}_{K_{W_j}}\{j\})\to\cdots
\]
it follows that $\w H^{n-2}(K_{W_j})\xr{h^*} \w H^{n-2}(K_U)$ is a surjection for $h:K_U\ha K_{W_j}$.
Since $\w\omega$ is a factor of $[L]$, there is an element $\xi_U\in\w H^{n-2}(K_U)$ such that $\w\omega\cdot\xi_U=[L]$.
Letting $\xi_{W_j}\in\w H^{n-2}(K_{W_j})$ be such that $h^*(\xi_{W_j})=\xi_U$, we have $g^*(\w\omega\cdot\xi_{W_j})=\w\omega\cdot\xi_U=[L]$ for $g:L\ha K_{I_j}$.
So $0\neq\w\omega\cdot\xi_{W_j}\in\w H^{n-1}(K_{I_j})\cong\kk$, i.e., $\w\omega$ is a factor of $\xi_{I_j}$. It remains to show that $\xi_{I_j}\subset (R^+)^{*n}$. Since $L$ is Gorenstein* of dimension $n-1$,  $\w H^{n-2}(\mathrm{lk}_L\{i\})=\kk$ for any $i\in[l]$. So $\mathrm{lk}_L\{i\}\not\subset \mathrm{lk}_K\{j\}$ for any $l+1\leqslant j\leqslant m-1$ by Proposition \ref{prop:nsc} and the fact that $\mathrm{lk}_L\{i\}\subset \mathrm{lk}_K\{m\}$ is a full subcomplex of $K$. 
Hence there exits $\{k\}\in \mathrm{lk}_L\{i\}$ such that $\{j,k\}\in MF(K)$, because $K$ is flag. 
Take a facet $\sigma\in L$ containing $k$. Let $\tau=\sigma\setminus\{k\}$, and let $\omega$ be the missing face corresponding to $\mathrm{lk}_L\tau\cong S^0$. Then for $W=\{j\}\cup\omega$ and $q:K_\omega\ha K_W$, $q^*(\xi_W)=\w\omega$ for some $\xi_W\in\w H^0(K_W)$. 
From the proof of Lemma \ref{lem:nil of gorenstein}, we see that (by an induction on the dimension of $L$) for such $\omega$, there exist $\alpha_1,\dots,\alpha_{n-1}\in H^{+}(\ZZ_{L})$ such that $\w\omega\cdot\alpha_1\cdots\alpha_{n-1}=[L]$. Hence for $g:L\ha K_{I_j}$, we have $g^*(\xi_W\cdot\alpha_1\cdots\alpha_{n-1})=\w\omega\cdot\alpha_1\cdots\alpha_{n-1}=[L]$,
which means that $\xi_W\cdot\alpha_1\cdots\alpha_{n-1}\in\w H^{n-1}(K_{I_j})\subset (R^+)^{*n}$, and we conclude that the $m-l-1$ elements $\xi_{I_j}$ satisfy the required condition.

We have already seen from Corollary \ref{cor:goren} that $|\VV(K')|=|\VV(K)|=m$, and from Lemma \ref{lem:gorenstein subcomplex} that $\phi([L])\in\w H^{n-1}(K_J')$ for some $J\subset\VV(K')$, $|J|=l$.
We will show that if $K'_J$ is not the link of any vertex of $K'$, then for every
vertex $i\in\VV(K')\setminus J$, there exists a nonzero element $\xi_{J_i}'\in\w H^{n-1}(K'_{J_i})$ for $J_i=\{i\}\cup J$,
such that $\w\omega'$ is a factor of $\xi'_{J_i}$ for every $\omega'\in MF(K_J')=\phi_\MM(MF(L))$ (by Lemma \ref{lem:1-1}), and $\xi_{J_i}'\in (R'^+)^{*n}$.
This will give a contradiction because any isomorphism preserves algebraic properties.

To see this,
note that $K_J'$ is not a proper subcomplex of  $\mathrm{lk}_{K'}\{i\}$ for any $i\in \VV(K')\setminus J$ since $\w H^{n-1}(K'_J)\neq 0$. So for an arbitrary $i\in \VV(K')\setminus J$, if $K'_J\neq\mathrm{lk}_{K'}\{i\}$, then $J\setminus\VV(\mathrm{lk}_{K'}\{i\})\neq\emptyset$. Choose a vertex $j\in J\setminus\VV(\mathrm{lk}_{K'}\{i\})$.
Since $K'_J=\mathrm{core}\,K'_J$, we can choose an $\omega'\in MF(K'_J)$ such that $j\in\omega'$.
Let $W=\{i\}\cup\omega'$ and let $\xi_W'\in\w H^0(K'_W)$ be such that $q^*(\xi_W')=\w\omega'$ for $q:K'_{\omega'}\ha K'_W$. According to Proposition \ref{prop:scc} \eqref{st:a}, $\phi^{-1}(\xi'_W)=\xi_U\in\w H^0(K_U)$ for some $U\subset[m]$, $|U|=3$. Moreover, from Proposition \ref{prop:scc} \eqref{st:b} and Lemma \ref{lem:1-1} we  deduce that $U\cap[l]=\phi_\MM^{-1}(\omega')=:\omega$.
Since $\w\omega$ is a factor of $[L]$, there is an element $\xi_V\in\w H^{n-2}(K_V)$ for $V=[l]\setminus\omega$, such that $\w\omega\cdot\xi_V=[L]$, and it is easy to see that
\[p_{J_i}\circ\phi(\xi_U\cdot\xi_V)=p_{J_i}(\xi_W'\cdot\phi(\xi_V))=\xi_W'\cdot p_X\circ\phi(\xi_V),\]
in which $X=J_i\setminus W=J\setminus\omega'$. Let $f:K_J'\ha K'_{J_i}$. We also have
\[f^*(\xi_W'\cdot p_X\circ\phi(\xi_V))=\w\omega'\cdot p_X\circ\phi(\xi_V)=p_J\circ\phi(\w\omega\cdot\xi_V)=\xi'_J\neq0.\]
These formulae imply that $p_{J_i}\circ\phi(\xi_U\cdot\xi_V)\neq0$.

Suppose $U\setminus\omega=\{k\}$ (clearly $l<k<m$). Since $\xi_U\cdot\xi_V\in\w H^{n-1}(K_{I_k})$ (remember $I_k=\{k\}\cup[l]$), $\w\omega$ is a factor of $\xi_U\cdot\xi_V$ for every $\omega\in MF(L)$ as we have seen.
It follows that for every $\omega'\in \phi_\MM(MF(L))=MF(K'_J)$, $\w\omega'$ is a factor of $\phi(\xi_U\cdot\xi_V)$, and so it is also a factor of $p_{J_i}\circ\phi(\xi_U\cdot\xi_V)\in\w H^{n-1}(K'_{J_i})$. Thus we get the desired element $\xi_{J_i}':=p_{J_i}\circ\phi(\xi_U\cdot\xi_V)$, noting that $\xi_U\cdot\xi_V\in (R^+)^{*n}$.
The proof of Lemma \ref{lem:link} is finished.
\end{proof}

\section{Construction of flag spheres satisfying the SCC}\label{app:scc}
The aim of this appendix is to construct higher dimensional simplicial spheres, which are flag and satisfy the SCC, from lower dimensional simplicial spheres.
\begin{const}\label{const:SCC}
Given an $n\geqslant3$ dimensional simplicial polytope $P$ with $m$ vertices, let $P^*$ be the dual simple polytope, and for each $q$-face $\sigma$ of $P$, let $F_\sigma$ be the dual $n-q-1$ face of $P^*$. Suppose there is a triangulation $T$ of the $(n-2)$-skeleton of $P^*$ such that the vertex sets $\VV(T)=\VV(P^*)$.
For each face $\sigma$ of $P$ define a subcomplex $L_\sigma$ of $T$ to be
\[L_\sigma:=
\begin{cases}
\{\tau\in T: |T_\tau|\subset F_\sigma\}\quad &\text{if } \dim\sigma>0,\\
\{\tau\in T: |T_\tau|\subset \partial F_\sigma\}\quad &\text{if }\dim\sigma=0.
\end{cases}\]
In other words, $L_\sigma\subset T$ is the triangulation of $F_\sigma$ for $\dim\sigma>0$ and $L_i$ is the triangulation of $\partial F_i$ for each vertex $i\in[m]$.
Let $K,K'$ be two copies of $\partial P$, and let $\{v_i\}_{i=1}^m$ and $\{v_i'\}_{i=1}^m$ be their vertex sets respectively. Here $v_i$ and $v_i'$ correspond to the same vertex $i\in[m]$, and this correspondence gives a bijection $K\to K'$, $\sigma\mapsto\sigma'$.
Define a simplicial complex 
\[L:=(\bigcup_{i=1}^m 2^{\{v_i,v_i'\}}*L_i)\cup(\bigcup_{\sigma\in K}2^{\sigma}*L_{\sigma})\cup(\bigcup_{\sigma'\in K'}2^{\sigma'}*L_{\sigma}).\]
It can be shown that $L$ is a triangulation of $S^{n-1}\times D^1$, and the two components of the boundary $\partial L$ are $K$ and $K'$. To see this, we start with the cartesian product $\Gamma_0=K\times \Delta^1$ with a polyhedral face structure induced by the faces of $K$ and $\Delta^1$. Then for each facet $\sigma$ of $K$, we repace the polyhedron $2^\sigma\times\Delta^1\subset \Gamma_0$ with $\partial(2^\sigma\times\Delta^1)*F_\sigma$ to get $\Gamma_1$. It is easy to see that $|\Gamma_1|\cong |\Gamma_0|$, and $(2^\tau\times \Delta^1)*\partial L_\tau\subset \Gamma_1$ for each $\tau\in K$ of codimension one (note that $\partial L_\tau$ is the two vertices of the edge $F_\tau$). Inductively, if we construct $\Gamma_{i}$ ($i<n-1$), which means that $(2^\tau\times \Delta^1)*\partial L_\tau\subset \Gamma_i$ for each $\tau\in K$ of codimension $i$, then we replace these polyhedra  $(2^\tau\times \Delta^1)*\partial L_\tau$ with $\partial(2^\tau\times \Delta^1)*L_\tau$ to get $\Gamma_{i+1}$ (clearly $|\Gamma_{i+1}|\cong|\Gamma_i|$). Finally, we get the desired triangulation $\Gamma_{n-1}=L$ of $S^{n-1}\times D^1$. Now adding two more vertices $u$ and $u'$, and define a triangulation of $S^{n}$ by
\[\EE_P:=(\{u\}*K)\cup L\cup (\{u'\}*K').\]
\end{const}

\begin{prop}\label{prop:exmp}
In the notation of Construction \ref{const:SCC}, suppose the simplicial $(n-1)$-sphere $\partial P$ is flag and not a suspension, and suppose there is a flag triangulation $T$ of the $(n-2)$-skeleton of $P^*$ such that $\VV(T)=\VV(P^*)$. Then $\EE_P$ is flag and satisfies the SCC.
\end{prop}
\begin{rem}
(a) Note that if $n=3$ in Proposition \ref{prop:exmp}, then the assumption that $T$ is flag is automatically satisfied. Since in this case the $1$-skeleton of $(P^*)$ is clearly a simplicial complex, and from the assumption that $\partial P$ is flag, we can readily deduce that there are no $3$-circuits in $T$.

(b) For $n=4$, there are also infinitely many examples satisfying the hypothesis of Proposition \ref{prop:exmp}.
For example, let $P_1$, $P_2$ be two polygons with $n_1$ and $n_2$ ($n_1,n_2\geqslant 5$) vertices, respectively, and let $P$ be the dual of $P_1\times P_2$. Then $\partial P=\partial P_1*\partial P_2$ is obviously flag and not a suspension.
There is an unique triangulation $K_1$ of $P_1$ (and similarly $K_2$ of $P_2$) determined by the following rules: $\VV(K_1)=\VV(P_1)=[n_1]$; $\{\{i,i+1\}:4\leqslant i\leqslant n_1\}\cup\{\{1,2\},\{1,3\},\{2,4\},\{3,n_1\}\}$ is the set of edges of $\partial K_1$; $\{\{1,i\}:4\leqslant i\leqslant n_1\}$ is the set of edges in $K_1\setminus\partial K_1$.

Let $K$  be the canonical triangulation of $K_1\times K_2$, i.e.,
\[\begin{split}
	K:=&\{\{(i_1,j_1),\dots,(i_k,j_k)\}:i_1\leqslant\cdots\leqslant i_k,\ j_1\leqslant\cdots\leqslant j_k,\\
&\{i_1,\dots,i_k\}\in K_1,\ \{j_1,\dots,j_k\}\in K_2\}.
\end{split}\]
Let $T$ be the restriction of $K$ to the $2$-skeleton of $P_1\times P_2$. 
We claim that $T$ is flag, hence by Proposition \ref{prop:exmp}, $\EE_P$, as a triangulation of $S^4$, is flag and satisfies the SCC. The claim can be proved as follows. Suppose on the contrary that  $\sigma=\{(i_1,j_1),\dots,(i_k,j_k)\}$, $k\geqslant 3$, is a missing face of $T$. From the definitions of $K$ and $\sigma$, one easily sees that any two-elements subset of $\{i_1,\dots,i_k\}$ or $\{j_1,\dots,j_k\}$ is in $K_1$ or $K_2$. Thus the flagness of $K_1$ and $K_2$ shows that $\sigma_1=\{i_1,\dots,i_k\}\in K_1$ and $\sigma_2=\{j_1,\dots,j_k\}\in K_2$. It follows that $|\sigma_1|,|\sigma_2|\neq 1$, since $\sigma\not\in T$. Furthermore, since any proper subset of $\sigma$ is a simplex of $T$, we may assume that $i_1\leqslant\cdots\leqslant i_k$, $j_1\leqslant\cdots\leqslant j_k$  by the construction rule of $K$.  Hence we have $\{i_1,i_k\}\in \partial K_1$ and $\{j_1,j_k\}\in\partial K_2$, since otherwise the edge $\{(i_1,j_1),(i_k,j_k)\}$ is not in the $2$-skeleton of $P_1\times P_2$ for $|\sigma_1|,|\sigma_2|\geqslant 2$. This implies that $\sigma_1\in\partial K_1$ and $\sigma_2\in\partial K_2$ by the construction  rule of $K_1$ and $K_2$, but then $\sigma\subset\sigma_1\times\sigma_2$ is a face of $T$, a contradiction. So $T$ is actually flag.

For a general flag simple $n$-polytope $P$ ($n>3$), we do not know whether there is a flag triangulation $T$ of the $(n-2)$-skeleton of $P$ such that $\VV(T)=\VV(P)$. Here we propose a more general question:
\begin{que}
If $P$ is a simple flag $n$-polytope, is there always a flag triangulation $T_k$ of the $k$-skeleton of $P$ such that
$\VV(T_k)=\VV(P)$ for all $n>3$ and $2\leqslant k\leqslant n-1$?
\end{que}

(c) From Proposition \ref{prop:4-circuit} and (a), (b) above, we see that for $n=2,3,4$, there are infinitely many simplicial $n$-spheres which are flag and satisfy the SCC. Hence by  Example \ref{exmp:join}, we can actually construct infinitely many simplicial $n$-spheres of this type for all $n\geqslant2$.
\end{rem}

\begin{proof}[Proof of Proposition \ref{prop:exmp}]
First we check that $\EE_P$ is flag. Suppose on the contrary that $\omega$ is a missing face of $\EE_P$ with $|\omega|\geqslant 3$.
It follows from the construction of $\EE_P$ that $u,u'\not\in\omega$. The assumption that $T$ is flag implies that $\omega\not\subset\VV(T)$.
So we may assume that $\omega\cap\VV(K)\neq\emptyset$, by the symmetry of $\EE_P$.
If $|\omega\cap\VV(K)|=1$, say $\omega\cap\VV(K)=\{v_i\}$, let $\sigma=\omega\setminus\{v_i\}$. Then $u\not\in\sigma\in \EE_P$, $|\sigma|\geqslant 2$ and $\sigma\cap\VV(K)=\emptyset$.  Since any proper face of $\sigma$ belongs to $\mathrm{lk}_{\EE_P}\{v_i\}$, and  $\VV(\mathrm{lk}_{\EE_P}\{v_i\})\setminus(\{u\}\cup\VV(K))=\{v_i'\}\cup\VV(L_i)$, it follows that $\sigma\in(\EE_P)_{\{v_i'\}\cup\VV(L_i)}\subset\mathrm{lk}_{\EE_P}\{v_i\}$.
But this gives that $\omega\in\EE_P$, a contradiction.
On the other hand, if $|\omega\cap \VV(K)|\geqslant2$, then the fact that $\{v_i,v_j'\}\in\EE_P$ if and only if $i=j$ implies that $\omega\cap\VV(K')=\emptyset$. Setting $\sigma=\omega\cap\VV(T)$ and $\tau=\omega\cap \VV(K)$ (clearly $\emptyset\neq\sigma\in\EE_P$ since $K$ is flag), we have $\sigma\in L_i$ for any $v_i\in\tau$.
Thus $\sigma\in L_{\tau}$, and then $\omega=\sigma\cup\tau\in\EE_P$ gives a contradiction again.

Next we verify that $\EE_P$ satisfies the SCC case by case. Suppose $\{i_1,i_2\}\in MF(\EE_P)$ and $i_k\in\VV(\EE_P)\setminus\{i_1,i_2\}$.
We need to find $I\subset\VV(\EE_P)$ such that $|(\EE_P)_I|\cong S^1$, $i_1,i_2\in I$, $i_k\notin I$ and $\w H_0((\EE_P)_{J})\neq0$ for $J=\{i_k\}\cup I\setminus\{i_1,i_2\}$.

\begin{enumerate}[(i)]
\setlength{\itemsep}{1em}

\item $\{i_1,i_2\}=\{u,u'\}$.

\begin{enumerate}[(\theenumi a)]\setlength{\itemsep}{1ex}
\item $i_k\in \VV(K)$ (or $\VV(K')$). Since $K$ is flag and not a suspension, there exists a vertex $v_i\in\VV(K)\setminus\VV(\mathrm{st}_K\{i_k\})$ such that $ \mathrm{lk}_K\{v_i\}\neq\mathrm{lk}_K\{i_k\}$.  Moreover, since $\mathrm{lk}_K\{v_i\}$ and $\mathrm{lk}_K\{i_k\}$ are full subcomplexes by the flagness of $K$, there exists a vertex $v_j\in\VV(\mathrm{lk}_K\{i_k\})\setminus\VV(\mathrm{lk}_K\{v_i\})$, using the fact that $\mathrm{lk}_K\{i_k\}$ is not a proper subcomplex of $\mathrm{lk}_K\{v_i\}$. Hence  $\{v_i,v_j\},\{v_i,i_k\}\in MF(K)$. It is easy to verify that $I=\{u,v_i,v_j,v_i',v_j',u'\}$ satisfies the above condition.

\item\label{case:1b}  $i_k\in\VV(T)$. Choose a missing face $\{v_i,v_j\}\in MF(K)$. Note that $L_i\cap L_j=\emptyset$, so $\{v_i,i_k\}, \{v_i',i_k\}\in MF(\EE_P)$ or $\{v_j,i_k\},\{v_j',i_k\}\in MF(\EE_P)$. Thus we can take $I=\{u,v_i,v_j,v_i',v_j',u'\}$.\vspace{10pt}
\end{enumerate}

\item $i_1=u$, $i_2=v_i'\in \VV(K')$.

\begin{enumerate}[(\theenumi a)]\setlength{\itemsep}{1ex}
\item $i_k=u'$. Choose a vertex $v_j'$ such that $\{v_i',v_j'\}\in MF(K')$.
Suppose $K'_U$ is a path connecting $v_i'$ and $v_j'$ in $K'$, i.e., $|K'_U|\cong D^1$ with end points $v_i',v_j'$.
Note that such a full subcomplex always exists, since if we assume every edge of $K'$ has length $1$, then the shortest path connecting $v_i'$ and $v_j'$ in $K'$ is a full subcomplex. Let $I=\{u,v_i,v_j\}\cup U$. It is easy to check that $I$ is the desired subset.
\item \label{case:2b} $i_k\in\VV(K')$. Since $K'$ is not a suspension, there exists a vertex $v_j'\neq i_k$, such that $\{v_i',v_j'\}\in MF(K')$. Then let $I=\{u,v_i,v_j,v_i',v_j',u'\}$.
\item $i_k\in\VV(T)$. As in \eqref{case:1b} we may choose a missing face $\{v_i,v_j\}\in MF(K)$ and take $I=\{u,v_i,v_j,v_i',v_j',u'\}$.
\item\label{case:2d} $i_k\in\VV(K)$.  First we consider the case $i_k\neq v_i$.  Since $K$ is not a suspension, we can take a vertex $v_j\neq v_i$ such that $\{v_j,i_k\}\not\in K$. If $\{v_i,v_j\}\not\in K$, then $I=\{u,v_i,v_j,v_i',v_j',u'\}$ is the desired subset. Otherwise, take $\sigma$ to be a facet of $K$ containing $\{v_i,v_j\}$, and take a vertex $v_s\neq i_k$ such that $\{v_j,v_s\}\not\in K$ (such a $v_s$ exists because $K$ is not a suspension). Recall that $F_\sigma\in \VV(T)$ is the vertex dual to $\sigma$, and note that $\{i_k,F_\sigma\}\not\in\EE_P$ since $i_k\not\in\sigma$. So we can take
    \[
    I=
    \begin{cases}
    \{u,v_j,v_s,F_\sigma, v_i',v_s'\}\quad&\text{if }\{v_i',v_s'\}\in K',\\
    \{u,v_j,v_s,F_\sigma,v_i',v_s',u'\}\quad&\text{if }\{v_i',v_s'\}\not\in K'.\\
    \end{cases}
    \]
    
On the other hand, if $i_k=v_i$, take $\sigma\in K$ to be a facet containing $v_i$, and let $\tau=\sigma\setminus\{v_i\}$. Since $K$ is flag and  not a suspension,
there exists a vertex $v_j\in\VV(K)\setminus\VV(\text{lk}_\tau K)$ such that $\{v_i,v_j\}\not\in K$. Furthermore, there must be a vertex $v_s\in\tau$ such that $\{v_j,v_s\}\not\in K$ as $\{v_j\}\not\in\text{lk}_\tau K$ and $K$ is flag.
Note that $\{F_\sigma,v_i'\},\{F_\sigma,v_s\}\in \EE_P$ and $\{F_\sigma,v_j\},\{F_\sigma,v_j'\}\notin \EE_P$. Then we can take $I=\{u,v_j,v_s,F_\sigma,v_i',v_j',u'\}$.
\end{enumerate}

\item $i_1=u$, $i_2\in \VV(T)$. Let $\sigma\in K$ and $\sigma'\in K'$ be the facets dual to $i_2$.

\begin{enumerate}[(\theenumi a)]\setlength{\itemsep}{1ex}
\item\label{case:3a} $i_k=u'$. Choose two vertices $v_i,v_j\in\sigma$. Take $v_s$ to be a vertex such that $\{v_i,v_s\}\in MF(K)$, and take $K'_U$ to be a shortest path connecting $v_j'$ and $v_s'$ with $U\cap\sigma'=\{v_j'\}$. Such $U$ exists since $K'_V$ is path-connected for $V=\VV(K')\setminus\{\sigma'\setminus\{v_j'\}\}$.
It is easily verified that $I=\{u,v_i,v_s,i_2\}\cup U$ is the desired subset.
 
\item\label{case:3b} $i_k\in \VV(K')$. Assume $i_k=v_k'$.
Choose two vertices $v_i,v_j\in\sigma$ such that $v_i,v_j\neq v_k$.
Since $K$ is not a suspension, there exists a vertex $v_s$ such that $\{v_i,v_s\}\in MF(K)$ and $v_s\neq v_k$. Let
\[
I=
\begin{cases}
\{u,v_i,v_s,i_2,v_j',v_s'\}\quad&\text{if }\{v_j',v_s'\}\in K',\\
\{u,v_i,v_s,i_2,v_j',v_s',u'\}\quad&\text{if }\{v_j',v_s'\}\not\in K'.\\
\end{cases}
\]

\item $i_k\in \VV(T)$. Let $\tau\in K$ be the facet dual to $i_k$.
Choose two vertices $v_i,v_j\in\sigma$ with $v_i\not\in\tau$, and take $v_s$ to be a vertex such that $\{v_i,v_s\}\in MF(K)$.
Then we can take
\[
I=
\begin{cases}
\{u,v_i,v_s,i_2,v_j',v_s'\}\quad&\text{if }\{v_j',v_s'\}\in K',\\
\{u,v_i,v_s,i_2,v_j',v_s',u'\}\quad&\text{if }\{v_j',v_s'\}\not\in K'.\\
\end{cases}
\]

\item $i_k\in \VV(K)$. If $i_k\not\in\sigma$, we can take a vertex $v_s\in\sigma$ such that $\{v_s,i_k\}\in MF(K)$,
and take a vertex $v_j\neq i_k$ such that $\{v_j,v_s\}\in MF(K)$, since $K$ is flag and not a suspension;
or if $i_k\in\sigma$, then by the same reasoning as in the second paragraph of \eqref{case:2d} there exist $v_j\not\in\sigma$ and $v_s\in\sigma\setminus\{i_k\}$ such that $\{i_k,v_j\},\{v_j,v_s\}\in MF(K)$.
In either case, choose a vertex $v_i\in\sigma\setminus\{i_k,v_s\}$ ($\dim\sigma\geqslant 2$), and let
\[
I=
\begin{cases}
\{u,v_j,v_s,i_2,v_i',v_j'\}\quad&\text{if }\{v_i',v_j'\}\in K',\\
\{u,v_j,v_s,i_2,v_i',v_j',u'\}\quad&\text{if }\{v_i',v_j'\}\not\in K'.\\
\end{cases}
\]
\end{enumerate}

\item $i_1\in \VV(K)$, $i_2\in \VV(K')$. Assume $i_1=v_i$, $i_2=v_j'$ (clearly $i\neq j$).

\begin{enumerate}[(\theenumi a)]\setlength{\itemsep}{1ex}
	
\item $i_k=u$. If $\{v_i,v_j\}\in MF(K)$, take $K_U$ to be a shortest path connecting $v_i$ and $v_j$, and let $I=\{u',v_i',v_j'\}\cup U$.
On the other hand, if $\{v_i,v_j\}\in K$, let $I=\{v_i,v_j,v_i',v_j'\}$.

\item $i_k\in\VV(K')$ (or $\VV(K)$). If $i_k\neq v_i'$, then let
\[
I=
\begin{cases}
\{v_i,v_j,v_i',v_j'\}\quad&\text{if }\{v_i,v_j\}\in K,\\
\{u,v_i,v_j,v_i',v_j',u'\}\quad&\text{if }\{v_i,v_j\}\not\in K.\\
\end{cases}
\]

For the case $i_k=v_i'$, take $\sigma\in K$ to be a facet such that $v_i\in\sigma$, $v_j\notin\sigma$.
Choose a vertex $v_s\in\sigma\setminus\{v_i\}$, and let
\[
I=
\begin{cases}
\{v_i,v_j,F_\sigma,v_j',v_s'\}\quad&\text{if }\{v_i,v_j\}\in K\text{ and }\{v_j',v_s'\}\in K',\\
\{u,v_i,v_j,F_\sigma,v_j',v_s'\}\quad&\text{if }\{v_i,v_j\}\not\in K\text{ and }\{v_j',v_s'\}\in K',\\
\{v_i,v_j,F_\sigma,v_j',v_s',u'\}\quad&\text{if }\{v_i,v_j\}\in K\text{ and }\{v_j',v_s'\}\not\in K',\\
\{u,v_i,v_j,F_\sigma,v_j',v_s',u'\}\quad&\text{if }\{v_i,v_j\}\not\in K\text{ and }\{v_j',v_s'\}\not\in K'.
\end{cases}
\]

\item $i_k\in \VV(T)$. If $\{v_i,v_j\}\not\in K$,
then either $\{v_i,i_k\}, \{v_i',i_k\}\in MF(\EE_P)$ or $\{v_j,i_k\},\{v_j',i_k\}\in MF(\EE_P)$.
Thus we can take
\[I=\{u,u',v_i,v_j,v_i',v_j'\}.\]

On the other hand, if $\{v_i,v_j\}\in K$, let $\tau\in K$ be the facet dual to $i_k$. Then if $v_i\not\in\tau$ or $v_j\not\in\tau$, we can just let $I=\{v_i,v_j,v_i',v_j'\}$.
Otherwise, take $v_s$ to be a vertex such that $\{v_i,v_s\}\in MF(K)$ and take
$\sigma\in K$ to be a facet such that $v_i,v_j\in\sigma$ and $F_\sigma\neq i_k$. Then we can take
 \[
I=
\begin{cases}
\{u,v_i,v_s,F_\sigma,v_j',v_s'\}\quad&\text{if }\{v_j',v_s'\}\in K',\\
\{u,v_i,v_s,F_\sigma,v_j',v_s',u'\}\quad&\text{if }\{v_j',v_s'\}\not\in K'.\\
\end{cases}
\]
\end{enumerate}

\item $i_1,i_2\in \VV(T)$.
Let $\sigma_1,\sigma_2\in K$ and $\sigma_1',\sigma_2'\in K'$ be the facets  dual to $i_1$ and $i_2$, respectively.
Assuming every edge of a connected simplicial complex $K$ has length $1$, define the \emph{distance}  $d_K(v_i,v_j)$ between two different vertices $v_i,v_j\in \VV(K)$ to be the length of a shortest path connecting $v_i$ and $v_j$ in $K$, and $d_K(v_i,v_i)=0$ for all $v_i\in\VV(K)$.
Furthermore, define the distance between two subcomplexes $\emptyset\neq K_1,K_2\subset K$ to be
\[d_K(K_1,K_2):=\text{min}\{d_K(v,u)\mid\{v\}\in K_1,\{u\}\in K_2\}.\]

\begin{enumerate}[(\theenumi a)]\setlength{\itemsep}{1ex}
\item\label{case:5a} $i_k=u'$ (or $u$).
If $d_{K'}(\sigma_1',\sigma_2')=0$, take a vertex $v_s'\in\sigma_1'\cap\sigma_2'$. Since $\sigma_1\neq\sigma_2$, we can choose two vertices $v_i\in\sigma_1$, $v_j\in\sigma_2$ such that $v_i\not\in\sigma_2$, $v_j\not\in\sigma_1$.
Hence $\{i_1,v_j\},\{i_2,v_i\}\in MF(\EE_P)$, and then we can take
\[
I=
\begin{cases}
\{v_i,v_j,i_1,i_2,v_s'\}\quad&\text{if }\{v_i,v_j\}\in K,\\
\{u,v_i,v_j,i_1,i_2,v_s'\}\quad&\text{if }\{v_i,v_j\}\not\in K.\\
\end{cases}
\]

On the other hand, if $d_{K'}(\sigma_1',\sigma_2')=l>0$, suppose $K'_U$ is a path of length $l$ with $U\cap\sigma_1'=v_s',\,U\cap\sigma_2'=v_t'$.
Choose vertices $v_i\in\sigma_1$, $v_j\in\sigma_2$ such that $v_i\neq v_s$ and $v_j\neq v_t$.
Then we can take
\[
I=
\begin{cases}
\{v_i,v_j,i_1,i_2\}\cup U\quad&\text{if }\{v_i,v_j\}\in K,\\
\{u,v_i,v_j,i_1,i_2\}\cup U\quad&\text{if }\{v_i,v_j\}\not\in K.\\
\end{cases}
\]

\item $i_k\in \VV(K')$ (or $\VV(K)$). Assume $i_k=v_k'$. This situation is similar to \eqref{case:5a}.
Set $V=\VV(K')\setminus\{i_k\}$, $\tau_1'=\sigma_1'\cap V$, $\tau_2'=\sigma_2'\cap V$.
Then $|K'_V|\cong D^{n-1}$. If $d_{K'_V}(\tau_1',\tau_2')=0$, take a vertex $v_s'\in\tau_1'\cap\tau_2'$.
Since $\{i_1,i_2\}\in MF(T)$, this means $|\sigma_1\cap\sigma_2|\leqslant |\sigma_1|-2$.
Hence there exist vertices $v_i\in\sigma_1\setminus\{v_k\}$, $v_j\in\sigma_2\setminus\{v_k\}$ such that $v_i\not\in\sigma_2$, $v_j\not\in\sigma_1$, and then we can take
\[
I=
\begin{cases}
\{v_i,v_j,i_1,i_2,v_s'\}\quad&\text{if }\{v_i,v_j\}\in K,\\
\{u,v_i,v_j,i_1,i_2,v_s'\}\quad&\text{if }\{v_i,v_j\}\not\in K.\\
\end{cases}
\]
On the other hand, if $d_{K'_V}(\tau_1',\tau_2')=l>0$, we have $|\sigma_1\cap\sigma_2|\leqslant1$. Suppose $K'_U\subset K_V'$ is a path of length $l$ with $U\cap\tau_1'=v_s'$ and $U\cap\tau_2'=v_t'$.
Since $|\sigma_1|=|\sigma_2|\geqslant3$ and $|\sigma_1\cap\sigma_2|\leqslant1$, we can take a vertex $v_i\in\sigma_1\setminus\{v_k,v_s\}$ and a vertex $v_j\in\sigma_2\setminus\{v_k,v_t\}$ such that $v_i\not\in\sigma_2$, $v_j\not\in\sigma_1$. Then let
\[
I=
\begin{cases}
\{v_i,v_j,i_1,i_2\}\cup U\quad&\text{if }\{v_i,v_j\}\in K,\\
\{u,v_i,v_j,i_1,i_2\}\cup U\quad&\text{if }\{v_i,v_j\}\not\in K.\\
\end{cases}
\]

\item $i_k\in \VV(T)$. Let $\sigma'\in K'$ be the facet dual to $i_k$, and let $V=\VV(K')\setminus\sigma'$, $\tau_1'=\sigma_1'\cap V$, $\tau_2'=\sigma_2'\cap V$.
Clearly $K'_V$ is path-connected.

If $d_{K'_V}(\tau_1',\tau_2')=0$,
take a vertex $v_s'\in\tau_1'\cap\tau_2'$, and vertices $v_i\in\sigma_1$, $v_j\in\sigma_2$ such that $v_i\not\in\sigma_2$, $v_j\not\in\sigma_1$.
It is obvious that $\{i_k,v_s'\}\in MF(\EE_P)$, so we can take
\[
I=
\begin{cases}
\{v_i,v_j,i_1,i_2,v_s'\}\quad&\text{if }\{v_i,v_j\}\in K,\\
\{u,v_i,v_j,i_1,i_2,v_s'\}\quad&\text{if }\{v_i,v_j\}\not\in K.\\
\end{cases}
\]

On the other hand, if $d_{K'_V}(\tau_1',\tau_2')=l>0$, suppose $K'_U\subset K_V'$ is a path of length $l$ and $U\cap\tau_1'=v_s'$, $U\cap\tau_2'=v_t'$.
Since $l>0$, $\tau_1'\cap\tau_2'=\emptyset$, so $v_s'\notin\sigma_2'$, $v_t'\not\in\sigma_1'$.
Choose a vertex $v_i\in\sigma_1\setminus\{v_s\}$ and a vertex $v_j\in\sigma_2\setminus\{v_t\}$ (requiring $v_i=v_j\in\sigma_1\cap\sigma_2$ whenever $\sigma_1\cap\sigma_2\neq\emptyset$), and let
\[
I=
\begin{cases}
\{v_i,i_1,i_2\}\cup U\quad&\text{if }v_i=v_j,\\
\{v_i,v_j,i_1,i_2\}\cup U\quad&\text{if }\{v_i,v_j\}\in K,\\
\{u,v_i,v_j,i_1,i_2\}\cup U\quad&\text{if }\{v_i,v_j\}\not\in K.\\
\end{cases}
\]
Since one of the two components of $(\EE_P)_{I\setminus\{i_1,i_2\}}$ is $K'_U$ and $\{i_k,v_p'\}\in MF(\EE_P)$ for any $v_p'\in U$, $I$ is the desired subset.
\end{enumerate}

\item $i_1\in\VV(K)$, $i_2\in \VV(T)$. Assume $i_1=v_i$, $i_2=F_\sigma=F_{\sigma'}$ for some facet $\sigma\in K$. Clearly, $v_i\not\in\sigma$ since $\{v_i,F_\sigma\}\not\in \EE_P$.

\begin{enumerate}[(\theenumi a)]\setlength{\itemsep}{1ex}
\item\label{case:via} $i_k=u$. Suppose $K_U$ is a path of length $d_K(v_i,\sigma)$ connecting $v_i$ and $\sigma$, $v_j=U\cap\sigma$. Then choose a vertex $v_s'\in\sigma'\setminus\{v_j'\}$, and let
\[
I=
\begin{cases}
U\cup \{F_\sigma,v_i',v_s'\}\quad&\text{if }\{v_i',v_s'\}\in K',\\
U\cup \{F_\sigma,v_i',v_s',u'\}\quad&\text{if }\{v_i',v_s'\}\not\in K'.\\
\end{cases}
\]

\item $i_k=u'$. Choose a vertex $v_j\in\sigma$, and a vertex $v_s'\in\sigma'\setminus\{v_j'\}$. Let $V=(\VV(K')\setminus\sigma')\cup\{v_s'\}$.
It is clear that $K_V'$ is path-connected. Take $K'_W$ to be a shortest path in $K_V'$ connecting $v_i'$ and $v_s'$.
Then let
\[
I=
\begin{cases}
\{v_i,v_j,F_\sigma,W\}\quad&\text{if }\{v_i,v_j\}\in K,\\
\{u,v_i,v_j,F_\sigma,W\}\quad&\text{if }\{v_i,v_j\}\not\in K.\\
\end{cases}
\]

\item\label{case:vic} $i_k\in \VV(K')$. Choose a vertex $v_j'\in\sigma'\setminus\{i_k\}$, and a vertex $v_s'\in\sigma'\setminus\{v_j',i_k\}$ ($\dim K\geqslant 2$). If $i_k\neq v_i'$, then let
\begin{equation}\label{eq:E.1}
I=
\begin{cases}
\{v_i,v_j,F_\sigma,v_i',v_s'\}\quad&\text{if }\{v_i,v_j\}\in K\text{ and }\{v_i',v_s'\}\in K',\\
\{u,v_i,v_j,F_\sigma,v_i',v_s'\}\quad&\text{if }\{v_i,v_j\}\not\in K\text{ and }\{v_i',v_s'\}\in K',\\
\{v_i,v_j,F_\sigma,v_i',v_s',u'\}\quad&\text{if }\{v_i,v_j\}\in K\text{ and }\{v_i',v_s'\}\not\in K',\\
\{u,v_i,v_j,F_\sigma,v_i',v_s',u'\}\quad&\text{if }\{v_i,v_j\}\not\in K\text{ and }\{v_i',v_s'\}\not\in K'.
\end{cases}
\end{equation}

For the case $i_k=v_i'$, there is a facet $\tau\in \mathrm{st}_K\{v_i\}$ satisfying $|\tau\cap\sigma|\leqslant n-2$, since otherwise $2^\sigma\cup\mathrm{st}_K\{v_i\}$ would be a simplicial $(n-1)$-sphere, which would imply that $K=\partial\Delta^n$. Hence we have $\{F_\tau,F_\sigma\}\in MF(\EE_P)$, $\tau\setminus(\sigma\cup \{v_i\})\neq\emptyset$, and we may assume that $v_j,v_s\in\sigma\setminus\tau$. Choose a vertex $v_t'\in \tau'\setminus(\sigma'\cup \{v_i'\})$. Then we can take
\[
I=
\begin{cases}
\{v_i,v_j,F_\tau,F_\sigma,v_t',v_s'\}\quad&\text{if }\{v_i,v_j\}\in K\text{ and }\{v_t',v_s'\}\in K',\\
\{u,v_i,v_j,F_\tau,F_\sigma,v_t',v_s'\}\quad&\text{if }\{v_i,v_j\}\not\in K\text{ and }\{v_t',v_s'\}\in K',\\
\{v_i,v_j,F_\tau,F_\sigma,v_t',v_s',u'\}\quad&\text{if }\{v_i,v_j\}\in K\text{ and }\{v_t',v_s'\}\not\in K',\\
\{u,v_i,v_j,F_\tau,F_\sigma,v_t',v_s',u'\}\quad&\text{if }\{v_i,v_j\}\not\in K\text{ and }\{v_t',v_s'\}\not\in K'.
\end{cases}
\]

\item $i_k\in \VV(T)$. Assume $i_k=F_\tau$ for some facet $\tau\in K$. Choose a vertex $v_j\in\sigma\setminus\tau$, and a vertex $v_s'\in\sigma'\setminus\{v_j'\}$. Then we can take $I$ to be the same as in \eqref{eq:E.1}.

\item $i_k\in \VV(K)$. Choose a vertex $v_j\in\sigma\setminus\{i_k\}$, and a vertex $v_s\in\sigma\setminus\{v_j,i_k\}$. Then we can also take $I$ to be the same as in \eqref{eq:E.1}.
\end{enumerate}

\item $i_1,i_2\in\VV(K)$. Assume $i_1=v_i$, $i_2=v_j$. Clearly, $i_k\neq u$.

\begin{enumerate}[(\theenumi a)]\setlength{\itemsep}{1ex}
\item $i_k\in\VV(K)$ or $\VV(T)$. We can take $I=\{u,v_i,v_j,v_{i'},v_{j'},u'\}$.

	\item $i_k\in \VV(K')$. If $i_k\neq v_i',v_j'$, then we can take $I$ as above. For the other case, say $i_k=v_j'$, choose a facet $\sigma'\in K'$ containing $v_j'$  and a vertex $v_s'\in \sigma'\setminus\{v_j'\}$. Then we can take 
	\[I=
	\begin{cases}
	\{u,v_i,v_j,F_{\sigma'},v_i',v_s'\}\quad &\text{ if } \{v_i',v_s'\}\in K',\\
	\{u,v_i,v_j,F_{\sigma'},v_i',v_s',u'\}\quad &\text{ if }\{v_i',v_s'\}\not\in K'.
	\end{cases}\]
\item $i_k=u'$. In this case, choose a shortest path $K'_U$ in $K'$ connecting $v_i'$ and $v_j'$, and take $I=\{u,v_i,v_j\}\cup U$.
\end{enumerate}

\end{enumerate}

It follows from the symmetry of $\EE(P)$ that these cases exhaust all possibilities, and then we finish the proof.
\end{proof}

\section{A generalization of Proposition \ref{prop:4-circuit}}\label{app:n=3}
The following proposition is a generalization of Proposition \ref{prop:4-circuit}, which is an essential ingredient in the second part of this paper.
This result was also obtained by Erokhovets \cite[Lemma 5.2]{Ero20} (see also \cite{Ero20b}) in the dual form.
\begin{prop}\label{prop:gen}
Let $K$ be a simplicial $2$-sphere from $\QQ$ (see Section \ref{subsec:the class Q} for the definition), $\VV(K)=[m]$. Suppose a missing face $\omega=\{i_1,i_2\}\in MF(K)$ and a vertex $i_s\in [m]\setminus\omega$ satisfy the following condition:
\[\mathrm{lk}_K\{i\}\ \text{is not a $4$-circuit for any}\  i\in\VV(\mathrm{lk}_K\{i_s\})\cap\omega.\]
Then there exists a subset $I\subset[m]\setminus\{i_s\}$ with $\omega\subset I$ such that $|K_I|\cong S^1$ and $\w H_0(K_J)\neq0$ for $J=\{i_s\}\cup I\setminus\omega$.
\end{prop}

\begin{proof}
Since $K$ is flag, we can easily deduce that $\Gamma:=K_{[m]\setminus\{i_s\}}$ is a flag triangulation of $D^2$, and $\partial \Gamma=K_\Ss$, where $\Ss=\VV(\mathrm{lk}_K\{i_s\})$.
So by \cite[Lemma 6.1]{FW15} (see also \cite[Lemma A.2]{BEMPP17}), there exists $I_0\subset[m]\setminus\{i_s\}$ with $\omega\subset I_0$ such that $|K_{I_0}|\cong S^1$.
Note that $K_{I_0}$ bounds a simplicial disk $K_0\subset \Gamma$. Let $J_0=\{i_s\}\cup(I_0\setminus\omega)$.
If $\w H_0(K_{J_0})\neq0$, then $I=I_0$ is the desired subset. Otherwise, assign an orientation to $K_{I_0}$ such that $K_0$ is on our left when we are outside the sphere and go along $K_{I_0}$ in the positive orientation.
For the two components of $K_{I_0\setminus\omega}$, denote $L_0^+$ (resp. $L_0^-$) the one from $i_1$ to $i_2$ (resp. from $i_2$ to $i_1$), the words ``from'' and ``to'' referring to the given orientation of $K_{I_0}$.

Since $\w H_0(K_{J_0})=0$, $\Ss\cap\VV(L_0^+),\,\Ss\cap\VV(L_0^-)\neq\emptyset$.
Taking a vertex $l_0\in \Ss\cap\VV(L_0^+)$, the fact that $\partial\Gamma=K_\Ss$ implies that $\Gamma_0:=\Gamma-\{l_0\}=\{\sigma\in \Gamma:l_0\not\in\sigma\}$ is still a triangulation of $D^2$, but $\partial \Gamma_0$ may not be a full subcomplex of $K$. Let
\[\Aa_0=\{\sigma\in\Gamma_0:\dim\sigma=1,\,\sigma\not\in\partial \Gamma_0,\,\sigma\subset \VV(\partial\Gamma_0)\}.\]
Then we can uniquely write $\Gamma_0$ as
\[\Gamma_0=\bigcup_i \Gamma_{0,i},\quad \Gamma_{0,i}=K_{\VV(\Gamma_{0,i})},\]
where each $\Gamma_{0,i}$ is a triangulation of $D^2$ such that either $\Gamma_{0,i}=\Delta^2$ or $\partial \Gamma_{0,i}$ is a full subcomplex of $K$, and $\Gamma_{0,j}\cap \Gamma_{0,k}$ contains at most one edge from $\Aa_0$ for $j\neq k$ (see Figure \ref{fig:gamma0}).
\input{gamma0.TpX}

We claim that $\omega\in \Gamma_{0,i}$ for some $i$.
Since otherwise there would be an edge $\{j_1,j_2\}\in \Aa_0$ ($j_1,j_2\neq i_1,i_2$) separating $\Gamma_0$ into two components, such that $i_1$ and $i_2$ are in different ones.
By definition, $j_1,j_2\in\VV(\partial\Gamma_0)$, $\{j_1,j_2\}\notin\partial \Gamma_0$, so since $\partial\Gamma_0=\mathrm{lk}_\Gamma\{l_0\}\cup K_{\Ss\setminus\{l_0\}}$ and $\mathrm{lk}_\Gamma\{l_0\}$, $K_{\Ss\setminus\{l_0\}}$ are full subcomplexes of $K$, we may assume $j_1\in\VV(\mathrm{lk}_\Gamma\{l_0\})\setminus\Ss$ and $j_2\in\Ss\setminus\VV(\mathrm{st}_\Gamma\{l_0\})$.
Hence $K$ restricted to the subset $U:=\{i_s,l_0,j_1,j_2\}$ is a $4$-circuit, and then it is the link of a vertex.
This vertex is either $i_1$ or $i_2$, since $i_1$ and $i_2$ are in different components of $\Gamma_0-\{j_1,j_2\}$, as well as components of $K_{[m]\setminus U}$.
However this contradicts the assumption in the proposition.

Suppose $\omega\in \Gamma_{0,a_0}$. Clearly $\Gamma_{0,a_0}\neq\Delta^2$, so $\partial\Gamma_{0,a_0}=K_{\VV(\Gamma_{0,a_0})}$.
Hence by applying \cite[Lemma 6.1]{FW15} again we can find a subset $I_1\subset\VV(\Gamma_{0,a_0})$ with $\omega\subset I_1$ such that $|K_{I_1}|\cong S^1$. Note that $K_{I_1}$ bounds a simplicial disk $K_1\subset \Gamma_{0,a_0}$. Let $J_1=\{i_s\}\cup(I_1\setminus\omega)$. If $\w H_0(K_{J_1})\neq0$, then $I=I_1$ is the desired subset. Otherwise, assign to $K_{I_1}$ the same orientation as $K_{I_0}$, and let $L_1^+$ (resp. $L_1^-$) be the component of $K_{I_1\setminus\omega}$ from $i_1$ to $i_2$ (resp. from $i_2$ to $i_1$).
As before, taking a vertex $l_1\in \Ss\cap\VV(L_1^+)$, we can show that $\Gamma_1:=\Gamma_{0,a_0}-\{l_1\}$  is still a triangulation of $D^2$, and we can uniquely write $\Gamma_1$ as
\[\Gamma_1=\bigcup_i \Gamma_{1,i},\quad \Gamma_{1,i}=K_{\VV(\Gamma_{1,i})},\]
where each $\Gamma_{1,i}$ is a triangulation of $D^2$ such that  either $\Gamma_{1,i}=\Delta^2$ or $\partial \Gamma_{1,i}$ is a full subcomplex of $K$, and $\Gamma_{1,j}\cap \Gamma_{1,k}$ contains at most one edge from $\Aa_1=\{\sigma\in\Gamma_1:\dim\sigma=1,\,\sigma\not\in\partial \Gamma_1,\,\sigma\subset \VV(\partial\Gamma_1)\}$ for $j\neq k$.

If $\omega\in \Gamma_{1,a_1}$ for some $a_1$, then we can continue this process inductively.
Actually, we will show that if we have defined $I_k$, $K_k$, $L_k^+$, $L_k^-$, $l_k$, $\Gamma_k=\Gamma_{k-1,a_{k-1}}-\{l_k\}$, $\{\Gamma_{k,i}\}$, there must exist $\Gamma_{k,a_k}$ such that $\omega\in\Gamma_{k,a_k}$.
Then since in each step we eliminate one vertex $l_k$ from the finite set $\Ss$, this process eventually terminates, say, at the $n$th step, and we can take $I=I_n$ as the desired subset. The proof concludes by showing this fact.

Suppose on the contrary that for some $k>0$, there is an edge $\{j_1,j_2\}\in \Gamma_k$ separating $\Gamma_k$ into two components, such that $i_1$ and $i_2$ are in the different ones. Let $\Ss_0=\Ss$ and $\Ss_k=\VV(\partial\Gamma_{k-1,a_{k-1}})$ for $k>0$.  Then as before, we may assume $j_1\in\VV(\mathrm{lk}_K\{l_k\})\setminus \Ss_k$ and $j_2\in\Ss_k\setminus\VV(\mathrm{st}_K\{l_k\})$.
Clearly, we have $j_1\not\in\Ss$. Moreover, consider the full subcomplex on $\{i_s,l_k,j_1,j_2\}$, then the same argument as in the third paragraph of the proof shows that $j_2\not\in\Ss$.
We claim that $j_2\in\VV(\mathrm{lk}_K\{l_j\})$ for some $j<k$. To see this, note that
\[\Ss_k\subset\VV(\partial\Gamma_{k-1})=\VV(\mathrm{lk}_{\Gamma_{k-2,a_{k-2}}}\{l_{k-1}\})\cup\Ss_{k-1}\setminus\{l_{k-1}\},\]
so $j_2\in \VV(\mathrm{lk}_K\{l_{k-1}\})$ or $j_2\in\Ss_{k-1}$. The first case is what we want, and for the second case we can continue this process by an inductive argument, and eventually get the desired $l_j$ for some $j<k$, noting $j_2\not\in\Ss_0=\Ss$.

\input{gamma.TpX}
Since $l_j,l_k\in\Ss=\VV(\partial\Gamma)$, $j_1,j_2\not\in\Ss$, it is easy to see that the path consisting of three edges $\{l_k,j_1\}$, $\{j_1,j_2\}$, $\{j_2,l_j\}$ separates $\Gamma$ into two components such that $i_1$ and $i_2$ are in different ones.
Without loss of generality, we may assume $i_1$ is in the upper components and $i_2$ is in the lower one, and we first consider the case $l_j$ is on the right side, just as shown in Figure \ref{fig:gamma}.
From $|K_{I_j}|\cong S^1$ it follows that
\[\VV(L_j^+)\cap\{l_k,j_1,j_2,l_j\},\ \VV(L_j^-)\cap\{l_k,j_1,j_2,l_j\}\neq\emptyset.\]
This, together with the fact that $l_j\in \VV(L_j^+)$ implies that $\VV(L_j^+)\cap\{l_k,j_1,j_2,l_j\}=\{l_j\}$ or $\{j_2,l_j\}$. However in either case, from a geometric observation we can easily see that when we go from $i_1$ to $i_2$ along $L_j^+$,
$K_j$ would be on the right side of $L_j^+$, contradicting the given orientation.
For the case that $l_k$ is on the right side, we can apply the same argument to $L_k^+$.
\end{proof}

\end{document}